\documentclass[11pt]{amsart}
\usepackage{amsmath,amsthm,amsfonts}
\usepackage{fullpage}
\usepackage{paralist}
\usepackage{graphics}
\usepackage{epsfig}
\usepackage{cite}

\usepackage{enumerate}
\usepackage{comment}
\usepackage{xcolor}
\usepackage{subfigure}
\usepackage{dsfont}
\usepackage{bbm}
\usepackage{bm}
\usepackage{soul}
\usepackage{enumitem}
\usepackage{mathtools}
\usepackage{cancel}
\mathtoolsset{showonlyrefs}

\usepackage{amssymb,amsxtra,setspace,xspace,lmodern,psfrag,color,latexsym,pifont,mathrsfs,caption,microtype,accents}

\usepackage[colorlinks=true]{hyperref}
\hypersetup{urlcolor=blue, citecolor=red, linkcolor=black}

\allowdisplaybreaks[4]

\newenvironment{listi}
  {\begin{list}
 {(\roman{broj})}
{ \usecounter{broj}}
  \setlength{\labelwidth}{30pt}
  \addtolength{\itemsep}{2pt}}
{   \end{list} }
\newcounter{broj}

\DeclareMathAlphabet{\mathup}{OT1}{\familydefault}{m}{n}
\newcommand{\di}[1]{\mathop{}\!\mathup{d} #1}
\newcommand{\dx}[1]{\mathop{}\!\mathup{d} #1}
\newcommand{\pderiv}[3][]{\frac{\mathop{}\!\mathup{d}^{#1} #2}{\mathop{}\!\mathup{d} #3^{#1}}}

\newcommand{\N}{\mathbb{N}}
\newcommand{\R}{\mathbb{R}}
\newcommand{\Rd}{{\mathbb{R}^{d}}}

\newcommand{\Rdd}{{\mathbb{R}^{2d}}}
\newcommand{\cA}{\mathcal{A}}

\newcommand{\cD}{\mathcal{D}}
\newcommand{\cE}{\mathcal{E}}
\newcommand{\cF}{\mathcal{F}}
\newcommand{\cG}{\mathcal{G}}
\newcommand{\cH}{\mathcal{H}}

\newcommand{\cM}{\mathcal{M}}
\newcommand{\mP}{\mathcal{P}}
\newcommand{\cP}{\mathcal{P}}
\newcommand{\cV}{\mathcal{V}}

\newcommand{\cS}{\mathcal{S}}
\newcommand{\mt}{\mathcal{T}}
\newcommand{\cT}{\mathcal{T}}

\newcommand{\bj}{\bm{j}}
\newcommand{\dgrad}{\overline\nabla}
\newcommand{\eps}{\varepsilon}
\newcommand{\brho}{\bm{\rho}}

\newtheorem{theorem}{Theorem}[section]
\newtheorem{corollary}[theorem]{Corollary}
\newtheorem{lemma}[theorem]{Lemma}
\newtheorem{proposition}[theorem]{Proposition}

\newtheorem{definition}[theorem]{Definition}

\theoremstyle{remark}
\newtheorem{rem}[theorem]{Remark}
\newtheorem{exm}[theorem]{Example}

\numberwithin{equation}{section}

\DeclareMathOperator{\AC}{AC}
\DeclareMathOperator{\grad}{grad}
\DeclareMathOperator{\Diff}{Diff}
\DeclareMathOperator{\CE}{CE}
\DeclareMathOperator{\loc}{loc}
\DeclareMathOperator{\supp}{supp}
\DeclareMathOperator*{\argmin}{arg\,min}
\DeclareMathOperator{\Div}{div}

\DeclareMathOperator*{\esssupmu}{\mu-ess\,sup}

\DeclarePairedDelimiter{\abs}{\lvert}{\rvert}
\DeclarePairedDelimiter{\norm}{\lVert}{\rVert}
\DeclarePairedDelimiter{\bra}{(}{)}

\DeclarePairedDelimiter{\set}{\{}{\}}

\title{nonlocal-interaction equation on graphs: \\ gradient flow structure and continuum limit}

\begin{document}
\author{Antonio Esposito \and Francesco S. Patacchini \and Andr\'e Schlichting \and  Dejan Slep\v{c}ev}
\address{Antonio Esposito -- 
Department Mathematik, Friedrich-Alexander-Universit\"at Erlangen-N\"urnberg, Cauerstrasse 11,
91058 Erlangen, Germany}
\address{Francesco S. Patacchini -- Department of Mathematical Sciences, Carnegie Mellon University, Pittsburgh, PA 15213, USA; IFP Energies nouvelles, 1 et 4 avenue de Bois-Pr\'eau, 92852 Rueil-Malmaison, France}
\address{Andr\'e Schlichting -- Institute of Applied Mathematics, University of Bonn, Endenicher Allee 60, D-53115 Bonn, Germany}
\address{Dejan Slep\v{c}ev -- Department of Mathematical Sciences, Carnegie Mellon University, Pittsburgh, PA 15213, USA}
\email{antonio.esposito@fau.de}
\email{fpatacch@math.cmu.edu; francesco.patacchini@ifpen.fr}
\email{schlichting@iam.uni-bonn.de}
\email{slepcev@math.cmu.edu}

\begin{abstract}
We consider dynamics driven by interaction energies on graphs. We introduce graph analogues of the continuum nonlocal-interaction equation and interpret them as gradient flows with respect to a graph Wasserstein distance. The particular Wasserstein distance we consider arises from the graph analogue of the Benamou--Brenier formulation where the graph continuity equation uses an upwind interpolation to define the density along the edges. While this approach has both theoretical and computational advantages, the resulting distance is only a quasi-metric. We investigate this quasi-metric both on graphs and on more general structures where the set of ``vertices'' is an arbitrary positive measure. 
We call the resulting gradient flow of the nonlocal-interaction energy the nonlocal nonlocal-interaction equation (NL$^2$IE).
We develop the existence theory for the solutions of the NL$^2$IE as curves of maximal slope with respect to the upwind Wasserstein quasi-metric. Furthermore, we show that the solutions of the NL$^2$IE on graphs converge as the empirical measures of the set of vertices converge weakly, which establishes a valuable discrete-to-continuum convergence result.
\end{abstract}


\date{\today}

\maketitle

\tableofcontents

\section*{Notation}
We list here some symbols used throughout the paper. 
\begin{itemize}
  \item $\cM(A)$ is the set of Borel measures on $A \subseteq \R^d$. 
  \item $\cM^+(A)$  is the set of non-negative Borel measures on $A$.
  \item $\mathcal{P}(A)\subset \cM^+(A)$ is the set of Borel probability measures on $A$.
  \item $\cP_{2}(A)\subseteq \cP(A)$ stands for the elements of $\mP(A)$ with finite second moment, that is,
    \begin{equation*}\label{eq:def:M2}
      M_2(\rho) := \int_{A} |x|^2 \di\rho(x) < \infty. 
    \end{equation*}
  \item $C_\mathrm{b}(A)$ is  the set of bounded continuous functions from $A$ to $\R$. 
  \item $a_+:=\max\{0,a\}$ and $a_-:=(-a)_+$ are the positive and negative parts of $a \in \R$.
  \item $\mu \in \cM^+(\R^d)$ sets the underlying geometry of the state space; it is sometimes referred to as base measure.
  \item $\rho\in \cP(\R^d)$ denotes a configuration; the natural setting is that $\supp \rho \subseteq \supp \mu$, although we allow for 
    general supports as needed for stability results.
  \item $\eta \colon \{ (x,y)\in \R^d \times \R^d : x\neq y \}\to [0,\infty)$ is the edge weight function. 
  \item $G= \{ (x,y) \in \R^d \times \R^d : x\neq y ,\, \eta(x,y)>0\}$ are the edges.
  \item $\rho_1\otimes \rho_2 \in \cM^+(G)$ is the product measure of $\rho_1, \rho_2 \in \cM^+(\R^d)$ restricted to $G$. 
  \item $\gamma_1 = \rho \otimes \mu$ and $\gamma_2 = \mu \otimes \rho$.
  \item $\cV^{\mathrm{as}}(G)$
    is the set of antisymmetric graph vector fields on $G$, defined in \eqref{eq:nonloc-vect-field-set}.
  \item $\dgrad f$ is the nonlocal gradient of a function $f \colon \R^d \to \R$, while $\dgrad \cdot \bj$ is the nonlocal divergence of a measure-valued flux $\bj\in \cM(G)$; see Definition \ref{def:nl_grad_div}. 
  \item $\cA$ stands for the action functional; see Definition~\ref{def:action}.
  \item $\mt$ denotes the nonlocal transportation quasi-metric; see \eqref{eq:nonloc-upwind-transp-cost}.
  \item $\CE_T(\rho_0,\rho_1)$ denotes the set of paths (solutions to the nonlocal continuity equation for densities~\eqref{eq:nce} or measures~\eqref{eq:nlce_measures}) on the time interval $[0,T]$ connecting two measures $\rho_0, \rho_1\in \cP(\R^d)$; we set $\CE:=\CE_1$.
\end{itemize}
Let us also specify the notions of \emph{narrow convergence} and \emph{convolution}. 
A sequence $(\rho^n)_n\subset \cM(A)$ is said to converge narrowly to $\rho\in \cM(A)$, in which case we write $\rho^n \rightharpoonup \rho$, provided that
\[
 \forall f\in C_\mathrm{b}(A), \qquad \qquad \int_A f \di{\rho^n} \to \int_A f \di{\rho} \qquad \textrm{as } n \to \infty. 
\]
Given a function $f \colon A \times A \to \R$ and $\rho \in \cM(A)$, we write $f*\rho$ the convolution of $f$ and $\rho$, that is,
\begin{equation*}
    f*\rho(x)=\int_A f(x,y)\di\rho(y) \quad \text{for any $x \in A$ such that the right-hand side exists}.
\end{equation*}
%
%
%
\section{Introduction}
We investigate dynamics driven by interaction energies on graphs, and their continuum limits. 
We interpret the relevant dynamics as gradient flows of the interaction energy with respect to a particular graph analogue of the Wasserstein distance. We prove the convergence of the dynamics on finite graphs to a continuum dynamics as the number of vertices goes to infinity. To do so we create a unified setup where the continuum and the discrete dynamics are both seen as particular instances of the gradient flow of the same energy, with respect to a nonlocal Wasserstein quasi-metric whose state space is adapted to the configuration space considered. 

Let us first introduce the problem on finite graphs where it is the simplest to describe. 
\subsection{Graph setting with general interactions}  \label{sec:gen-graph}
Consider an undirected graph with vertices $X =\{x_1, \dots, x_n\}$ and edge weights $w_{x,y} \geq 0$, satisfying $w_{x,y} = w_{y,x}$ for all $x,y \in X$. Although technically not necessary, we impose the natural requirement that $w_{x,x}=0$. The interaction potential is a symmetric function $K \colon X \times X \to \R$, while the external potential is denoted $P\colon X \to \R$. We consider a ``mass'' distribution $\rho\colon X \to [0, \infty)$, and we require $\sum_{x \in X} \rho_x =1$. The total energy~$\cE_X\colon \cP(X)\to \R$ is a combination of the interaction energy $\cE_I$ and the potential energy $\cE_P$:
\begin{equation} \label{eq:discE}
    \cE_X(\rho) = \cE_{I}(\rho) + \cE_P(\rho) =  \frac12 \sum_{x \in X} \sum_{y \in X} K_{x,y} \rho_x \rho_y + \sum_{x \in X} P_x \rho_x. 
\end{equation}
The gradient descent of $\cE_X$ that we study is described by the following system of ODE for the mass distribution:
\begin{align}
     \frac{\di \rho_x}{\di t} &= - \frac{1}{2} \sum_{y\in X} \bra[\big]{j_{x,y} - j_{y,x}} w_{x,y}, 
     \label{eq:intro:CE} \\
      j_{x,y} &= \frac{1}{n} \bra[\big]{\rho_x (v_{x,y})_+ - \rho_y (v_{x,y})_-},  \label{eq:intro:flux-velocity}\\
        v_{x,y} & = -\sum_{z\in X}  \rho_z (K_{y,z}-K_{x,z}) - (P_y - P_x). \label{eq:intro:vE}
\end{align}
The quantities $v\colon X \times X \to \R$ and $j:X \times X \to \R$ are defined on edges and model the graph analogues of velocity and flux. 
An evolution by such system is illustrated on Figure \ref{fig:two-moon}.
\begin{figure}[ht]
  \includegraphics[width=\textwidth]{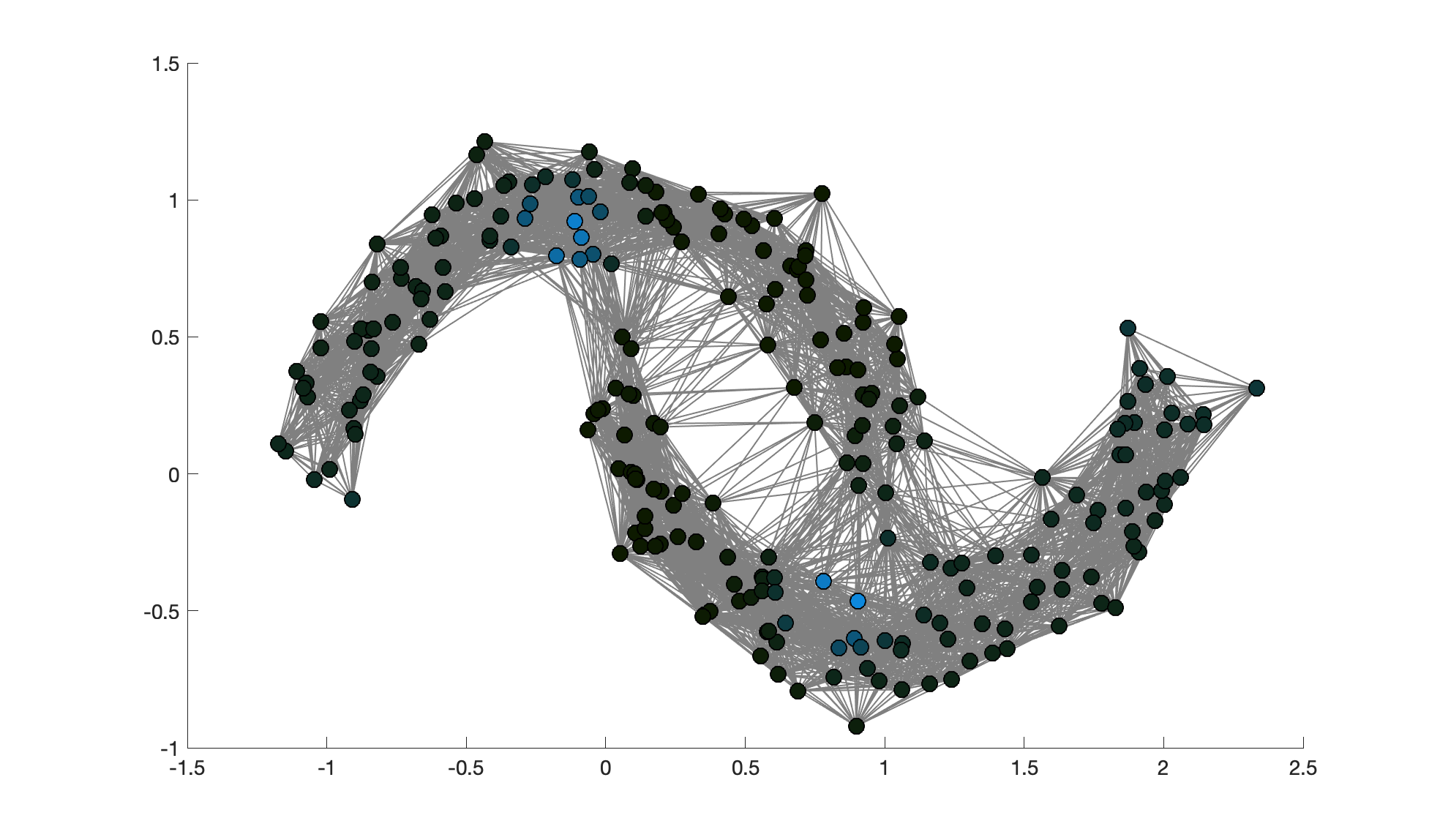}
  \caption{A solution of the nonlocal-interaction equation on graphs driven by the energy \eqref{eq:discE}. We consider a graph based on 240 sample points $X$ from a 2D two-moon data set. 
  The connectivity distance is $\varepsilon=0.7$.
  The edge weights are $w_{x,y} = \exp(-6|x-y|^2) $ if $|x-y| \leq \eps$ and zero otherwise. 
  The interaction potential is $K_{x,y} = 1-\exp(-d(x,y)^2/10)$, where $d(x,y)$ is the graph distance between vertices $x$ and $y$ of~$X$ with edge weights $1/w_{x,y}$, and the external potential is $P\equiv 0$. 
  The solution, starting from a uniform distribution, is shown at time $t=60$. Brighter color indicates more mass. }\label{fig:two-moon}
\end{figure} 
The system~\eqref{eq:intro:CE}--\eqref{eq:intro:vE} is the gradient flow of the energy $\cE_X$ with respect to a new graph equivalent of the Wasserstein metric. The concept of Wasserstein metrics on finite graphs were introduced independently by Chow, Huang, Li, and Zhou \cite{CHLZ12}, Maas \cite{Maas11}, and Mielke \cite{Mielke2011gradient,Miel13}. All of the approaches rely on graph analogues of the continuity equation to describe the paths in the configuration space. On graphs the mass is distributed over the vertices and is exchanged over the edges. Hence, the analogues of the vector field and the flux are defined over the edges. However, the flux should be the product of the velocity (an edge-based quantity) by the density (a vertex-based quantity). Thus, one has to interpolate the densities at vertices to define the density (and hence the flux) along the edges. The choice of interpolation is not unique, and has important ramifications. 

While the overall structure of our setup is derived from one in~\cite{Maas11}, which we recall in Section~\ref{subsec:flux:SG}, the form of the interpolation used is related to the upwind interpolation used in~\cite{CHLZ12} and is almost identical to one in \cite{ChTrTa18}. While in~\cite{CHLZ12} the authors considered only the direction of the flux due to the potential energy to determine which density to use on the edges, in our case the density chosen depends on the total velocity and we furthermore include the interaction term which itself depends on the configuration. In particular, we use an upwind interpolation based on the total velocity. In the context of graph Wasserstein distance, such interpolation was first used by Chen, Georgiou, and Tannenbaum
in \cite{ChTrTa18}.

The ``velocities'' $v$ we consider 
can be assumed to be antisymmetric: $v_{x,y} = - v_{y,x}$ for all $x,y \in X$. 
In the graph setting, which we normalize in order to consider limit $n \to \infty$, the continuity equation with upwind interpolation is obtained by combining~\eqref{eq:intro:CE} with the flux-velocity relation~\eqref{eq:intro:flux-velocity}.
Similarly to \cite{Maas11} and exactly as in \cite{ChTrTa18}, we define the graph Wasserstein distance by minimizing the action, which is the graph analogue of the kinetic energy:
\[ A(\rho,v) =  \frac{1}{n} \sum_{x \in X} \sum_{y \in X} (v_{x,y+})^2 w_{x,y} \rho_x. \]
As in \cite{ChTrTa18,CHLZ12,Maas11,Miel13}, the graph Wasserstein distance is defined by adapting the Benamou--Brenier formula:
\[ \mt(\rho^0,\rho^1)^2 = \inf_{(\rho,v) \in  \CE_X(\rho^0,\rho^1)} \int_0^1 A(\rho(t), v(t)) \di{t},  \] 
where $\CE_X(\rho^0,\rho^1)$ is the set of all paths (i.e., solutions of \eqref{eq:intro:CE}--\eqref{eq:intro:vE}) connecting $\rho^0$ and $\rho^1$. 

It is important to observe that, in our setting, $\mt$ is not symmetric (that is, $\mt(\rho^0,\rho^1)$ is in general different from $\mt(\rho^1,\rho^0)$). The reason is that in general $A(\rho,v) \neq A(\rho, -v)$. Therefore the nonlocal Wasserstein distance which arises from the upwind interpolation is only a quasi-metric. The action~$A(\rho,v)$ provides a Finsler structure to the tangent space, instead of the usual  Riemannian structure. Formally the system \eqref{eq:intro:CE}--\eqref{eq:intro:vE} is the gradient flow of 
$\cE_X$ with respect to this Finsler structure; we present a derivation of this fact in a more general setting in Section \ref{subsec:HeuristicFinsler}. The system is also  the curve of steepest descent with respect to quasi-metric $\mt$, which is the point of view we use to create rigorous theory in the general setting.  
\begin{rem}\label{rem:existence-strong-discrete}
The well-posedness of \eqref{eq:intro:CE}--\eqref{eq:intro:vE}  is a straightforward consequence of the Picard existence theorem. Namely, note that the simplex $1 \geq \rho_x \geq 0$, $\sum_{x \in X} \rho_x =1$ is an invariant region of the dynamics and that on it the vector field \eqref{eq:intro:vE} is Lipschitz continuous in $\rho_x$, $x \in X$. 
\end{rem}
\begin{rem} \label{rem:other-interp}
One could consider other interpolations instead of the upwind one. In particular, if we considered an interpolation of the form $I(\rho_x, \rho_y)$ instead of the upwind one, the only change in the gradient flow would be that the velocity-flux relation \eqref{eq:intro:flux-velocity} would become
$      j_{x,y} =  \frac{1}{n} I(\rho_x, \rho_y) v_{x,y} $. We note that this can have major implications on the resulting dynamics. In particular, for the logarithmic interpolation, $I(r,s) = (r-s)/(\ln r - \ln s)$, or the geometric interpolation, $I(r,s) = \sqrt{rs}$, the resulting dynamics would never expand the support of the solutions, so even for repulsive potentials the mass may not spread throughout the domain. On the other hand, using the arithmetic interpolation, $I(r,s) =(r+s)/2$, would not work directly since the solutions may become negative. In this case additional technical steps, like a Lagrange multiplicator as in~\cite{NataleTodeschi20}, are necessary to obtain the evolution of a non-negative probability density. We use the more physical inspired upwind flux, which automatically ensures the positivity of the density.
\end{rem}
Before we turn to the general setting we point out that the system~\eqref{eq:intro:CE}--\eqref{eq:intro:vE} offers a new model of graph-based clustering, which is briefly discussed in Section \ref{sec:ML}.
\subsection{General setting for vertices in Euclidean space}
Here we introduce the general framework for studies of interaction equations on families of graphs and their limits as the number of vertices~$n$ goes to $\infty$. In particular, in the applications to machine learning which we briefly discuss in Section~\ref{sec:ML}, the graphs considered are random samples of some underlying measure in Euclidean space, and the edge weights, as well as the interaction energy, depend on the positions of the vertices. 
The vertices are points in $\R^d$.
The edges are given in terms of a non-negative symmetric weight function $\eta \colon  \{ (x,y) \in \R^d \times \R^d : x \neq y \} \to [0, \infty)$, which defines the set of edges as $G=\set{ (x,y)\in \R^d\times \R^d : x\neq y , \,\eta(x,y)>0}$. From the discrete setting, the set of vertices is replaced by the more general notion of a measure on $\R^d$; the discrete graphs with vertices $X = \{x_1, \dots, x_n\} \subset \R^d$ correspond to $\mu$ being the empirical measure of the set of points,  $\mu = \frac{1}{n} \sum_{i=1}^n \delta_{x_i}$. The distribution of mass over the vertices is described by the measure $\rho \in \cP(\R^d)$ and in most applications we consider $\supp \rho \subseteq \supp \mu$. However, in order to prove general stability results (e.g., Theorem \ref{thm:graph:stability}), we need to allow that initially part of the support of $\rho$ is outside of the support of $\mu$; we think of such mass as outside of the domain specified by $\mu$. The mass starting outside of the support of $\mu$ can only flow into the support of $\mu$. Here we present the evolution assuming  $\rho \ll \mu$, while in Sections \ref{sec:ce-ut} and \ref{sec:NLIE} we present the setup in full generality. Furthermore, we denote by $\rho$ both the measure and its density with respect to $\mu$.

The evolution of interest is the gradient descent of the energy $\cE:\cP(\R^d)\to \R$ given by
\begin{equation}\label{e:energy:intro}
    \cE(\rho)= \frac{1}{2}\int_\Rd\int_\Rd K(x,y)\di\rho(x)\di\rho(y) + \int_\Rd P(x) \di\rho(x),
\end{equation}
where  $K\colon \Rd \times \Rd \to \R$ is symmetric and $P:\Rd \to \R$. This energy generalizes \eqref{eq:discE} in terms of the configurations $\rho$ and specializes it in terms of the type of potentials $K$ and $P$ considered. In fact, from now on we omit the subscripts $X$ referring to the vertices (e.g.\ in the energy) since our general setting allows for distribution of mass outside of the support of $\mu$.
The gradient flow we consider takes the form
\begin{align}\label{eq:nlnl-intro}
\begin{split}
    \partial_t\rho_t(x) & = - \int_\Rd j_t(x,y) \eta(x,y) \di\mu(y) =: - (\dgrad \cdot j_t) (x),  \\
    j_t(x,y) &= \rho_t(x) v_t(x,y)_+ -  \rho_t(y) v_t(x,y)_-, \\
    v_t(x,y) &= - \left(K*\rho_t(y) - K*\rho_t(x) + P(y) - P(x)\right).
\end{split} \tag{NL$^2$IE}
\end{align}
The system \eqref{eq:nlnl-intro} consists first of a nonlocal continuity equation, where the divergence $\dgrad \cdot$ is encoded with the graph structure described through $\mu$ and $\eta$ (see Definition~\ref{def:nl_grad_div}). Secondly, it involves a mapping from velocity to flux, which in our case is the upwind flux and encodes the geometry of the gradient structure. Finally, the third equation identifies the driving velocity as the nonlocal gradient of the variation of the energy~\eqref{e:energy:intro}.
Overall, we obtain that~\eqref{eq:nlnl-intro} is the gradient flow of the energy $\cE$ with respect to a generalization of the graph Wasserstein metric we now introduce. 
\subsubsection*{Nonlocal continuity equation}
Let us set
\begin{equation}\label{eq:nonloc-vect-field-set}
    \cV^{\mathrm{as}}(G) := \{v\colon G\to\R : v(x,y) = -v(y,x)\ \text{for all}\ (x,y)\in G\} 
\end{equation}
and call its elements \emph{nonlocal (antisymmetric) vector fields} on $G$; for any pair $(x,y) \in G$ the value~$v(x,y)$ can be regarded as a jump rate from $x$ to $y$. Let us fix a final time $T>0$ throughout the paper and let a family $\{v_t\}_{t\in[0,T]}\subset\cV^{\mathrm{as}}(G)$ be given. In the case $\rho_t \ll \mu$ for all $t\in [0,T]$, it is possible to combine the first two equations in~\eqref{eq:nlnl-intro} in order to arrive at the nonlocal continuity equation
\begin{equation}\label{eq:nce}
    \partial_t\rho_t(x)+\int_\Rd\left(\rho_t(x)v_t(x,y)_+-\rho_t(y)v_t(x,y)_-\right)\eta(x,y)\di \mu(y) = 0,\quad  \mu\text{-a.e.\ } x \in \Rd.
\end{equation}
For general curves $\rho\colon [0,T] \to \cP(\Rd)$, it is necessary to consider the weak form of~\eqref{eq:nce}, which is discussed in Section~\ref{subsec:nonloc-cont-eq}.

We remark that the general setup we develop allows for the solution $\rho$ to develop atoms and persist even after the atoms have formed. Heuristic arguments and numerical experiments indicate that there are equations covered by our theory for which this is the case. For example, if $\mu$ is the Lebesgue measure on $\R$, $\rho_0$ the restriction of the Lebesgue measure to $[-0.5,0.5]$, $K(x,y) = |x-y|$ and $\eta(x,y)= 1/(x-y)^2$, then the solutions develop delta mass concentrations at $0$ in finite time. Understanding for which $K$ and $\eta$ solutions do develop finite time singularities is an interesting open problem.

We note that when defining the flux in \eqref{eq:nce} we define the density along edges to be the density at the source; analogously to an upwind numerical scheme. 
While, as we show, this leads to a convenient framework to consider the dynamics, it creates the difficulty that the resulting distance, that we are about to define, is not symmetric and is thus only a quasi-metric. 
\subsubsection*{Upwind nonlocal transportation metric}
We use the nonlocal continuity equation~\eqref{eq:nce} to define a \emph{nonlocal Wasserstein quasi-distance} in analogy to the Benamou--Brenier formulation~\cite{BB2000} for the classical Kantorovich--Wasserstein distances~\cite{V1}. That is, for two probability measures $\rho_0,\rho_1\in\mP_2(\Rd)$, let
\begin{equation}\label{eq:ben-bre-formula}
    \mt_\mu(\rho_0,\rho_1)^2:=\inf_{(\rho,v)\in \CE(\rho_0,\rho_1)}\left\{\int_0^1\iint_G|v_t(x,y)_+|^2\eta(x,y)\di\rho_t(x)\di\mu(y)\di t\right\},
\end{equation}
where $\CE(\rho_0,\rho_1)$ is the set of weak solutions $\rho$ to the nonlocal continuity equation (see Definition~\ref{def:nce-flux-form}) on $[0,1]$ with $\rho(0)=\rho_0$ and $\rho(1)=\rho_1$. We note that the notion of the nonlocal Wasserstein distance for measures on $\R^d$ was introduced  by  Erbar in \cite{Erb14}, who used it to study the fractional heat equation. One difference is that the interpolation we consider is beyond the scope of \cite{Erb14}. 
Very recently \cite{PeletierRossiSavareTse2020} has extended the gradient flow viewpoint of the jump processes to
generalized gradient structures driven by a broad class of internal energies.

Another difference is that here the measure $\mu$ plays an important role in how the action is measured and allows one to incorporate seamlessly both the continuum case (e.g., $\mu$ is the Lebesgue measure on~$\R^d$) and the graph case ($\mu$ is the empirical measure of the set of vertices).

The notions above are rigorously developed in Section \ref{sec:ce-ut}, where we list the precise assumption~\eqref{it:as:pos-sym-lsc} on the edge weight function $\eta$ and the joint assumptions~\ref{it:as-conv} and~\ref{it:as-tight} on $\eta$ and the underlying measure $\mu$. We then rigorously introduce the action (Definition \ref{def:action}), which is a nonlocal analogue of kinetic energy; we show its fundamental properties, in particular joint convexity (Lemma \ref{lem:convexity}) and lower semicontinuity with respect to narrow convergence (Lemma \ref{lem:l.s.c.action}).
In Section \ref{subsec:nonloc-cont-eq} we rigorously introduce the nonlocal continuity equation in measure-valued flux form \eqref{eq:nlce_measures}; we introduce the notion on all of $\R^d$ where $\mu$ does not initially play a role. The measure $\mu$ enters the framework by considering paths of finite action.
Proposition \ref{prop:compactness-sol-ce} establishes an important compactness property of sequences of solutions.
In Section \ref{sec:nl-trans} we turn our attention to the nonlocal Wasserstein quasi-metric based on the upwind interpolation, which we introduce in Definition \ref{defn:metric}. The compactness of solutions of the nonlocal continuity equation and the lower semicontinuity of the action imply the 
existence of (directed) geodesics (Proposition \ref{prop:min-metric}). 
Following the work of Erbar \cite{Erb14} we show that the nonlocal Wasserstein quasi-metric generates a topology on the set of probability measures which is stronger than the $W_1$ topology (i.e., the Monge distance or the $1$-Wasserstein metric). 
Analogously to \cite{AGS} we show the equivalence between the paths of finite length with respect to the quasi-metric and the solutions of the nonlocal continuity equation with finite action (Proposition~\ref{prop:min-metric}). 
The set of probability measures endowed with the quasi-metric $\mt$ has a formal structure of a Finsler manifold, and parts of this structure can be described; in particular, in \eqref{eq:charact:tangent} we describe the tangent space at a given measure $\rho$ using the fluxes. We note that using fluxes, instead of velocities, is necessary since, because of the upwinding, the relation between the velocities and the tangent vectors is not linear (Proposition \ref{prop:tangent-bundle}) and in particular not symmetric. 
For this reason the resulting gradient structure is also different to the large class of nonlinear, however still symmetric, flux-velocity relations considered in~\cite{PeletierRossiSavareTse2020}.
We conclude Section \ref{sec:ce-ut} by showing that, given a measure $\mu$, the finiteness of the action ensures that any path starting within the support of $\mu$ will remain within the support of~$\mu$ (Proposition \ref{prop:supp-in-mu}).
\subsubsection*{Nonlocal nonlocal-interaction equation}
In Section~\ref{sec:NLIE} we develop the existence theory of the equation~\eqref{eq:nlnl-intro} based on the interpretation as the gradient flow of $\cE$ with respect to the quasi-metric~$\mt$ defined in \eqref{eq:ben-bre-formula}.
We begin by listing the precise conditions~\ref{as:K:cont}--\ref{as:K:LipQuad} on the interaction kernel~$K$. We note that these are less restrictive than the typical conditions for the well-posedness of the standard nonlocal-interaction equation in Euclidean setting
\cite{AGS, CDFLS11}.

Before we turn to the rigorous theory of weak solutions as curves of maximal slope on quasi-metric space, we discuss the gradient flow structure in a more geometric setting, namely the Finsler structure related to~$\mt$. 
Indeed, the action (formally given by the time integrand in \eqref{eq:ben-bre-formula}, and rigorously defined by \eqref{eq:def-action}) defines a positively homogeneous norm (namely a Minkowski norm) on the tangent space. The Hessian of the square of the norm endows the tangent space at each measure with the formal structure of a Riemann manifold. We compute this Riemann metric in Appendix ~\ref{app:minkowski} under an absolute-continuity assumption. With this assumption, we show that \eqref{eq:nlnl-intro} is the gradient flow of $\cE$ with respect to the Finsler structure in Section \ref{subsec:HeuristicFinsler}. For simplicity, we consider $P \equiv 0$, since the extension to $P \not\equiv 0$ is straightforward, as it is explained in Remark~\ref{rem:ExtPot}.
 
In Section \ref{sec:DeGiorgi} we develop the rigorous gradient descent formulation based on curves of maximal slope in the space of probability measures endowed with the quasi-metric $\mt$. The theory of gradient flows in the spaces of probability measures endowed with the standard Wasserstein metric was developed in \cite{AGS}. Here we extend it to the setting of quasi-metric spaces, endowed with the nonlocal Wasserstein distance. This requires several delicate arguments. We start by introducing the notions of one-sided strong upper gradient (Definition \ref{eq:def-strong-upper-gr}) and curves of maximal slope (Definition~\ref{defn:ls-Giorgi}). We define the local slope $\cD$ in \eqref{eq:def:dissipation} by using a heuristically derived gradient of the energy $\cE$, and show, using a chain rule established in Proposition \ref{prop:chain-rule}, that 
$\sqrt{\cD}$ is a one-sided strong upper gradient for $\cE$ with respect to $\cT$. One of our main results is Theorem \ref{thm:CurveMaxSlope}, which establishes the equivalence between curves of maximal slope and weak solutions of \eqref{eq:nlnl-intro}. 
In Section \ref{sec:stab-exist} we prove several important results. Namely Theorem \ref{thm:graph:stability} establishes
that the De Giorgi functional~$\cG_T$ is stable under variations of the base measure $\mu$ and of the solutions. A consequence of this result is the convergence of solutions of \eqref{eq:nlnl-intro} on graphs defined on random samples of a measure to solutions of \eqref{eq:nlnl-intro} corresponding to the full underlying measure (Remark \ref{rem:DC-data}). The proof of Theorem \ref{thm:graph:stability} relies on 
the lower semicontinuity of the local slope (Lemma \ref{lem:lsc-dissipation}) and 
the lower semicontinuity of the De Giorgi functional (\ref{lem:lsc-Giorgi}). Another important consequence is the existence of weak solutions of  \eqref{eq:nlnl-intro}, which is proved in Theorem \ref{thm:existence-weak}.  
\subsection{Relation to the numerical finite-volume upwind scheme}\label{subsec:numerics}
Equation \eqref{eq:nce} can be interpreted in several ways. For example, it can be understood as the master equation of a continuous-time and continuous-space Markov jump process on the graphon $(\Rd, \eta)$, that is, a continuous graph with vertices $\Rd$, and symmetric weight $\eta(x,y)$ for $(x,y)\in \set{(x,y)\in \R^d\times \R^d: x\ne y}$. The stochastic interpretation is that a particle at position $x\in\R^d$ jumps according to the measure $v(x,y)_+\eta(x,y)\di\mu(y)$ to $y\in\R^d$. In this way it gives rise to a Markov jump process related to the numerical upwind scheme.

The numerical upwind scheme is one of the basic finite-volume methods used to solve  conservation laws; see \cite{EGH2000}. 
To draw the connection, let $\set{x_1, \dots, x_n}$ be a suitable representative of a tessellation $\set{K_1,\dots,K_n}$, for instance a Voronoi tessellation, of some bounded domain $\Omega\subset \R^d$. Let $\mu$ be the Lebesgue measure on $\Omega$ and take $\eta$ to be the transmission coefficient common in finite-volume schemes: $\eta(x_i,x_j) = \mathcal{H}^{d-1}(\overline{K_i}\cap \overline{K_j})/\operatorname{Leb}(K_i)$, for $i,j\in\{1,\dots,n\}$, where $\cH^{d-1}(\overline{K_i}\cap \overline{K_j})$ is the~$d-1$~dimensional Hausdorff measure of the common face between $K_i$ and $K_j$.
With this choice the equation \eqref{eq:nce} becomes the (continuous-time) discretization of the classical continuity equation
\[
  \partial_t \rho_t + \nabla \cdot \bra*{ \mathbf{v}_t \, \rho_t } = 0 ,
\]
for some vector field $\mathbf{v}_t\colon \Omega \to \R^d$. Hereby, the discretized vector field $v_t$ is obtained from $\mathbf{v}_t$ by taking the average over common interfaces:
\[
  v_t(x_i,x_j) = \frac{1}{\mathcal{H}^{d-1}(\overline{K_i}\cap \overline{K_j})} \int_{\overline{K_i}\cap \overline{K_j}} \mathbf{v}_t \cdot \nu_{K_i,K_j} \di\mathcal{H}^{d-1},
\]
where $\nu_{K_i,K_j}$ is the unit normal to $K_i$ pointing from $K_i$ to $K_j$. We refer to the recent work~\cite{Cances2019} for a variational interpretation of the upwind scheme, which is close to that we propose for the more general equation~\eqref{eq:nce}. Earlier results in this direction are contained in~\cite{Miel13,DL2015}.

The connection to finite-volume schemes explains also that the nonlocality in~\eqref{eq:nce} introduces a regularization, which in the numerical literature is referred to as \emph{numerical diffusion}. That the numerical diffusion is actually an honest Markov jump process, as described at the beginning of this section, was observed and used to find optimal convergence rates in the works~\cite{DelarueLagoutiere11,DelarueLagoutiereVauchelet16,SS2017,SS2018}. 
\subsection{Comparison with other discrete metrics and gradient structures}\label{subsec:flux:SG}
The interpretation of diffusion on graphs as gradient flows of the entropy was independently carried out in~\cite{Maas11,Mielke2011gradient,CHLZ12}. Here we recall the descriptions of the flows relying on reversible Markov chains, which was the framework used in~\cite{EFLS2016,ErbMaa14,Maas11}. Starting with Markov chains, which then determine the edge weights, offers an additional layer of modeling flexibility. In particular, consider the Markov chain with state space $X = \{x_1, \dots, x_n\}$ and  jump rates
$\set{Q_{x,y}}_{x,y\in X}$. Let 
$\pi_x$ be the reversible probability measure for the Markov chain, meaning that it satisfies the detailed balance condition $\pi_x Q_{x,y} = \pi_y Q_{y,x}$. The edge weights $\set{w_{x,y}}_{x,y\in X}$ are given by $w_{x,y}=\pi_x Q_{x,y}$.  The energy considered is the relative entropy: for $\rho\colon X \to [0,1]$ with $\sum_{x \in X} \rho_x = 1$ we define
\begin{equation}\label{e:relEnt}
  \cH(\rho \mid \pi) = \sum_{x} \rho_x \log \frac{\rho_x}{\pi_x} = \sum_{x} \rho_x \log \rho_x - \sum_{x} \rho_x \log \pi_x = \cS(\rho) + \cE_{P}(\rho) \quad\text{with}\quad P_x = -\log \pi_x . 
\end{equation}
The paths in the configuration space are given as the solution of the continuity equation which for the flux $\set{j_{x,y}\colon [0,T]\to \R}_{x,y\in X}$ takes the form \eqref{eq:intro:CE}.

To compute the flux from a given velocity $\set{v_{x,y}}_{x,y\in X}$ (an edge-based quantity) and density $\set{\rho_x}_{x\in X}$ (a vertex-based quantity), one interpolates the densities at vertices to define the density (and hence the flux) along the edges. The literature so far has considered a proportional constitutive relation of the form
\begin{equation}\label{e:flux:velocity:MC}
  j_{x,y} = v_{x,y} \, \theta\bra*{\frac{\rho_x}{\pi_x},\frac{\rho_y}{\pi_y}} ,
\end{equation}
where the function $\theta:\R_+\times \R_+\to \R_+$ needs to be one-homogeneous for dimensional reasons. In addition, it is assumed that the function $\theta$ is an interpolation, that is, $\min\set{a,b}\leq \theta(a,b)\leq \max\set{a,b}$. The choice providing a gradient flow characterization for linear Markov chains is the logarithmic mean, defined by $\theta(a,b)= \frac{a-b}{\log a - \log b}$ for $a \ne b$ and $\theta(a,a)=a$. 

The associated transportation distance is obtained by minimizing the action functional
\begin{align}
  \cA(\rho,j) = \frac{1}{2}\sum_{x,y} j_{x,y} \, v_{x,y} \, w_{x,y}
  &= \frac{1}{2} \sum_{x,y} \frac{\abs*{ j_{x,y}}^2}{\theta\bra*{\frac{\rho_x}{\pi_x},\frac{\rho_y}{\pi_y}}} \pi_x Q_{x,y} .\label{e:action:discrete}
\end{align}
The corresponding transportation distance is induced as the minimum of the action along paths:
\begin{equation*}
  \cT^{\theta}(\rho_0,\rho_1) = \inf\set*{ \int_0^1 \cA(\rho(t),j(t)) \dx{t} : \bra[\big]{\rho(t),j(t)}_{t\in [0,1]} \text{ solves~\eqref{eq:intro:CE} and } \rho(0)=\rho_0,\, \rho(1)=\rho_1 } .
\end{equation*}
As we do in Corollary \ref{cor:antsym_vect_field_lower_action}, it was shown that it suffices to consider antisymmetric fluxes. 
To arrive at a gradient flow formulation, one considers the metric induced by the action function~\eqref{e:action:discrete}:
\begin{equation}\label{e:metric:discret}
  g_\rho(j^1,j^2) = \frac{1}{2} \sum_{x,y} \frac{j_{x,y}^1 \, j_{x,y}^2}{\theta\bra*{\frac{\rho_x}{\pi_x},\frac{\rho_y}{\pi_y}}} \pi_x Q_{x,y} .
\end{equation}
Then the gradient $\grad \cH$ of the relative entropy~\eqref{e:relEnt} with respect to this metric is given as the antisymmetric flux $j^*$ of minimal norm satisfying
\begin{equation}\label{e:def:grad:discrete}
  g_\rho(\grad \cH,j) = \Diff \cH[j] = \left.\pderiv{}{t}\right|_{t=0} \cH(\tilde \rho(t)) , 
\end{equation}
for any curve $(\tilde \rho(t))_{t\geq 0}$ such that $\partial_t \rho(0) = - \bra[\big]{\dgrad \cdot j}$. 
Expanding \eqref{e:def:grad:discrete} and using that $j^*$ is antisymmetric gives
\begin{equation*}
  \frac{1}{2} \sum_{x,y} \frac{j_{x,y} \, j^*_{x,y}}{\theta\bra*{\frac{\rho_x(t)}{\pi_x},\frac{\rho_y(t)}{\pi_y}}} \pi_x Q_{x,y}  = - \frac{1}{2}\sum_{x,y} \bra*{\log \frac{\rho_x(t)}{\pi_x} - \log\frac{\rho_y(t)}{\pi_y}} j_{x,y} \, w_{x,y}.
\end{equation*}
Since this identity holds for all $j_{x,y}$, the flux $j^*$ is identified by
\[
  j_{x,y}^* = - \bra*{\log \frac{\rho_x}{\pi_x} - \log \frac{\rho_y}{\pi_y}} \theta\bra*{\frac{\rho_x(t)}{\pi_x},\frac{\rho_y(t)}{\pi_y}} = -\bra*{\frac{\rho_x(t)}{\pi_x} - \frac{\rho_y(t)}{\pi_y}},
\]
where the last equality holds for the particular choice of the logarithmic mean interpolation 
$\theta(r,s) = \frac{r-s}{\ln r - \ln s}$. 
By plugging $j_{x,y}^*$ into the continuity equation~\eqref{eq:intro:CE}, one recovers the (linear) heat equation on graphs. 

The next relevant step is the introduction of the interaction and the potential energies as in~\eqref{eq:discE}. In particular, ~\cite{EFLS2016} provides a gradient structure for free energy functionals of the form
\begin{equation}\label{e:cF:discrete}
  \cF_{\beta}(\rho) = \beta^{-1} \cS(\rho) + \cE_X(\rho) ,
\end{equation}
where $\beta>0$ is the inverse temperature. This is nontrivial since considering the logarithmic interpolation, which makes the diffusion term linear, would make the potential term nonlinear, and thus the Fokker--Planck equation on graphs would be nonlinear. 
To cope with this, the framework of~\cite{EFLS2016} extends the linear theory outlined above to a family of nonlinear Markov chains satisfying a local detailed balance condition. 
The consequence for the resulting gradient structure is that the quantities $\set{\pi_x}_{x\in X}$, $\set*{Q_{x,y}}_{x,y\in X}$ and $\set*{w_{x,y}}_{x,y\in X}$ depend on the current state $\rho$ in such a way that the detailed balance condition $w_{x,y}[\rho] = \pi_x[\rho] Q_{x,y}[\rho] = \pi_y[\rho] Q_{y,x}[\rho] $ is still valid for all $\rho\in \cP(X)$.
In particular, for $\cF_\beta$ defined in~\eqref{e:cF:discrete}, it holds that 
\begin{equation*}
  \pi_x[\rho] = \frac{1}{Z_\beta} \exp\bra*{ - \beta\bra*{ P_x + \sum_{y} K_{x,y}\rho_y}} \quad\text{with}\quad Z_\beta= \sum_x \exp\bra*{ - \beta\bra*{ P_x + \sum_{y} K_{x,y}\rho_y}}.
\end{equation*}
It would be natural to try to build the framework for the case $\beta=\infty$, which we consider in this paper, by taking the limit $\beta \to \infty$ in the framework of~\cite{EFLS2016}. 
It turns out that this limit is singular for the constructed gradient structure. First of all, the measure $\pi_x[\rho]$ degenerates at all points except at the argmin of the effective potential $x\mapsto  P_x + \sum_{y} K_{x,y}\rho_y$. This causes the constitutive relation~\eqref{e:flux:velocity:MC} to become meaningless. A more detailed analysis also shows that the metric in~\eqref{e:metric:discret} degenerates. 

We also note that in this setting the potential functions $P$ and $K$ and inverse temperature $\beta$ enter the metric in~\eqref{e:action:discrete} through the weights $w_{x,y}$ and rate matrix $Q_{x,y}$. This is in stark contrast with the continuous classical gradient flow formulation for free energies of the form~$\cF_\beta$ form~\eqref{e:cF:discrete}, where the metric is always the $L^2$-Wasserstein distance, independently of the potentials $P$ and $K$ and also of the inverse temperature $\beta>0$, including $\beta=\infty$~\cite{AGS,JKO,CMcCV03,CDFLS11}.

The above problems lead us to consider the upwind interpolation in 
the flux-velocity relation~\eqref{e:flux:velocity:MC}. In view of~\eqref{eq:intro:CE}, this relation is replaced in the present setting by 
\begin{equation}\label{eq:upwind:flux}
  j_{x,y} = \rho_x (v_{x,y})_+ - \rho_y (v_{x,y})_- = \Theta(\rho_x, \rho_y; v_{x,y}) v_{x,y} \qquad \textrm{where } \quad
  \Theta(a,b; v) = \begin{cases}
                     a \; & \textrm{if }  v>0, \\
                     b  & \textrm{if }  v<0, \\
                     0 & \textrm{if } v=0.
                   \end{cases}
\end{equation}
Note that the relation~\eqref{eq:upwind:flux} is a functional relation between velocity and flux with the interpolation~$\Theta$ depending on the velocity. 

We remark that solutions of system \eqref{eq:intro:CE}--\eqref{eq:intro:vE}
are not the limit of the gradient flows in \cite{EFLS2016} as $\beta \to \infty$. We emphasize here that the limit of these dynamics as $\beta \to \infty$ would in fact not be the desirable gradient flow of the nonlocal-interaction energy, since the initial support of the solutions would never expand; see the related Remark \ref{rem:other-interp}. 

We conclude this section by observing that it seems possible to generalize the upwind interpolation in a continuous way to define a flux-velocity relation to deal with free energies $\cF_\beta$ for $\beta>0$.
A candidate, inspired by the Scharfetter--Gummel scheme \cite{SG1969}, is the following constitutive flux-velocity relation depending on $\beta$:
\begin{equation*}
  j^\beta_{x,y} = v_{x,y} \, \frac{\rho_x \exp\bra*{\beta v_{x,y}/2} - \rho_y \exp\bra*{-\beta v_{x,y}/2}}{ \exp\bra*{\beta v_{x,y}/2} -  \exp\bra*{-\beta v_{x,y}/2}}.
\end{equation*}
In particular, it holds that $j^\beta_{x,y} \to j_{x,y}$ as $\beta\to \infty$, where $j_{x,y}$ is as in~\eqref{eq:upwind:flux}. 
The form of $j^\beta_{x,y}$ can be physically deduced from the one-dimensional cell problem for the unknown value $j^\beta_{x,y}\in \R$ and function $\rho\colon [0,1]\to \R$:
\[
  j^\beta_{x,y} = -\beta^{-1} \dgrad \rho(\cdot) + v_{x,y} \, \rho(\cdot) \quad\text{on } [0,1], \qquad\text{with } \rho(0)= \rho_x \text{ and } \rho(1) = \rho_y . 
\]
Note that $j^\beta_{x,y} = \frac{\rho_x-\rho_y}{\beta}$ for $v_{x,y} =0$, which is the flux due to Fick's law. Likewise, $j^\beta_{x,y} = 0$ for $v_{x,y} = \beta^{-1} \log\frac{\rho_y}{\rho_x}$, which is the velocity needed to counteract the diffusion. 
In~\cite{SchlichtingSeis20}, it is shown that the Scharfetter--Gummel finite volume scheme provides a stable positivity preserving numerical approximation of the diffussion-aggregation equation, which also respects the thermodynamic free energy structure. We pursue the investigation of the existence of a possible related gradient structure in future research.
\subsection{Connections to data science}  \label{sec:ML}
Part of the motivation for the present work comes from applications to data science. 
Here we introduce a family of nonlinear gradient flows that is relevant to discovering local concentrations in networks akin to modes of a distribution. 
\begin{figure}[ht]
  \includegraphics[width=0.48\textwidth]{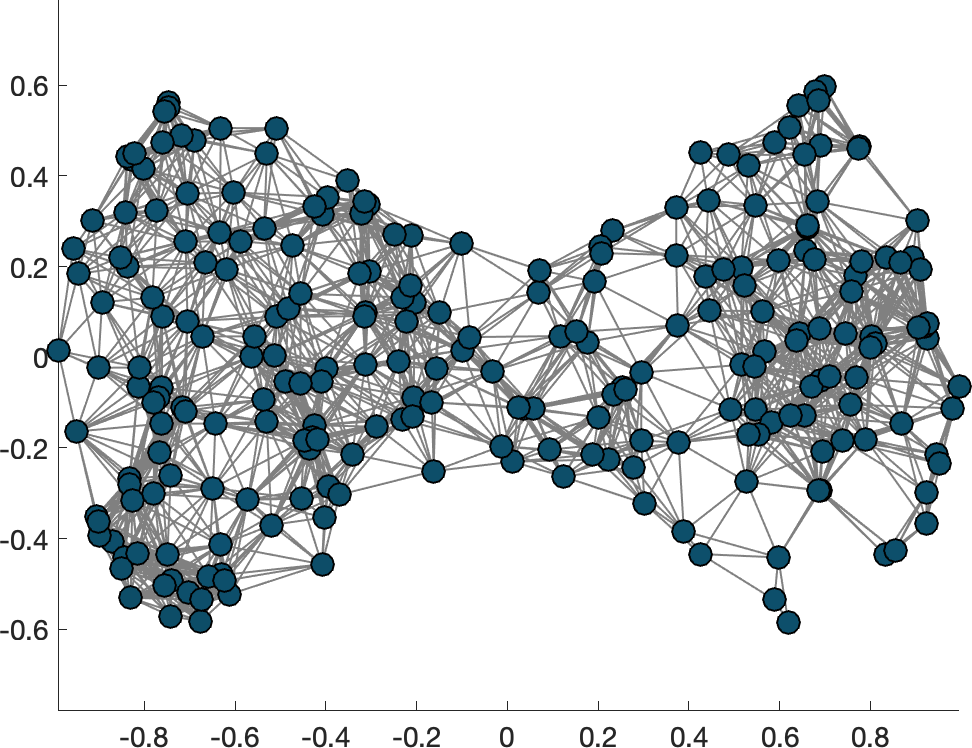} \hspace*{4pt}
  \includegraphics[width=0.48\textwidth]{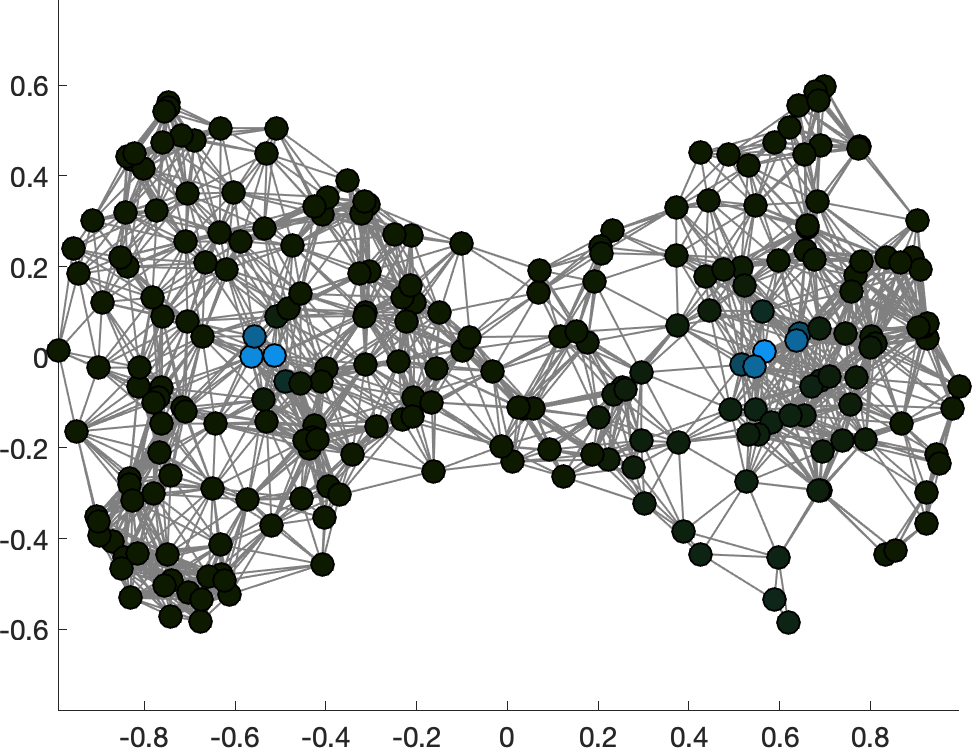}
  \caption{A solution of the nonlocal-interaction equation on graphs driven by the energy \eqref{eq:discE}. We consider a random geometric graph based on 240 sample points $X$ from a 2D bean data set. 
  The connectivity distance is $\varepsilon=0.23$.
  The edge weights are $w_{x,y} = \exp(-24|x-y|^2) $, provided that the vertices $x$ and $y$ of $X$ are connected.
  The interaction potential is $K_{x,y} = 1-\exp(-8|x-y|^2)$ and the external potential is $P\equiv0$. The solution, starting from a uniform distribution, is shown at time $t=200$.  Brighter color indicates more mass (right). 
  }\label{fig:NLIEbean}
\end{figure}

Our main interest is in equations posed on graphs whose vertices are random samples of some underlying distribution and whose edge weights are a function of distances between vertices. In machine learning one often deals with data in the form of a point cloud in high-dimensional space. While the ambient dimension may be very large, the data often possess an underlying low-dimensional structure that can be used in making reliable inferences about the underlying data distribution. To use the geometric information, we follow one of the standard approaches and consider graphs associated to point clouds. Formulating the machine learning tasks directly on the point cloud enables one to access the geometric structure of the distribution in a simple and computationally efficient way. The works in the literature have mostly focused on models based on minimizing objective functionals modeling tasks such as clustering or dimensionality reduction \cite{Belkin02laplacianeigenmaps, GTS16, GTSspectral, JordanNg, KanVemVet04}, or based on characterizing clusters through estimating some property of the data distribution (most often the density); see \cite{CDKvL14} and references therein. 
Only few dynamical models have been considered---notable among them are diffusion maps \cite{CoifmanLafon}, where the heat equation is used to redistance the points. 

Here we focus on models that are motivated by nonlocal PDEs.
Consider a probability measure $\mu$ on $\R^d$ with finite second moments. 
Let  $X =\{x_1, \dots, x_n\}$ be random i.i.d.\ samples of the measure $\mu$. Let $\mu^n = \frac{1}{n} \sum_{i=1}^n \delta_{x_i}$ be the empirical measure of the sample and let $K\colon \R^d \times \R^d \to \R$ be symmetric and $P:\R^d \to \R$. The total energy $\cE_X:\cP(X)\to \R$, given in~\eqref{eq:discE}, for the empirical measure $\mu^n$ can be rewritten as
\begin{equation}\label{eq:discE:empirical}
  \cE_X(\mu^n) = \cE_I(\mu^n) + \cE_P(\mu^n) =  \frac12  \iint_{\Rd\times\Rd} K(x,y) \di \mu^n(x) \di\mu^n(y) + \int_\Rd P(x) \di\mu^n(x). 
\end{equation}
The gradient flow of $\cE_X$  with respect to the graph Wasserstein metric $\mt_{\mu^n}$ defined in \eqref{eq:ben-bre-formula} is described by the  ODE system \eqref{eq:intro:CE}--\eqref{eq:intro:vE}, where $K_{x_i,x_j} = K(x_i,x_j)$ and $P_{x_i} =P(x_i)$ for all $i,j\in\{1,\dots,n\}$. Another evolution by such system is illustrated on Figure \ref{fig:NLIEbean}.

Here we remark on the contrast between \eqref{eq:intro:CE}--\eqref{eq:intro:vE} and the gradient flow of \eqref{eq:discE:empirical} in the ambient space~$\R^d$, with respect to the standard Wasserstein metric, which takes the form
\begin{equation} \label{eq:discRdGF}
\dot x_i =  - \nabla P(x_i)-   \sum_{j=1}^n \rho_j \nabla_{x_i}  K(x_i , x_j).
\end{equation}
The first notable difference is that, on the graph, masses change and the positions remain fixed, while in~$\R^d$ positions change and the masses remain fixed. This difference is somewhat superficial, since both equations describe the rearrangement of mass in order to decrease the same energy in the most efficient way measured by two different metrics. The main difference is that the graph encodes the geometry of the space that mass is allowed to occupy. In particular, it ensures that the geometric mode discovered will be a data point itself. 

We note that the popular mean-shift algorithm \cite{CMMS} can be interpreted as a time-stepping algorithm to approximate solutions of \eqref{eq:discRdGF} with $K\equiv 0$ and $P = \ln(\theta * \mu^n(0))$, where $\mu^n(0)$ is the empirical measure of the initial distribution of particles and  $\theta * \mu^n(0)$ is the kernel density estimate of the density $\brho$ of the underlying distribution. Namely the step of the mean-shift algorithm is to replace the position of the particle at $x_j$ by the center of mass of $\theta( \,\cdot \, - x_j)* \mu_n(0)$ and iterate the procedure. Formal expansion shows that this is a time step of the forward scheme for the flow driven by $P = \ln(\theta * \mu^n(0))$.
We note that considering the gradient flow of the corresponding energy on the graph~\eqref{eq:intro:CE}--\eqref{eq:intro:vE}
ensures that the modes of the distribution discovered by the (graph) mean-shift algorithm will remain within the data set. Furthermore, we note that adding nonlocal attraction on the graph progressively clumps nearby masses together and thus provides an approach to agglomerative clustering. 

One of our main results, stated in Theorem \ref{thm:graph:stability}, is that as $n \to \infty$ the solutions of the graph-based equation~\eqref{eq:intro:CE}--\eqref{eq:intro:vE} narrowly converge along a subsequence to a solution of the nonlocal nonlocal-interaction equation \eqref{eq:nlnl-intro}.

\section{Nonlocal continuity equation and upwind transportation metric} \label{sec:ce-ut}
\subsection{Weight function}\label{subsec:weight}
Throughout the paper we consider a \emph{weight function} $\eta\colon \{(x,y)\in \R^d\times \R^d : x\ne y\} \to[0,\infty)$, which shall always satisfy
\begin{equation}\label{it:as:pos-sym-lsc}\tag{\textbf{W}}
\begin{cases}
 \text{$\eta$ is continuous on $G=\set[\big]{(x,y) \in \Rd\times\Rd : x\neq y,\, \eta(x,y)>0}$;}\\
 \text{$\eta$ is symmetric and non-negative, that is $\eta(x,y)=\eta(y,x)$ for all $(x,y) \in G$.}
\end{cases}
\end{equation}
Since $\eta$ is symmetric, we regard the edges set $G$ as undirected graph. Many of the edge-based quantities we consider, like vector fields and fluxes, will lie in an $\eta$-weighted $L^2$ space, $L^2(\eta\, \lambda)$ for some $\lambda \in \cM(G)$.
The space $L^2(\eta\,\lambda)$ is equipped with the inner product 
\begin{equation}\label{eq:def:L2G}
  \langle f,g\rangle_{L^2(\eta\,\lambda)} = \frac{1}{2} \iint_G f(x,y) g(x,y) \eta(x,y) \dx\lambda(x,y) \qquad \text{for all } f,g\in L^2(\eta\,\lambda) , 
\end{equation}
where the factor $\frac12$ ensures that each undirected edge is counted only once. 

Below we state two assumptions on the base measure $\mu\in\cM^+(\Rd)$ and the weight function $\eta$, where we use the notation $\vee$ to denote the maximum.
\begin{enumerate}[label=\textbf{(A\arabic*)}]
\item\label{it:as-conv} (moment bound) The family of functions $\{\bra*{ |x-\cdot|^2 \vee |x-\cdot|^4 }\eta(x,\cdot)\}_{x\in\Rd}$ is uniformly integrable with respect to $\mu$, that is, for some $C_\eta \in (0,\infty)$ it holds
\begin{equation*}
  \sup_{x\in \R^d} \int \bra*{ |x-y|^2 \vee |x-y|^4 }  \, \eta(x,y)\di \mu(y) \leq C_\eta\, .
\end{equation*}
\item\label{it:as-tight} (local blow-up control) The family of measures $\{|x - \cdot|^2\eta(x,\cdot) \mu(\cdot)\}_{x\in\Rd}$ is locally uniformly integrable, that is, 
\begin{equation*}
  \lim_{\eps \to 0}  \sup_{x\in \R^d} \int_{B_\eps(x)\setminus\{x\}} \abs{x-y}^2 \, \eta(x,y) \di\mu(y)= 0, \quad\text{ where } B_\eps(x) = \set[\big]{y\in \R^d: |x-y|<\eps}.
\end{equation*}
\end{enumerate}
\begin{rem}
Continuity on $G$ in \eqref{it:as:pos-sym-lsc} is needed to obtain lower semicontinuity of the action functional; see Lemma~\ref{lem:l.s.c.action}. Assumption~\ref{it:as-conv} ensures well-posedness of the nonlocal continuity equation we shall introduce in Section~\ref{subsec:nonloc-cont-eq}, whereas Assumption~\ref{it:as-tight} is necessary for compactness of solutions to the nonlocal continuity equation; see Proposition~\ref{prop:compactness-sol-ce}.
\end{rem}
\begin{exm}
  Typically the function $\eta$ is a function of the distance: 
  \[
    \eta(x,y) = \vartheta\bigl(|x-y|) \quad \mbox{for all $(x,y)\in G$},
  \]
  where $\vartheta \colon (0,\infty) \to [0,\infty)$ is continuous on $\set{\vartheta>0}$ and satisfies analogues of~\ref{it:as-conv} and~\ref{it:as-tight}. An important example are geometric graphs with connectivity distance given by $\varepsilon>0$ and weight
  \begin{equation}\label{eq:def:eta:local}
        \eta_\varepsilon(x,y) = \frac{2(2+d)}{\varepsilon^2} \frac{\chi_{B_\varepsilon(x)}(y)}{|B_\varepsilon|} \quad \text{for all $(x,y) \in G$} .
  \end{equation}
In this example, fixing $\mu = \operatorname{Leb}(\Rd)$, we conjecture that the weak formulation of~\eqref{eq:nlnl-intro}---see Section \ref{sec:NLIE}---converges to the nonlocal aggregation equation $\partial_t \rho_t = \nabla \cdot\bra*{\rho _t \nabla K*\rho_t+ \rho_t \nabla P}$ as $\eps\to 0$ for sufficiently smooth potentials $K$ and $P$. See Section~\ref{subsec:local_limit} for a discussion on the local limit. 
\end{exm}
\subsection{Action}\label{subsec:action}
The form of the action inside \eqref{eq:ben-bre-formula} seems practical, but it does not have any obvious convexity and lower semicontinuity properties. Therefore, we define the action in flux variables. We start by introducing some notation.
For a signed measure $\bj\in\cM(G)$, we denote by $\bj=\bj^+-\bj^-$ its Jordan decomposition. Moreover, for any measurable $A\subseteq G$, let $A^\top=\{(y,x) \in \Rd\times\Rd: (x,y)\in A\}$ be its transpose. Likewise, for $\bj\in\cM(G)$, we denote by $\bj^\top$ the transposed measure defined by $\bj^\top(A)=\bj(A^\top)$.

For any measures $\mu\in\cM^+(\Rd)$ and $\rho\in\mP(\Rd)$, we define the (restricted) product measures $\gamma_i\in\cM^+(G)$ for $i=1,2$ as
\begin{equation}\label{eq:def:gamma}
 \di\gamma_1(x,y) = \di\rho(x)\di\mu(y) \qquad\text{and}\qquad \di\gamma_2(x,y) = \di\mu(x) \di\rho(y) \qquad\text{for $(x,y)\in G$ }. 
\end{equation}
Note that $\gamma^\top_1 = \gamma_2$. 
We define the action for general $\eta$ which we only require to 
 satisfy Assumption~\eqref{it:as:pos-sym-lsc}, i.e., continuity on $G$, symmetry and positivity.
\begin{definition}[Action]\label{def:action}
For $\mu\in\cM^+(\Rd)$, $\rho\in\mP(\Rd)$ and $\bj\in\cM(G)$, consider $\lambda\in\cM(G)$ such that $\rho\otimes\mu,\mu\otimes\rho,|\bj|\ll|\lambda|$. We define
\begin{equation}\label{eq:def-action}
    \cA(\mu;\rho,\bj)=\frac12\iint_G\left(\alpha\left(\frac{\di\bj}{\di|\lambda|},\frac{\di(\rho\otimes\mu)}{\di|\lambda|}\right)+\alpha\left(-\frac{\di\bj}{\di|\lambda|},\frac{\di(\mu\otimes\rho)}{\di|\lambda|}\right)\right)\eta \di|\lambda| .
\end{equation}
Hereby, the lower semicontinuous, convex, and positively one-homogeneous function $\alpha\colon \R\times\R_+\to\R_+\cup\{\infty\}$ is defined, for all $j\in\R$ and $r\geq 0$, by
\begin{equation}\label{eq:def:alpha}
\alpha(j,r):=\begin{cases}
\frac{(j_+)^2}{r} \qquad &\text{if}\ r>0,\\
0 \qquad &\text{if}\ j\leq 0\ \text{and}\ r=0,\\
\infty \qquad &\text{if}\ j> 0\ \text{and}\ r=0,
\end{cases}
\end{equation}
with $j_+=\max\{0,j\}$. If the measure $\mu$ is clear from the context, we  write $\cA(\rho,\bj)$ for $\cA(\mu;\rho,\bj)$.
\end{definition}
Note that Definition \ref{def:action} is well-posed since the one-homogeneity of $\alpha$ makes it independent of the particular choice of $\lambda$ as long as the absolute continuity condition in Definition \ref{def:action} is satisfied. An example of such measure is a $\lambda$ such that $|\lambda|=|\rho\otimes \mu|+|\mu\otimes \rho|+|\bj|$. Moreover, $\lambda$ can be chosen symmetric, otherwise it can be replaced by $\frac{1}{2}(\lambda +\lambda
^\top)$.
\begin{rem} \label{rem:action-mu}
We note that the action is inversely proportional to the measure $\mu$: doubling the measure $\mu$ leads to halving the action. This has important consequence for the way $\mu$ influences the geometry of the space of measures. In particular, $\mu$ not only sets the region where mass can be transported, but also makes the transport less costly in the regions of high density of $\mu$.
\end{rem}
\begin{rem}\label{rem:action:abs_cont}
    If $\rho\ll \mu$, then we denote its density by $\rho$ by abuse of notation, and if furthermore $\bj\ll \mu\otimes \mu$ with density $j$, then it holds
\begin{equation}
 \cA(\mu;\rho,\bj) = \frac12\iint_G \bra*{ \frac{(j(x,y)_+)^2}{\rho(x)} + \frac{(j(x,y)_-)^2}{\rho(y)}} \eta(x,y) \di\mu(x) \di\mu(y).
\end{equation}
\end{rem}

In the following lemma we can see that the action takes the form from the tentative definition of the metric in \eqref{eq:ben-bre-formula}, as soon as it is bounded.
\begin{lemma}\label{lem:action}
Let $\mu\in \cM^+(\Rd)$, $\rho\in\mP(\Rd)$ and $\bj\in\cM(G)$ be such that $\cA(\mu;\rho,\bj)<\infty$. Then there exists a measurable $v\colon G\to\R$ such that
\begin{equation}\label{e:bj:v}
\di\bj(x,y)=v(x,y)_+\di\rho(x)\di\mu(y)-v(x,y)_-\di\mu(x)\di\rho(y),
\end{equation}
and it holds
\begin{equation}\label{eq:action:v}
\cA(\mu;\rho,\bj)=\frac12\iint_G\left(|v(x,y)_+|^2+|v(y,x)_-|^2\right)\eta(x,y)\di\rho(x)\di\mu(y).
\end{equation}
In particular, if $v\in \cV^{\mathrm{as}}(G)$, then
\begin{equation}\label{eq:action:vas}
    \cA(\mu;\rho,\bj)=\iint_G|v(x,y)_+|^2\eta(x,y)\di\rho(x)\di\mu(y).
\end{equation}
\end{lemma}
\begin{proof}
  Let $\lambda \in \cM^+(G)$ be such that $\di\gamma_1(x,y) = \di\rho(x) \di\mu(y) = \tilde\gamma_1(x,y) \di\lambda(x,y)$, likewise $\di\gamma_2(x,y) = \di\mu(x) \di\rho(y) = \tilde\gamma_2(x,y) \di\lambda(x,y)$, and $\di\bj = \tilde j \di\lambda$ for some measurable $\tilde\gamma_1,\tilde\gamma_2,\tilde j\colon G\to\R$. Without loss of generality we can assume $\lambda$ to be symmetric; for instance by considering $\tfrac{1}{2} (\lambda + \lambda^\top)$ instead. Thus, \eqref{eq:def-action} implies
  \[
    \cA(\mu;\rho,\bj) = \frac12\iint_G  \bra*{ \alpha\bra[\big]{ \tilde{j}(x,y), \tilde\gamma_1(x,y)} + \alpha\bra[\big]{ -\tilde{j}(x,y), \tilde\gamma_2(x,y)}}\eta(x,y)\di\lambda(x,y) < \infty.
  \]
  By the definition of the function $\alpha$ in~\eqref{eq:def:alpha}, it immediately follows that the vector field $\tilde v^+(x,y) = \frac{\tilde j(x,y)_+}{\tilde\gamma_1(x,y)}$ is well-defined $\gamma_1$-a.e.\ on $G$. By the same argument, we find that $\tilde v^-(x,y) = \frac{\tilde j(x,y )_-}{\tilde\gamma_2(x,y)}$ is well-defined $\gamma_2$-a.e.\ on $G$. Since $\gamma_1=\gamma_2^\top$ we have that $\bra*{\tilde v^-}^\top$ exists $\gamma_1$-a.e.\ on $G$. 
  Hence, we obtain the measurable vector field
  \[
    v(x,y) = \tilde v^+(x,y) - \tilde v^-(x,y). 
  \]
  The statement \eqref{eq:action:v} follows by using the positively one-homogeneity of $\alpha$, the identity $\alpha(j,r)=\alpha(j_+,r)$ and the symmetry of $\lambda$:
  \begin{align*}
     \cA(\mu;\rho,\bj) &= \frac12\iint_G \alpha\bra[\big]{v(x,y)_+ \tilde\gamma_1(x,y),  \tilde\gamma_1(x,y)} \, \eta(x,y) \di\lambda(x,y)\\
     &\qquad +  \frac12\iint_G \alpha\bra[\big]{v(x,y)_- \tilde\gamma_2(x,y) , \tilde\gamma_2(x,y)} \, \eta(x,y)  \di\lambda(x,y) \\
     &=  \frac12\iint_G \abs{ v(x,y)_+}^2 \tilde\gamma_1(x,y) \,  \eta(x,y) \di\lambda(x,y) +  \frac12\iint_G \abs{ v(y,x)_-}^2 \tilde\gamma_1(x,y) \, \eta(x,y) \di\lambda(x,y). \qedhere
  \end{align*}
\end{proof}
\begin{definition}[Nonlocal gradient and divergence] \label{def:nl_grad_div}
For any function $\phi \colon \Rd \to \R$ we define its \emph{nonlocal gradient} $\dgrad \phi \colon G \to \R$ by
\begin{equation} \label{eq:nl_grad}
 \dgrad\phi(x,y)=\phi(y)-\phi(x) \quad \mbox{for all $(x,y)\in G$}.
\end{equation}
For any $\bj\in \cM(G)$, its \emph{nonlocal divergence} $\dgrad\cdot\bj \in \cM(\R^d)$ is defined as $\eta$-weighted adjoint of $\dgrad$, i.e.,
\begin{equation} \label{eq:nl_div}
       \int \phi \dx\dgrad \cdot \bj = - \frac12\iint_G\dgrad\phi(x,y)
\eta(x,y)\di\bj(x,y) = \frac{1}{2}\int \phi(x) \int \eta(x,y) \bra*{ \di\bj(x,y) - \di\bj(y,x)} .
  \end{equation}
In particular, for $\bj\in \cM^{\mathrm{as}}(G) := \set{\bj\in \cM(G): \bj^\top = - \bj}$,
  \begin{equation} \label{eq:nl_divas}
     \int \phi \dx\dgrad\cdot\bj= \iint_G \phi(x) \eta(x,y) \di\bj(x,y) .
  \end{equation}
\end{definition}
If $\bj$ is given by~\eqref{e:bj:v} for some $v\in \cV^{\mathrm{as}}(G)$, then the flux satisfies an antisymmetric relation on the support of $\gamma_1$-a.e.\ on $G$, i.e., $\bj^+=(\bj^\top)^-$ $\gamma_1$-a.e.\ on $G$. The following corollary shows that those antisymmetric fluxes are the relevant ones for the minimization of the action functional. For this reason, the natural class of fluxes are those measure on $G$ which are antisymmetric with positive part absolutely continuous with respect to $\gamma_1$, that is,
\begin{equation}\label{eq:def:Mas}
\cM_{\gamma_1}^{\mathrm{as}}(G)=\{\bj\in\cM(G): \bj^+\ll \gamma_1,\,  \bj^-\ll \gamma_1^\top \text{ and } \bj^+=(\bj^\top)^-\ \gamma_1\text{-a.e.}\}
\end{equation}
\begin{corollary}[Antisymmetric vector fields have lower action]\label{cor:antsym_vect_field_lower_action}
Let $\mu\in\cM^+(\Rd)$, $\rho\in\mP(\Rd)$ and $\bj\in\cM(G)$ be such that $\cA(\mu;\rho,\bj)<\infty$. Then there exists an antisymmetric flux $\bj^{\mathrm{as}}\in\cM_{\gamma_1}^{\mathrm{as}}$ such that
\[
\dgrad\cdot\bj=\dgrad\cdot\bj^{\mathrm{as}},
\]
with lower action:
\[
\cA(\mu;\rho,\bj^{\mathrm{as}})\le\cA(\mu;\rho,\bj).
\]
\end{corollary}
\begin{proof}
  Let us set $\bj^{\mathrm{as}} = \bra{\bj - \bj^\top}/2$. Since $\eta$ is symmetric and $\bra[\big]{\dgrad\phi}^\top = - \dgrad\phi$, we get
  \begin{align*}
    \iint_G \dgrad \phi \; \eta \di{\bj^{\mathrm{as}}} &= \frac12\iint_G \dgrad\phi \; \eta\; \bra{ \di\bj - \di\bj^\top} 
    =  \frac12\iint_G \dgrad\phi \; \eta \di\bj -   \frac12\iint_G \bra[\big]{\dgrad\phi}^\top \; \eta \di\bj =  \iint_G \dgrad\phi \; \eta \di\bj. 
  \end{align*}
By an application of Lemma \ref{lem:action} and comparison of \eqref{eq:action:v} and \eqref{eq:action:vas} it is enough to show that, for all $(x,y)\in G$,
  \begin{align*}
    \MoveEqLeft{\abs*{ v^{\mathrm{as}}(x,y)_+}^2 + \abs*{v^{\mathrm{as}}(x,y)_-}^2 + \abs*{ v^{\mathrm{as}}(y,x)_+}^2 + \abs*{v^{\mathrm{as}}(y,x)_-}^2} \\
    &\leq \abs*{ v(x,y)_+}^2 + \abs*{v(x,y)_-}^2 + \abs*{ v(y,x)_+}^2 + \abs*{v(y,x)_-}^2
  \end{align*}
  for any measurable $v\colon G\to \R$, where $v^{\mathrm{as}}(x,y) = \bra*{ v(x,y) -v(y,x)}/2$. This estimate is a consequence of Jensen's inequality applied to the convex functions
  \[ 
    \varphi^\pm \colon \R \to \R \qquad\text{with}\qquad \varphi^\pm(r) = \bra*{ r_\pm}^2. \qedhere
  \]
\end{proof}

\begin{lemma}[Lower semicontinuity of the action]\label{lem:l.s.c.action}
The action is lower semicontinuous with respect to the narrow convergence in $\cM^+(\Rd)\times \cP(\Rd)\times \cM(G)$. That is, if $\mu^n {\rightharpoonup}\mu$ in $\cM(\Rd)$, $\rho^n {\rightharpoonup}\rho$ in $\mP(\Rd)$, and $\bj^n {\rightharpoonup}\bj$ in $\cM(G)$, then
\[
\liminf_{n\to\infty}\cA(\mu^n;\rho^n,\bj^n)\ge\cA(\mu;\rho,\bj) \;.
\]
\end{lemma}
\begin{proof}
First, note that the narrow convergence of any sequences $(\rho^n)_n$ and $(\mu^n)_n$ implies the narrow convergence of the product: $\rho^n\otimes\mu^n \rightharpoonup \rho\otimes \mu$ in $\cP(\Rd)\times \cM^+(\Rd)$, therefore also in $\cM^+(G)$. Then, in Definition \ref{def:action} consider the vector-valued measure
\[
  \lambda = \bra*{ \bj , \rho \otimes \mu,\mu \otimes \rho}.
 \]
Further, we define the function
\[
    f\colon G \times \R^3 \to \R \quad\text{by} \quad f\bra[\big]{(x,y),(j,\gamma_1,\gamma_2)} = \bra[\big]{ \alpha(j,\gamma_1) + \alpha(-j,\gamma_2)} \, \eta(x,y).
\]
Since the function $\eta$ is lower semicontinuous by~\eqref{it:as:pos-sym-lsc} and $\alpha$ defined in \eqref{eq:def:alpha} is lower semicontinuous, jointly convex and positively one-homogeneous, $f$ satisfies the assumptions of \cite[Theorem 3.4.3]{But89}, whence the claim follows.
\end{proof}
According to Definition \ref{def:action}, fluxes and action are strictly related. In case $\cA(\mu;\rho,\bj)<+ \infty$, we get a useful upper bound in the following lemma that will be crucial in several technical parts later on.
\begin{lemma}\label{lem:A:TV:bound}
For any $\mu\in\cM^+(\Rd)$, $\rho\in\mP(\Rd)$, $\bj\in\cM(G)$ and any measurable $\Phi:G\to\R_+$ it holds
\begin{equation}\label{eq:control-phi}
\left(\frac12\iint_G\Phi\, \eta\di|\bj|\right)^2\le \, \cA(\mu;\rho,\bj) \iint_G\Phi^2\, \eta\,(\di\gamma_1+\di\gamma_2) .
\end{equation}
\end{lemma}
\begin{proof}
Let $\mu\in\cM^+(\Rd)$, $\rho\in \mP(\Rd)$ and $\bj\in \cM(G)$ be such that $\cA(\mu;\rho,\bj)<+ \infty$. Let $|\lambda| \in \cM^+(G)$ be such that $\gamma_1, \gamma_2, |\bj|\ll |\lambda|$ as in Definition \ref{def:action} and write $\gamma_i = \tilde\gamma_i |\lambda|$ and $|\bj| = |j| |\lambda|$ for the densities.

We have that $A:=\set[\big]{(x,y) \in G\colon \alpha(j,\tilde\gamma_1) = \infty \text{ or } \alpha(-j,\tilde\gamma_2)=\infty}$ is a $\lambda$-nullset. We observe the elementary inequality
\[
(j_+)^2 + (j_-)^2 \leq {\max\set{\tilde\gamma_1,\tilde\gamma_2}} \bra[\big]{\alpha(j,\tilde\gamma_1) + \alpha(-j,\tilde\gamma_2)}, \qquad\text{$\lambda$-a.e.\ in $A^\mathrm{c}$}. 
\]
In particular, it holds
  \[
   \abs{j} = j_+ + j_- \leq \sqrt{2 \max\set{\tilde\gamma_1,\tilde\gamma_2}} \sqrt{\alpha(j,\tilde\gamma_1) + \alpha(-j,\tilde\gamma_2)},  \qquad\text{$\lambda$-a.e.\ in $A^\mathrm{c}$}. 
  \]
Hence we can estimate
  \begin{align*}
     \frac12\iint_G \Phi\,\eta \di\abs{\bj} &= \frac12\iint_G \Phi\,\eta\,  \abs{j} \di|\lambda| = \frac12\iint_{A^\mathrm{c}} \Phi\,\eta\, \bra*{ j_+ + j_-} \di{|\lambda|} \\
    &\leq \frac12\iint_{A^\mathrm{c}} \Phi\,\eta\, \sqrt{2\max\set*{ \tilde\gamma_1,\tilde\gamma_2}} \sqrt{\alpha(j,\tilde\gamma_1) + \alpha(-j,\tilde\gamma_2)} \di|\lambda| \\
    &\leq \bra*{ \iint_G \Phi^2\,\eta\, \max\set*{ \tilde\gamma_1,\tilde\gamma_2}  \di|\lambda|}^{\frac12} \bra*{\frac12\iint_G \bra*{ \alpha(j,\tilde\gamma_1) + \alpha(-j,\tilde\gamma_2)}\,\eta  \di|\lambda|}^{\frac12}.
  \end{align*}
Now, the result follows by estimating $\max\set*{\tilde\gamma_1,\tilde\gamma_2} \leq \tilde\gamma_1 + \tilde\gamma_2$.
\end{proof}
As a consequence of the previous results we have the following corollary, which will be useful in Section \ref{subsec:nonloc-cont-eq}. 
\begin{corollary}\label{cor:boundwithA}
Let $\mu\in \cM^+(\R^d)$ satisfy~\ref{it:as-conv} for some $C_\eta\in(0,\infty)$, then for all $\rho\in\cP(\Rd)$ and $\bj\in\cM(G)$ there holds
\begin{equation}\label{eq:boundwithA}
  \frac12\iint_G(2\wedge|x-y|)\eta(x,y)\di\abs{\bj}(x,y)\le \sqrt{2C_\eta\, \cA(\mu;\rho,\bj)}.
\end{equation}
\end{corollary}
\begin{proof}
Let us consider the case $\cA(\mu;\rho,\bj)<\infty$, otherwise the result is trivial. From Lemma \ref{lem:action} we have $\di\bj(x,y)=v(x,y)_+\di\gamma_1(x,y)-v(x,y)_-\di\gamma_2(x,y)$, with $\di\gamma_1(x,y)=\di\rho(x)\mu(y)$ and $\di\gamma_2(x,y)=\di\mu(x)\di\rho(y)$. Applying Lemma \ref{lem:A:TV:bound} for $\Phi(x,y)=2\wedge|x-y|$ and noticing $\Phi(x,y) \leq |x-y| \leq |x-y|\vee |x-y|^2$, we arrive at the bound
\begin{align*}
\left(\frac12\iint_G(2\wedge|x-y|)\eta(x,y)\di\bj\right)^2&\le \cA(\mu;\rho,\bj) \iint_G(2\wedge|x-y|)^2\eta(x,y)(\di\gamma_1+\di\gamma_2)\\
&\le\cA(\mu;\rho,\bj)\, 2 \iint_G \bra*{\abs{x-y}^2 \vee \abs{x-y}^4} \, \eta(x,y)\di\mu(y) \di\rho(x) \\
&\le\cA(\mu;\rho,\bj) \, 2 C_\eta,
\end{align*}
where the last estimate follows from~\ref{it:as-conv} and the integral is finite since $\rho\in \cP(\Rd)$.
\end{proof}
\begin{lemma}[Convexity of the action]\label{lem:convexity}
  Let $\mu^i\in\cM^+(\Rd)$, $\rho^i \in \mP(\Rd)$ and $\bj^i \in \cM(G)$ for $i=0,1$. For $\tau \in (0,1)$ such that $\mu^\tau = (1-\tau) \mu^0 + \tau \mu^1$, $\rho^\tau = (1-\tau) \rho^0 + \tau \rho^1$ and $\bj^\tau = (1-\tau) \bj^0 + \tau \bj^1$, it holds
  \begin{equation*}
   \cA(\mu^\tau;\rho^\tau, \bj^\tau)\le (1-\tau) \cA(\mu^0;\rho^0,\bj^0) + \tau \cA(\mu^1;\rho^1,\bj^1).
  \end{equation*}
\end{lemma}
\begin{proof}
Let us consider a measure $\lambda\in\cM(G)$ such that $\di\gamma_j^i=\tilde\gamma_j^i\di\lambda$ and $\di\bj^i=\tilde{\bj}^i\di\lambda$ for $i=0,1$ and $j=1,2$. Then, the convex combinations are such that $\di\gamma_j^\tau=\tilde\gamma_j^\tau\di\lambda$ and $\di\bj^{\tau}=\tilde{\bj}^{\tau}\di\lambda$, where
\begin{align*}
    &\tilde{\gamma}_j^\tau=(1-\tau)\tilde{\gamma}_j^0 + \tau\tilde{\gamma}_j^1, \qquad \text{for } j=1,2,\\
  \text{and}\qquad  &\tilde{\bj}^{\tau}=(1-\tau)\tilde{\bj}^0 + \tau\tilde{\bj}^1.
\end{align*}
Using the convexity of the function $\alpha$ we get the result, that is,
\begin{align*}
    \cA(\mu^\tau;\rho^\tau,\bj^\tau)&=\frac12\iint_G\left(\alpha(\tilde{\bj}^\tau,\tilde{\gamma}_1^\tau)+\alpha(-\tilde{\bj}^\tau,\tilde{\gamma}_2^\tau)\right)\eta(x,y)\di\lambda(x,y)\\
    &\le \frac{1-\tau}{2}\iint_G\left(\alpha(\tilde{\bj}^0,\tilde{\gamma}_1^0)+\alpha(-\tilde{\bj}^0,\tilde{\gamma}_2^0)\right)\eta(x,y)\di\lambda(x,y)\\
    &\phantom{\le} +\frac{\tau}{2}\iint_G\left(\alpha(\tilde{\bj}^1,\tilde{\gamma}_1^1)+\alpha(-\tilde{\bj}^1,\tilde{\gamma}_2^1)\right)\eta(x,y)\di\lambda(x,y)\\
    &=(1-\tau)\cA(\mu^0;\rho^0,\bj^0) + \tau\cA(\mu^1;\rho^1,\bj^1). \qedhere
\end{align*}
\end{proof}
\subsection{Nonlocal continuity equation}\label{subsec:nonloc-cont-eq}
In view of the considerations made in Section \ref{subsec:action}, we now deal with the nonlocal continuity equation 
\begin{equation}\label{eq:nlce_measures}
    \partial_t\rho_t+\dgrad\cdot\bj_t=0 \qquad \text{on}\ (0,T)\times\Rd,
\end{equation}
where $(\rho_t)_{t\in[0,T]}$ and $(\bj_t)_{t\in[0,T]}$ are unknown Borel families of measures in $\mP(\Rd)$ and $\cM(G)$, respectively. Equation \eqref{eq:nlce_measures} is understood in the weak form: $\forall\varphi\in C_\mathrm{c}^\infty((0,T)\times\Rd)$,
\begin{equation}\label{eq:nce-weak}
\int_0^T\int_\Rd\partial_t\varphi_t(x)\di\rho_t(x)\di t +\frac12\int_0^T\iint_G\dgrad\varphi_t(x,y)\eta(x,y)\di\bj_t(x,y)\di t=0.
\end{equation}
Since $|\dgrad\varphi(x,y)|\le||\varphi||_{C^1}(2\wedge|x-y|)$, the weak formulation is well-defined under the integrability condition
\begin{equation}\label{eq:integrability-cond}
\int_0^T\iint_G(2\wedge|x-y|)\eta(x,y)\di\bj_t(x,y)\di t<\infty .
\end{equation}
\begin{rem}
The integrability condition~\eqref{eq:integrability-cond} is automatically satisfied by a pair $(\rho_t, \bj_t)_{t\in [0,T]}$ such that $\int_0^T \cA(\mu;\rho_t,\bj_t) \di t< \infty$, due to Corollary \ref{cor:boundwithA}.
\end{rem}
Hence we arrive at the following definition of weak solution of the nonlocal continuity equation:
\begin{definition}[Nonlocal continuity equation in flux form]\label{def:nce-flux-form}
A pair $(\rho,\bj)\colon [0,T] \to \cP(\Rd)\times \cM(G)$
is called a \emph{weak solution} to the nonlocal continuity equation \eqref{eq:nlce_measures} provided that\begin{listi}
\item 
$(\rho_t)_{t\in[0,T]}$ is weakly continuous curve in $\mP(\Rd)$;
\item 
$(\bj_t)_{t\in[0,T]}$ is a Borel-measurable curve in $\cM(G)$;
\item 
the pair $(\rho,\bj)$ satisfies \eqref{eq:nce-weak}.
\end{listi}
We denote the set of all weak solutions on the time interval $[0,T]$ by $ \CE_T$.
For $\rho^0,\rho^1\in\mP(\Rd)$, a pair $(\rho,\bj)\in \CE(\rho^0,\rho^1)$ if $(\rho,\bj)\in\CE:=\CE_1$ and in addition $\rho(0)=\rho^0$ and $\rho(1)=\rho^1$.
\end{definition}
The following lemma shows that any weak solution satisfying~\eqref{eq:nce-weak}, which additionally satisfies the integrability condition~\eqref{eq:integrability-cond} has a weakly continuous representative and hence is a weak solution in the sense of Definition~\ref{def:nce-flux-form}. This observation justifies the terminology of curve in the space of probability measures; see \cite[Lemma 8.1.2]{AGS} and \cite[Lemma 3.1]{Erb14}.
\begin{lemma}\label{lem:nce:weak_cont}
Let $(\rho_t)_{t\in[0,T]}$ and $(\bj_t)_{t\in[0,T]}$ be Borel families of measures in $\mP(\Rd)$ and $\cM(G)$ satisfying \eqref{eq:nce-weak} and \eqref{eq:integrability-cond}. Then there exists a weakly continuous curve $(\bar{\rho}_t)_{t\in[0,T]}\subset\mP(\Rd)$ such that $\bar{\rho}_t=\rho_t$ for a.e.\ $t\in[0,T]$. Moreover, for any $\varphi\in C_\mathrm{c}^\infty([0,T]\times\Rd)$ and all $0\le t_0\le t_1\le T$ it holds
\begin{equation}\label{eq:CE:weak:t01}
\begin{split}
\int_\Rd\varphi_{t_1}(x)\di\bar{\rho}_{t_1}(x)-\int_\Rd\varphi_{t_0}(x)\di\bar{\rho}_{t_0}(x)&=\int_{t_0}^{t_1}\int_\Rd\partial_t\varphi_t(x)\di\rho_t(x)\di t\\
&\quad +\frac12\int_{t_0}^{t_1}\iint_G\dgrad\varphi_t(x,y)\eta(x,y)\di\bj_t(x,y)\di t.
\end{split}
\end{equation}
\end{lemma}
We now prove propagation of second-order moments.
\begin{lemma}[Uniformly bounded second moments]\label{lem:CE:tightness}
Let $(\mu^n)_n\subset\cM^+(\Rd)$ such that~\ref{it:as-conv} holds uniformly in $n$. Let $(\rho_0^n)_n \subset \cP_{2}(\Rd)$ be such that $\sup_{n\in\N} M_2(\rho_0^n) < \infty$ and $(\rho^n,\bj^n)_n \subset \CE_T$ be such that $\sup_{n\in\N} \int_0^T \cA(\mu^n;\rho_t^n,\bj_t^n)\di t<\infty$. Then $\sup_{t\in [0,T]}\sup_{n\in\N} M_2(\rho_t^n) < \infty$.
\end{lemma}
\begin{proof}
We proceed by considering the time derivative of the second-order moment of $\rho_t^n$ for all $t\in[0,T]$ and $n\in\N$. Since $x\mapsto |x|^2$ is not an admissible test function in~\eqref{eq:nce-weak}, we introduce a smooth cut-off function $\varphi_R$ satisfying $\varphi_R(x)=1$ for $x\in B_R$, $\varphi_R(x)=0$ for $x \in \Rd \setminus B_{2R}$ and $\abs{\nabla \varphi_R}\leq \frac{2}{R}$. Then, we can use the definition of solution with the function $\psi_R(x)= \varphi_R(x)^2 (|x|^2+1)$ and apply Lemma \ref{lem:A:TV:bound} with $\Phi=\dgrad\psi_R$ to obtain, for all $t\in[0,T]$ and $n\in\N$,
\begin{align*}
    \frac{\di}{\di{t}}\int_\Rd \psi_R(x)\di\rho_t^n(x)
    &=\frac12\iint_G \dgrad \psi_R(x,y)\,\eta(x,y)\di\bj_t^n(x,y)\\
    &\le\sqrt{\cA(\mu^n;\rho_t^n,\bj_t^n)}\left(\iint_G \abs*{\dgrad\psi_R(x,y)}^2\eta(x,y)(\di\gamma_1^n+\di\gamma_2^{n})\right)^\frac{1}{2}.
\end{align*}
For $R\geq 1$, we estimate, for all $(x,y)\in G$,
\begin{align}\label{eq:2nd:p0}
  \abs{\dgrad \psi_R(x,y)}^2 &\leq 2 \abs{\varphi_R(y)^2-\varphi_R(x)^2}^2 + 2 \abs{ \varphi_R(y)^2 |y|^2 - \varphi_R(x)^2 |x|^2}^2 ,
\end{align}
and observe that
\[
  \abs*{ \dgrad \varphi_R^2(x,y) } = \abs*{\dgrad \varphi_R(x,y)\bra*{\varphi_R(x) + \varphi_R(y)}} \leq \frac{4}{R} \abs*{x-y}. 
\]
Hence the first term in~\eqref{eq:2nd:p0} is bounded by $32 \abs{x-y}^2$, since $R\ge1$. For the second term in~\eqref{eq:2nd:p0}, we abbreviate by setting $r = \varphi_R(x) \abs{x}$ and $s = \varphi_R(y)\abs{y}$ and compute the bound
\begin{align*}
 \abs{s^2 - r^2}^2 =  \abs{s-r}^2 \abs{s+r}^2 \leq 2 \abs{s-r}^4 + 8 |r|^2 \abs{s-r}^2 \leq 8 \bra*{|r|^2+1} \bra*{ |s-r|^2 \vee |s-r|^4 }. 
\end{align*}
It is easy to check that $x\mapsto \varphi_R(x) \abs{x}$ is globally Lipschitz and we can conclude that, for some numerical constant $C>0$, for all $(x,y)\in G$ we have
\[
  \abs*{ \dgrad \psi_R(x,y) }^2\leq 32\abs{x-y}^2 + C \abs{x}^2 \bra*{ \abs{x-y}^2 \vee \abs{x-y}^4} \leq C\bra*{\abs{x}^2+1}\bra*{ \abs{x-y}^2 \vee \abs{x-y}^4}. 
\]
Thus, by sending $R\to\infty$ and using~\ref{it:as-conv}, it follows that
\begin{align*}
    \frac{\di}{\di{t}}\int_\Rd \bra*{|x|^2+1}\di\rho_t^n(x) &\le \sqrt{\cA(\mu^n;\rho_t^n,\bj_t^n)}\left(2 C C_\eta\int_\Rd \bra*{|x|^2+1}\di\rho^n_t(x)\right)^\frac{1}{2}
\end{align*}
By integrating the above differential inequality, we arrive at the bound
\[
 \int_\Rd |x|^2\di\rho_t^n(x)\le 2 \int_\Rd \bra*{|x|^2+1}\di\rho_0^n(x)+ 2 C C_\eta T \int_0^T \cA(\mu^n;\rho_t^n,\bj_t^n)\di t,
\]
whence we conclude by taking the suprema in $n\in \N$ and $t\in [0,T]$.
\end{proof}
Now we are ready to show compactness for the solutions to \eqref{eq:nlce_measures}.
\begin{proposition}[Compactness of solutions to the nonlocal continuity equation]\label{prop:compactness-sol-ce}
  Let $(\mu^n)_n\subset\cM^+(\R^d)$ and suppose that $(\mu^n)_n$ narrowly converges to $\mu$. Moreover, suppose that the base measures $\mu^n$ and $\mu$ satisfy~\ref{it:as-conv} and~\ref{it:as-tight} uniformly in $n$. Let $(\rho^n,\bj^n) \in \CE_T$ for each $n\in\mathbb{N}$ be such that $(\rho_0^n)_n$ satisfies $\sup_{n\in\N} M_2(\rho_0^n)< \infty$ and
\begin{equation}\label{eq:uniform-integ-A}
  \sup_{n\in\N}\int_0^T\cA(\mu^n;\rho_t^n,\bj_t^n)\di t<\infty.
\end{equation}
Then, there exists $(\rho,\bj)\in \CE_T$ such that, up to a subsequence, as $n\to\infty$ it holds
\begin{align*}
  \rho_t^n\rightharpoonup\rho_t\quad &\text{for all}\ t\in[0,T],\\
  \bj^n\rightharpoonup\bj\quad\ &\text{in}\ \cM_{\mathrm{loc}}(G\times[0,T]),
\end{align*}
with $\rho_t\in\mP_2(\Rd)$ for any $t\in[0,T]$. Moreover, the action is lower semicontinuous along the above subsequences $(\mu^n)_n, (\rho^n)_n$ and $(\bj^n)_n$, i.e.,
\begin{equation*}
\liminf_{n\to\infty}\int_0^T\cA(\mu^n;\rho_t^n,\bj_t^n)\di t\ge\int_0^T\cA(\mu;\rho_t,\bj_t)\di t.
\end{equation*}
\end{proposition}
\begin{proof}
We argue similarly to \cite[Lemma 4.5]{DNS09}, \cite[Proposition 3.4]{Erb14}. For each $n\in\mathbb{N}$ we define $\bj^n\in\cM(G\times[0,T])$ as $\di\bj^n(x,y,t)=\di\bj_t^n(x,y)\di t$. In view of Lemma~\ref{lem:CE:tightness} there exists $C_2>0$ such that $\sup_{t\in [0,T]}\sup_{n\in \N}  M_2(\rho_t^n) \leq C_2 <+ \infty$.

For any compact sets $K\subset G$ and $I\subseteq[0,T]$, we apply the bound~\eqref{eq:boundwithA} of Corollary~\ref{cor:boundwithA} and the Cauchy--Schwarz inequality to get
\begin{align}
\sup_{n\in\N}|\bj^n|(K\times I)&\le\sup_{n\in\N}\int_I \iint_K \frac{(2\wedge |x-y|)\,\eta(x,y)}{(2\wedge |x-y|)\,\eta(x,y)} \di|\bj_t^n|(x,y)\di t \notag \\
&\le \frac{2 \sqrt{|I|} \sqrt{2C_\eta}}{\inf_{(x,y)\in K} (2\wedge |x-y|)\eta(x,y)}\left(\sup_{n\in\N}\int_0^T\cA(\mu^n;\rho_t^n,\bj_t^n)\di t\right)^\frac{1}{2} . \label{eq:compactness_bound}
\end{align}
Thanks to Assumption~\eqref{it:as:pos-sym-lsc}, we have that $\inf_{(x,y)\in K}  (2\wedge |x-y|)\eta(x,y)>0$ for any compact $K\subset G$. Hence, by \eqref{eq:uniform-integ-A}, $(\bj^n)_n$ has total variation uniformly bounded in $n$ on every compact set of $G\times[0,T]$, which implies, up to a subsequence, $\bj^n\rightharpoonup\bj$ as $n\to\infty$ in $\cM_{\loc}(G \times[0,T])$. Because of the disintegration theorem, there exists a Borel family $(\bj_t)_{t\in[0,T]}$ such that, for all compact sets $I\subseteq[0,T]$ and $K\subset G$, there holds that $\bj(K\times I)=\int_I \bj_t(K) \di t$.
Thanks to the bound~\eqref{eq:compactness_bound}, the family $\{\bj_t\}_{t\in[0,T]}$ still satisfies \eqref{eq:integrability-cond}.

Now, as we need to pass to the limit in \eqref{eq:nce-weak}, we consider a function $\xi\in C_\mathrm{c}^\infty(\Rd)$ and an interval $[t_0,t_1]\subseteq [0,T]$. The function $\chi_{[t_0,t_1]}(t)\dgrad\xi(x,y)$ has no compact support in $[t_0,t_1]\times G$, so we proceed by a truncation argument. Let $\eps>0$ and let us set $I^\eps = [t_0+\eps , t_1-\eps]$, $N_\eps = \overline{B}_{\eps^{-1}} \times \overline{B}_{\eps^{-1}}$, where $B_{\eps^{-1}}= \set*{x \in \R^d: |x|< \eps^{-1}}$, and $G_\eps=\{(x,y)\in G:\eps\le|x-y|\}$. Hence we can find $\varphi_\eps\in C_\mathrm{c}^\infty([t_0,t_1]\times G; [0,1])$ satisfying
\begin{equation}\label{eq:truncation:CE:compact}
  \set*{ \varphi_\eps = 1 } \supseteq I_\eps \times \bra*{ G_\eps \cap N_\eps},
\end{equation}
so that $\varphi_\eps \to \chi_{[t_0,t_1]} \, \chi_G$ as $\eps \to 0$ and $\varphi_\eps \, \chi_{[t_0,t_1]} \, \dgrad \xi$ has compact support in $[t_0,t_1]\times G$. Then, we get thanks to Assumption~\eqref{it:as:pos-sym-lsc}, that
\begin{align}
  \lim_{n\to\infty}\int_{t_0}^{t_1} \iint_G \varphi_\eps(t,x,y)\dgrad\xi(x,y)  \eta(x,y)  \di\bj_t^n(x,y)\di t &= \int_{t_0}^{t_1} \iint_G \varphi_\eps(t,x,y)\dgrad\xi(x,y) \eta(x,y)\di\bj_t(x,y)\di t .
  \label{eq:limn}
\end{align}
Now, it remains to show that 
\begin{equation}\label{eq:limeps}
  \lim_{\eps\to 0}  \sup_{n\in\N} \abs*{\int_{t_0}^{t_1} \iint_G  \bra*{ 1- \varphi_\eps(t,x,y)} \dgrad\xi(x,y)\eta(x,y)\di\bj_t^n(x,y)\di t} = 0.
\end{equation}
We need to estimate terms for which $\varphi_\varepsilon(t,x)<1$. First, setting $I_\eps^\mathrm{c} = [t_0,t_1]\setminus I_\eps$, we note that
\[
  [t_0,t_1] \times G \setminus \set{\varphi_\eps =1} \subseteq \bra[\big]{ I_\eps^\mathrm{c} \times G } \cup \bra[\big]{ I_\eps \times ( G\setminus (G_\eps \cap N_\eps))} =:  M_\eps,
\]
whence, by Lemma \ref{lem:A:TV:bound},
\begin{align*}
 \MoveEqLeft{\abs*{\int_{t_0}^{t_1} \iint_G  \bra*{ 1- \varphi_\eps(t,x,y)} \dgrad\xi(x,y)\eta(x,y)\di\bj_t^n(x,y)\di t}}\\
 &\leq \norm{\xi}_{C^1}  \int_{t_0}^{t_1} \iint_G  \bra*{ 1- \varphi_\eps(t,x,y)} \bra*{ 2 \wedge | x-y|}\eta(x,y) \di\abs{\bj_t^n}(x,y)\di t \\ 
 &\leq 2\norm{\xi}_{C^1} \bra[\bigg]{\int_0^T\cA(\mu^n;\rho_t^n,\bj_t^n)\di t}^\frac{1}{2} \left(\iiint_{M_\eps} \bra*{ 4\wedge | x-y|^2}\eta(x,y)\di\bra[\big]{\gamma^{n}_{1,t} + \gamma^{n}_{2,t}} \di t\right)^\frac{1}{2}.
\end{align*}
Since $4\wedge |x-y|^2 \leq |x-y|^2\vee |x-y|^4$ we have, by Assumption~\ref{it:as-conv}, the bound
\[
  \int_{I_\eps^\mathrm{c}}\iint_G  \bra*{ 4\wedge | x-y|^2}\eta(x,y) \di\bra[\big]{\gamma^{1,n}_t + \gamma^{2,n}_t} \di t \leq  2 |I_\eps^\mathrm{c}| C_\eta = 4 C_\eta \eps.
\]
Likewise, using the symmetry, we arrive at
\[
  \int_{I_\eps} \iint_{G_\eps^\mathrm{c}} \bra*{ 4\wedge | x-y|^2}\eta(x,y) \di\bra[\big]{\gamma^{n}_{1,t} + \gamma^{n}_{2,t}} \di t  = 2\int_0^T \iint_{G_\eps^\mathrm{c}} \bra*{ 4\wedge | x-y|^2} \eta(x,y) \di\mu^n(y) \di\rho_t^n(x)\di{t},
\]
which vanishes as $\eps\to 0$ in view of Assumption~\ref{it:as-tight}. Finally, the last term is estimated again using~\ref{it:as-conv}: 
\begin{align*}
  \int_{I_\eps}\iint_{G\setminus N_\eps} \bra*{ 4\wedge | x-y|^2}\eta(x,y)\di\gamma^{1,n}_t \di t &\leq \int_0^T \int_{\overline{B}_{\eps^{-1}}^\mathrm{c}} \int_\Rd \bra*{4\wedge |x-y|^2} \eta(x,y) \di\mu^n(y) \di\rho_t^n(x) \di t  \\
 &\leq T C_\eta \sup_{t\in [0,T]} \rho_t^n\bra*{\overline{B}_{\eps^{-1}}^\mathrm{c}} \to 0 \qquad\text{as } \eps \to 0 , 
\end{align*}
since $M_2(\rho_t^n) \leq C_2$ for any $n\in\N$ and $t\in [0,T]$ by Lemma~\ref{lem:CE:tightness}.

Combining \eqref{eq:limn} and \eqref{eq:limeps}, we get
\begin{align*}
 \lim_{n\to\infty}\int_{t_0}^{t_1}\iint_G\dgrad\xi(x,y)\,\eta(x,y)\di\bj_t^n(x,y)\di t
 &=\int_{t_0}^{t_1}\iint_G\dgrad\xi(x,y)\,\eta(x,y)\di\bj_t(x,y)\di t.
\end{align*}
By means of the last
convergence, the tightness of $(\rho_0^n)_n$, and \eqref{eq:CE:weak:t01} with $\varphi(t,x)=\xi(x)$, $t_0=0$ and $t_1=T$, we obtain that $(\rho_t^n)_n$ locally narrowly converges to some finite non-negative measure $\rho_t\in\cM^+(\Rd)$ for any $t\in[0,T]$. In particular, for any $\xi\in C_\mathrm{c}^\infty(\Rd)$ and any $t\in[0,T]$, we have
\[
\int_\Rd\xi(x)\di\rho_t(x)=\int_\Rd\xi(x)\di\rho_0(x)+\frac12\int_{0}^{t}\iint_G\dgrad\xi(x,y)\eta(x,y) \di\bj_s(x,y)\di s.
\]
Now, for $R>0$, let us consider a function $\xi_R\in C_\mathrm{c}^\infty(\Rd)$ such that $0\le\xi\le1$, $\xi=1$ on $B_R$, and $\|\xi\|_{C^1}\le1$. Because of the integrability condition~\eqref{eq:integrability-cond}, satisfied thanks to Corollary~\ref{cor:boundwithA}, we have
\[
\left|\int_{0}^{t}\frac12\iint_G\dgrad\xi_R(x,y)\,\eta(x,y) \di\bj_s(x,y)\di s\right|\le\frac12\int_0^t\iint_{G\setminus(B_R\times B_R)}\left(2\wedge|x-y|\right)\eta(x,y) \di|\bj_s|\di s \xrightarrow[R\to\infty]{}0.
\]
Hence the measure $\rho_t$ is actually a probability measure on $\Rd$ for all $t\in[0,T]$. Moreover Lemma~\ref{lem:CE:tightness} ensures that the convergence is global and not only local. As a direct consequence of the previous considerations, $(\rho,\bj)\in\CE_T$ and the lower semicontinuity follows from Lemma \ref{lem:l.s.c.action}.
\end{proof}
\subsection{Nonlocal upwind transportation quasi-metric} \label{sec:nl-trans}
Here, we give a rigorous definition of the nonlocal transportation quasi-metric we introduced in \eqref{eq:ben-bre-formula}. Let us recall that $\eta \colon \{ (x,y)\in \R^d \times \R^d : x\neq y \}\to [0,\infty)$ is the weight function satisfying \eqref{it:as:pos-sym-lsc}. 
\begin{definition}[Nonlocal upwind transportation cost]\label{defn:metric}
    For $\mu\in \cM^+(\Rd)$ satisfying Assumptions~\ref{it:as-conv} and~\ref{it:as-tight}, and $\rho_0,\rho_1\in\mP_2(\Rd)$, the \emph{nonlocal upwind transportation cost} between $\rho_0$ and $\rho_1$ is defined by
\begin{equation}\label{eq:nonloc-upwind-transp-cost}
\mathcal{T}_\mu(\rho_0,\rho_1)^2=\inf\left\{\int_0^1\cA(\mu;\rho_t,\bj_t)\di t:(\rho,\bj)\in\CE(\rho_0,\rho_1)\right\}.
\end{equation}
If $\mu$ is clear from the context, the notation $\cT$ is used in place of $\cT_\mu$.
\end{definition}
Note that Proposition \ref{prop:compactness-sol-ce} ensures the existence of minimizers to \eqref{eq:nonloc-upwind-transp-cost}, when $\cT_\mu<\infty$, which holds when there exists a path of finite action. On the other hand, if this is not the case, the nonlocal upwind transportation cost is infinite. For example, consider the graph with vertices set by $\mu$ and  $\eta$ which is disconnected, meaning that there are $x,y\in \supp \mu$ such that there is no sequence $(x_0=x,x_1,\dots,x_{n-1},x_n=y)_n$ with $\eta(x_i,x_{i+1})>0$ for all $i=0,\dots,n-1$; in this case, $\cT_\mu(\delta_x,\delta_y)=\infty$ since the set of solutions to the continuity equation $\CE(\delta_x,\delta_y)$ is empty.

Due to the one-homogeneity of the action density function $\alpha$ in \eqref{eq:def:alpha}, we have the following reparametrization result, which is similar to \cite[Theorem 5.4]{DNS09}.
\begin{lemma}[Reparametrization]\label{lem:reparametrization}
For any $\mu \in \cM^+(\R^d)$ satisfying Assumptions~\ref{it:as-conv} and~\ref{it:as-tight}, and any $\rho_0,\rho_T\in\mP_2(\Rd)$, it holds that
\[
\mathcal{T}_\mu(\rho_0,\rho_T)=\inf\left\{\int_0^T\sqrt[]{\cA(\mu; \rho_t,\bj_t)}\di t:(\rho,\bj)\in\CE_T(\rho_0,\rho_T)\right\}.
\]
\end{lemma}
Now, as consequence of the above reparametrization and Jensen's inequality, we have the following result, which implies that the infimum is in fact a minimum; see \cite[Proposition 4.3]{Erb14}.
\begin{proposition}\label{prop:min-metric}
For any $\mu \in \cM^+(\R^d)$ satisfying Assumptions~\ref{it:as-conv} and~\ref{it:as-tight}, and any $\rho_0,\rho_1\in\mP_2(\Rd)$ such that $\mathcal{T}_\mu(\rho_0,\rho_1)<\infty$, the infimum in \eqref{eq:nonloc-upwind-transp-cost} is attained by a curve $(\rho,\bj)\in\CE(\rho_0,\rho_1)$ so that $\cA(\rho_t,\bj_t)=\mathcal{T}_\mu(\rho_0,\rho_1)^2$ for a.e.\ $t\in[0,1]$. Such curve is a constant-speed geodesic for $\cT_\mu$, i.e.,
\[
\cT_\mu(\rho_s,\rho_t)=|t-s|\cT_\mu(\rho_0,\rho_1), \quad \mbox{for all $s,t\in[0,1]$}.
\]
\end{proposition}
The next proposition establishes a link between $\mathcal{T}_\mu$ and the $W_1$-distance. 
\begin{proposition}[Comparison with $W_1$]\label{prop:comparison_with_W_1}
Let $\mu\in \cM^+(\R^d)$ satisfy~\ref{it:as-conv} for some $C_\eta>0$ (depending only on $\mu$ and $\eta$). Then for any $\rho^0,\rho^1\in\mP_2(\Rd)$ it holds
\[
W_1(\rho^0,\rho^1)\le \sqrt{2C_\eta}\, \sqrt{\mathcal{T}(\rho^0,\rho^1)}.
\]
\end{proposition}
\begin{proof}
  By a standard regularization argument and the truncation procedure as in the proof of Lemma~\ref{lem:CE:tightness}, we can actually consider any $1$-Lipschitz function $\psi$ as a test function in the weak formulation~\eqref{eq:nce-weak} for some $(\rho,\bj)\in \CE(\rho^0,\rho^1)$. Then we can estimate, by Lemma~\ref{lem:A:TV:bound} and Assumption~\ref{it:as-conv},
  \begin{align*}
    \abs*{ \int_\Rd \psi \di\rho^1 - \int_\Rd \psi\di\rho^0} &= \abs*{ \frac12\int_0^1 \iint_G \dgrad \psi\, \eta \di\bj_t \di t}
    \leq \frac12\int_0^1 \iint_G \abs{x-y} \, \eta(x,y) \di\abs{\bj_t}(x,y) \di t \\
    &\leq \bra*{ \int_0^1 \cA(\rho_t, \bj_t) \di t}^{\frac{1}{2}} \bra*{ \int_0^1\iint_G \abs{x-y}^2 \eta(x,y) \bigl(\di\gamma_1+\di\gamma_2\bigr)}^{\frac{1}{2}}\\
    &\leq\bra*{ \int_0^1 \cA(\rho_t, \bj_t) \di t}^{\frac{1}{2}}\bra*{2 \int_0^1\iint_G\left( \abs{x-y}^2\vee\abs{x-y}^4\right) \eta(x,y)\di\mu(y)\di\rho_t(x)}^{\frac{1}{2}}\\
    &\leq\sqrt{2C_\eta}\bra*{ \int_0^1 \cA(\rho_t, \bj_t) \di t}^{\frac{1}{2}}.
  \end{align*}
Taking the supremum over all $1$-Lipschitz functions and the infimum in the couplings $(\rho,\bj)\in \CE(\rho^0,\rho^1)$ gives the result.
\end{proof}
The results above show that $\mt_\mu$ is an extended (meaning that it can take value $\infty$) quasi-metric on the set of probability measures which induces a topology stronger than the $W_1$-topology:
\begin{theorem}\label{thm:quasi-metric}
Let $\mu \in \cM^+(\R^d)$ satisfy Assumptions~\ref{it:as-conv} and~\ref{it:as-tight}. The nonlocal upwind transportation cost $\mathcal{T}_\mu$ defines an extended quasi-metric on $\mP_2(\Rd)$. The map $(\rho_0,\rho_1)\mapsto\mathcal{T}_\mu(\rho_0,\rho_1)$ is lower semicontinuous with respect to the narrow convergence. The topology induced by $\mathcal{T}_\mu$ is stronger than the $W_1$-topology and the narrow topology. In particular, bounded sets are narrowly relatively compact in $(\mP_2(\Rd),\mt_\mu)$.
\end{theorem}
\begin{proof}
 If $\mathcal{T}_\mu(\rho_0,\rho_1)=0$, then $\cA(\mu;\rho_t,\bj_t)=0$ for a.e.\ $t\in [0,1]$. Hence $\bj_t \equiv 0$ $\gamma_t$-a.e., which implies that $\rho_0\equiv \rho_1$ by the nonlocal continuity equation \eqref{eq:CE:weak:t01}. The triangle inequality is a consequence of Lemma~\ref{lem:reparametrization} and the fact that solutions to the nonlocal continuity equation can be concatenated. The lower semicontinuity and compactness properties of $\mathcal{T}_\mu$ are inherited from the action functional $\cA$ via Proposition \ref{prop:compactness-sol-ce}. 
In view of the comparison with $W_1$ from Proposition~\ref{prop:comparison_with_W_1}, we have that the topology induced by $\mathcal{T}_\mu$ is stronger than that induced by $W_1$ and the narrow topology.
\end{proof}
The next lemma provides a quantitative illustration of asymmetry of $\mathcal{T}$.
\begin{lemma}[Two-point space]\label{lem:two-point-space}
Let us consider the two-point graph $\Omega:=\set{0,1}$, with $\eta(0,1)=\eta(1,0)=\alpha>0$, $\mu(0)=p>0$ and $\mu(1)=q>0$. Let $\rho,\nu\in \mP_2(\Omega)$ and let $\rho_0, \rho_1, \nu_0, \nu_1 \in [0,1]$ be such that $\rho=\rho_0\delta_0+\rho_1\delta_1$ and $\nu=\nu_0\delta_0+\nu_1\delta_1$. There holds
\begin{equation}\label{eq:dits-two-point-space}
    \mathcal{T}(\rho,\nu)= 
    \begin{cases}
      \frac{2}{\sqrt{\alpha p}} \bra*{ \sqrt{\rho_1}-\sqrt{\nu_1}} &\text{ if } \rho_0 < \nu_0, \\
      \frac{2}{\sqrt{\alpha q}} \bra*{ \sqrt{\rho_0}-\sqrt{\nu_0}} &\text{ if } \nu_0 < \rho_0.
    \end{cases}
\end{equation}
\end{lemma}
\begin{proof}
Let us fix $\lambda=\delta_{(0,1)}+\delta_{(1,0)}$ and notice that $\rho_0+\rho_1=1$ and $\nu_0+\nu_1=1$ as $\rho,\nu$ are probability measures. Since $\Omega=\{0,1\}$, note that for any curve $t\in[0,1]\mapsto \rho_t\in \mP_2(\Omega)$ there exists a function $g\colon t\in [0,1]\mapsto g_t\in[0,1]$ accounting for the mass displacement. Thus, we notice that $(\rho,\bj)\in\CE(\rho,\nu)$ if
\begin{align*}
\rho_t=g_t\delta_0+(1-g_t)\delta_1,\quad \bj_t\ \text{ such that }\ \bj_t(0,1)=-\frac{\dot g_t}{\alpha}\ \text{ and }\ \bj_t(1,0)=\frac{\dot g_t}{\alpha}, \quad \mbox{for all $t\in[0,1]$}.  
\end{align*}
Hence, using that $\bj_t$ is antisymmetric yields
\begin{align*}
    \mathcal{T}(\rho_0,\rho_1)^2&=\inf\left\{\int_0^1\cA(\rho_t,\bj_t)\di t:(\rho,\bj)\in\CE(\rho_0,\rho_1)\right\}\\
    &=\inf_g\left\{\int_0^1\frac{|(\dot g_t)_-|^2}{\alpha g_t q}+\frac{|(\dot g_t)_+|^2}{\alpha(1-g_t)p}\di t : g_0 = \rho_0 \text{ and } g_1 = \nu_0 \right\}.
\end{align*}
Now, let us assume without loss of generality that $\rho_0 < \nu_0$. Obviously, in this configuration we can restrict the above infimum among non-decreasing $g$, as it gives a lower action. Therefore, by applying Jensen's inequality, we have
\begin{align*}
\mathcal{T}(\rho,\nu)^2&=\inf_{g\nearrow}\frac{1}{\alpha p}\int_0^1\frac{|\dot g_t|^2}{(1-g_t)}\di t = \inf_{g\nearrow}\frac{1}{\alpha p}\int_0^1\left|-2\frac{\di}{\di t} \left(\sqrt{1-g_t}\right)\right|^2\di t \\
&\geq \inf_{g\nearrow} \frac{4}{\alpha p} \abs*{ \int_0^1 -\frac{\di}{\di t}\left( \sqrt{1-g_t}\right)\di t}^2 = \frac{4}{\alpha p} \bra*{\sqrt{1-\rho_0} -\sqrt{1-\nu_0}}^2 =  \frac{4}{\alpha p} \bra*{\sqrt{\rho_1} -\sqrt{\nu_1}}^2  \,.
\end{align*}
The equality case is obtained by noting that the solution to $-\frac{\di}{\di t} \sqrt{1-g_t}=\sqrt{\rho_1}-\sqrt{\nu_1}$ for all $t\in[0,1]$, with consistent boundary values $g_0=\rho_0$ and $g_1=\nu_0$, is given by $g_t = 1-\bigl(\sqrt{\rho_1}(1-t)+\sqrt{\nu_1} t \bigr)^2$. The case $\nu_0<\rho_0$ is obtained in a similar manner, which gives formula \eqref{eq:dits-two-point-space}.
\end{proof}
\begin{rem}
 The quasi-metric is in general already non-symmetric on the two-point space, which one can best observe in Figure~\ref{fig:two-point}.
 In the case $p=\frac{1}{2}$, the swapping $\hat\rho_0 = \rho_1$ and $\hat\rho_1 = \rho_0$ preserves the quasi-distance $\mathcal{T}(\rho,\nu)= \mathcal{T}(\hat\rho,\hat{\nu})$.
  \begin{figure}[ht]
  \includegraphics[width=0.45\textwidth]{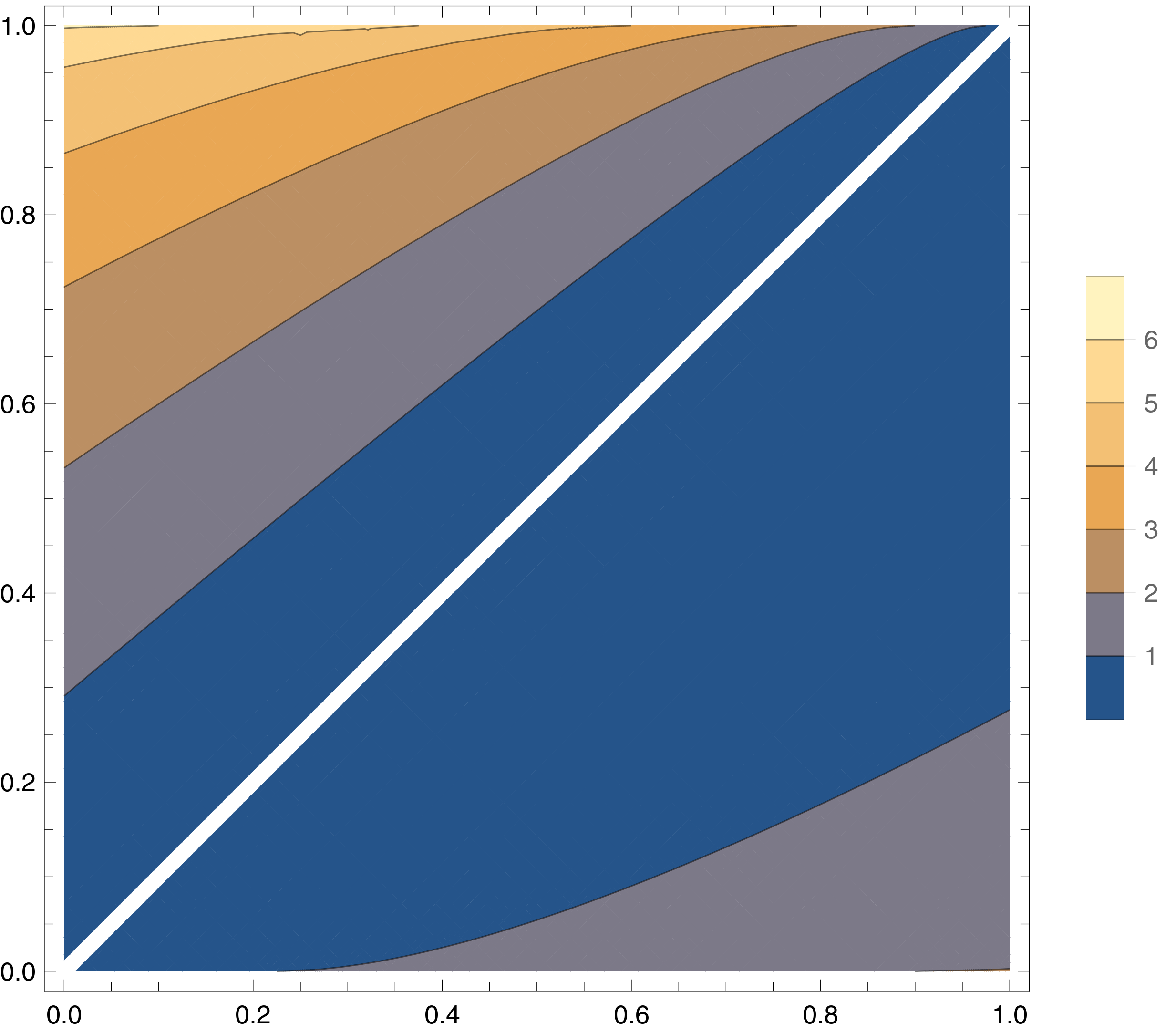}$\qquad$
  \includegraphics[width=0.45\textwidth]{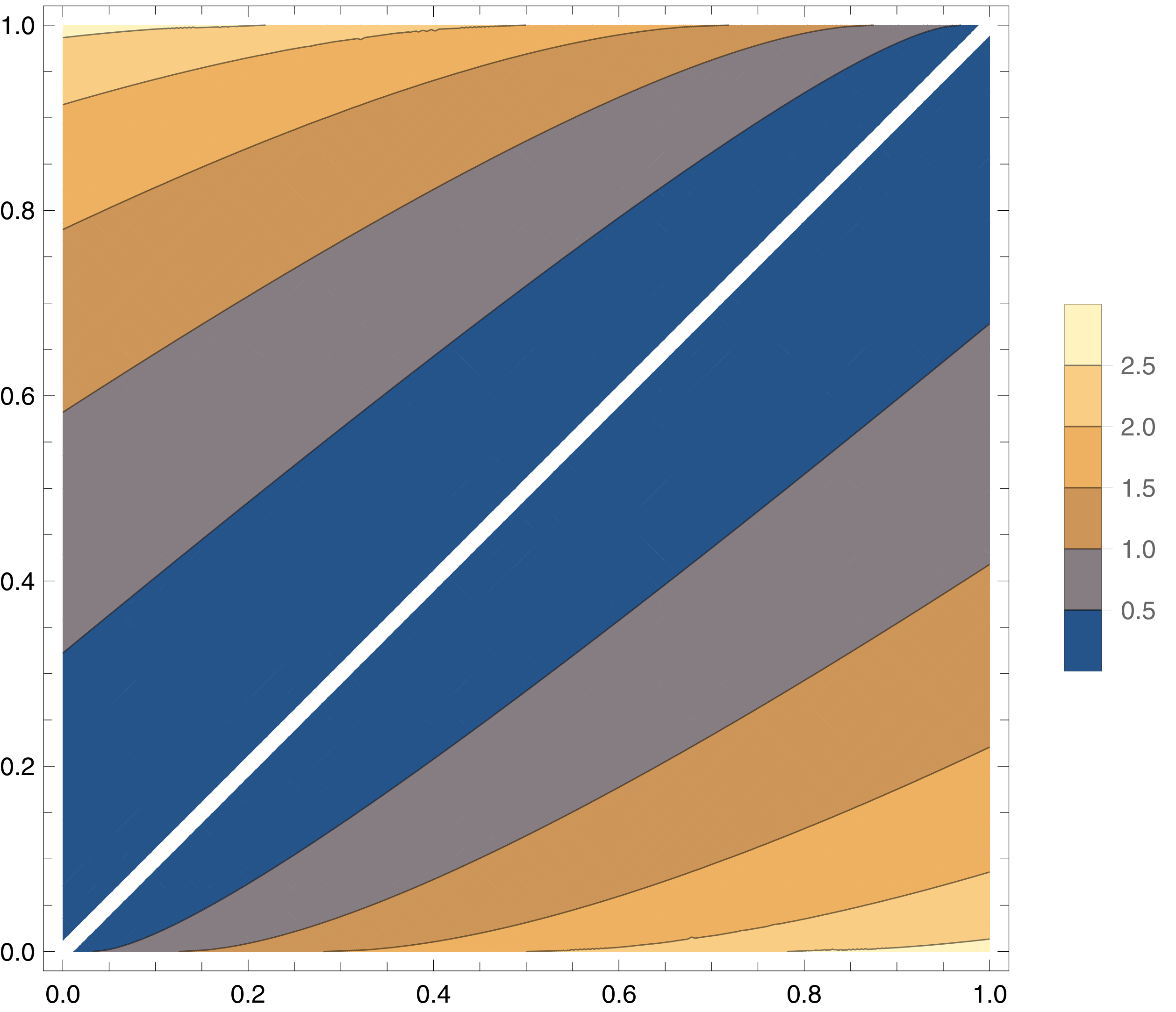}
  \caption{In the context of Lemma \ref{lem:two-point-space}, in the above figures we parametrize the quasi-distance $\mathcal{T}(\rho,\nu)$ by $\rho_0\in[0,1]$ and $\nu_0\in[0,1]$ for $\mu(0)=0.1$ (left) and $\mu(1)=0.5$ (right). Colors represent different values of $\mathcal{T}(\rho,\nu)$ with respect to the initial values $\rho_0$ and $\nu_0$. In the left figure, by swapping the values of $\rho_0$ and $\nu_0$ on the axes, we can see that $\mathcal{T}$ is non-symmetric.}\label{fig:two-point}
 \end{figure} 
\end{rem}
We now adapt the standard definition of absolutely continuous curves in metric spaces from \cite[Chapter 1]{AGS} to our setting. Let $\mu\in \cM^+(\Rd)$ satisfy Assumptions~\ref{it:as-conv} and~\ref{it:as-tight}. A curve $[0,T]\ni t\mapsto\rho_t\in\mP_2(\Rd)$ is said to be \emph{$2$-absolutely continuous} with respect to $\mathcal{T}_\mu$ if there exists $m\in L^2((0,T))$ such that
\begin{equation}\label{eq:abs-cont-cond}
\mathcal{T}_\mu(\rho_{t_0},\rho_{t_1})\le\int_{t_0}^{t_1}m(t)\di t\quad \mbox{for all $0< t_0\le t_1< T$}.
\end{equation}
In this case, we write $\rho\in\AC\bigl([0,T];(\mP_2(\Rd),\cT_\mu)\bigr)$. For any $\rho\in \AC\bigl([0,T];(\mP_2(\Rd),\cT_\mu)\bigr)$ the quantity
\begin{equation}\label{e:def:metric-derivative}
|\rho'_t|:=\lim_{h\to0}\frac{\mathcal{T}_\mu(\rho_t,\rho_{t+h})}{|h|}
\end{equation}
is well-defined for a.e.\ $t\in[0,T]$ and is called the \emph{metric derivative} of $\rho$ at $t$. Moreover, the function $t\to|\rho'|(t)$ belongs to $L^2((0,T))$ and it satisfies $|\rho'|(t)\le m(t)$ for a.e.\ $t\in[0,T]$, which means $\rho'$ is the minimal integrand satisfying \eqref{eq:abs-cont-cond}. The length of a curve $\rho\in \AC\bigl([0,T];(\mP_2(\Rd),\cT_\mu)\bigr)$ is defined by $L(\rho):=\int_0^T|\rho'|(t)\di t$.

\begin{proposition}[Metric velocity]\label{prop:metric-velocity}
Let $\mu\in \cM^+(\Rd)$ satisfy Assumptions~\ref{it:as-conv} and~\ref{it:as-tight}. A curve $(\rho_t)_{t\in[0,T]}\subset\mP_2(\Rd)$ belongs to $\AC([0,T];(\mP_2(\Rd),\cT_\mu))$ if and only if there exists a family $(\bj_t)_{t\in[0,T]}$ such that $(\rho,\bj)\in\CE_T$ and
\[
  \int_0^T\sqrt[]{\cA(\mu;\rho_t,\bj_t)}\di t<\infty.
\]
In this case, the metric derivative is bounded as in $|\rho'|^{2}(t)\le\cA(\mu;\rho_t,\bj_t)$ for a.e.\ $t\in[0,T]$. In addition, there exists a unique family $(\tilde{\bj}_t)_{t\in[0,T]}$ such that $(\rho,\tilde{\bj})\in\CE_T$ and
\begin{equation}\label{eq:metric-velocity:opt}
|\rho'|^2(t)=\cA(\mu;\rho_t,\tilde{\bj}_t)\qquad for\ a.e.\ t\in[0,T].
\end{equation}
Hereby, the previous identity holds if and only if $\tilde{\bj}_t\in T_{\rho}\mP_2(\Rd)$ for a.e.\ $t\in[0,T]$, where
\begin{equation}\label{eq:charact:tangent}
T_\rho\mP_2(\Rd)=\{\bj\in\cM_{\gamma_1}^{\mathrm{as}}(G):\cA(\mu;\rho,\bj)<\infty,\cA(\mu;\rho,\bj)\le\cA(\mu;\rho,\bj+\bm{d}) \text{ for all } \bm{d}\in\cM_{\Div}(G)\},
\end{equation}
with $\cM_{\gamma_1}^{\mathrm{as}}(G)$ defined in~\eqref{eq:def:Mas}, and $\cM_{\Div}(G)$ the set of nonlocal divergence-free fluxes, that is
\[
  \cM_{\Div}(G)=\left\{\bm{d}\in\cM(G):\iint_G\dgrad\psi\, \eta \di\bm{d}=0\ for\ all\ \psi\in C_\mathrm{c}^\infty(\Rd)\right\}.
\]
\end{proposition}
\begin{proof}
The first statement on the characterization of absolutely continuous curves as curves of finite action follows from \cite[Theorem 5.17]{DNS09}, in view of Lemma \ref{lem:reparametrization} and Propositions \ref{prop:compactness-sol-ce} and \ref{prop:min-metric}. Let us now show that \eqref{eq:metric-velocity:opt} holds if and only if $\tilde\bj_t$ belongs to $T_\rho\mP_2(\Rd)$ for a.e.\ $t\in[0,1]$, given by~\eqref{eq:charact:tangent}. Let $t\in[0,1]$ be so that $\bj_t$ verifies $\cA(\mu;\rho_t,\bj_t) <+ \infty$. Due to Corollary \ref{cor:antsym_vect_field_lower_action}, the element~$\tilde\bj_t$ of minimal action satisfying \eqref{eq:metric-velocity:opt} is characterized by $\partial_t \rho_t + \dgrad \cdot\bj_t = 0 = \partial_t \rho_t +\dgrad \cdot\tilde\bj_t$, that is,
 \[
   \tilde \bj_t = \argmin_{\bj \in \cM_{\gamma_1}^{\mathrm{as}}(G)}\bigl\{ \cA(\mu;\rho_t, \bj) : \dgrad\cdot \bj = \dgrad\cdot \bj_t\bigr\}.
 \]
Recalling the notation for the Jordan decomposition of a measure from Section \ref{subsec:action}, note that we use that the functional $\bj \mapsto \cA(\mu;\rho,\bj)$ is strictly convex for $\bj\in \cM(G)$ such that $\bj^+ \ll \rho\otimes \mu$ and $\bj^- \ll \mu \otimes \rho$, which is guaranteed above since $\cA(\mu;\rho,\bj) < \infty$ and $\bj\in\cM^{\mathrm{as}}_{\gamma_1}(G)$.
Then, we observe the set $\{\bj\in\cM_{\gamma_1}^{\mathrm{as}}(G):\dgrad\cdot\bj=\dgrad\cdot\bj_t\}$ is closed with respect to the narrow convergence. In addition, the estimate \eqref{eq:control-phi} from Lemma \ref{lem:A:TV:bound} with $\Phi(x,y) = |x-y|\vee |x-y|^2$ gives
\[
 \frac12\iint_K \eta(x,y) \di\abs*{\bj}(x,y)\leq \frac{\sqrt{2C_\eta} \sqrt{\cA(\mu;\rho_t,\bj)}}{\inf_K(|x-y|\vee|x-y|^2)} \quad \text{for all compact $ K\subset G$},
\]
showing that the sublevel sets of $\bj \mapsto \cA(\mu;\rho_t,\bj)$ are locally relatively compact with respect to the narrow convergence, arguing as in the proof of Proposition \ref{prop:compactness-sol-ce}. Hence the element $\tilde \bj_t$ is well-defined by applying the direct method of calculus of variations.
\end{proof}
We defined the tangent space $T_\rho\mP_2(\Rd)$ in \eqref{eq:charact:tangent} using the nonlocal fluxes $\bj$. We note that this is in some way a nonlocal, Lagrangian description of the tangent vectors and that the relationship between this Lagrangian description and the Eulerian description is the nonlocal continuity equation
\[ \partial_t \rho_t = - \dgrad \cdot \bj, \]
which is satisfied in the weak sense. This provides a useful heuristic, but as for classical Wasserstein gradient flows \cite{AGS} the precise, rigorous definition of the tangent space is in Lagrangian form; we note, however, that here we use fluxes instead of velocities. This is not just a superficial difference. Namely, as can be seen in Proposition  \ref{prop:tangent-bundle}, the relation between velocities and fluxes is not linear and thus the velocities do not provide a linear parametrization of the tangent space.
We use the argument from \cite[Theorem 5.21]{DNS09} to characterize the tangent space $T_\rho\mP_2(\Rd)$ in more detail:
\begin{proposition}[Tangent fluxes have almost gradient velocities]\label{prop:tangent-bundle}
Let $\mu\in \cM^+(\Rd)$ satisfy Assumptions~\ref{it:as-conv} and~\ref{it:as-tight}, and $\rho\in\mP_2(\Rd)$.
Then, it holds that $\bj\in T_{\rho}\mP_2(\Rd)$ if and only if $\bj\in \cM(G)$ with $\bj^+\ll\gamma_1$, $\bj^- \ll \gamma_2$, and $v^+:=\frac{\di\bj^+}{\di\gamma_1}$, $v^- :=\frac{\di\bj^-}{\di\gamma_2}$ satisfy, for $v:=v^+ - v^-:G\to \R$, the relation
\begin{equation}\label{eq:tangent-bundle}
  v \in \overline{\left\{\dgrad\varphi : \varphi\in C_\mathrm{c}^\infty(\Rd)\right\}}^{L^2(\eta\,\widehat\gamma^v)}, \qquad\text{where}\qquad
  \di\widehat\gamma^v = \chi_{\set{v>0}} \di\gamma_1 + \chi_{\set{v<0}} \di\gamma_2.
\end{equation}
\end{proposition}
\begin{proof}
  If $\cA(\mu;\rho,\bj)<\infty$, then by Lemma \ref{lem:action} it holds for some $v\in \cV^{\mathrm{as}}(G)$ that
  \begin{align*}
    \di\bj(x,y) &= v(x,y)_+ \di\gamma_1(x,y) - v(x,y)_- \di\gamma_1(y,x) = v(x,y) \di\gamma_+(x,y) - v(y,x) \di\gamma_+(y,x) \ ,
  \end{align*}
  where $\gamma_+ = \gamma_1|_{J^+}$, with $J^+ = \supp \bj^+$, and we used that $(J^+)^\top = \supp \bj^-$. Then, by recalling the definition of the norm on $L^2(\eta\,\gamma_1)$ from~\eqref{eq:def:L2G},
  \[
    \cA(\mu;\rho,\bj) = 2 \norm{v_+}_{L^2(\eta\,\gamma_1)}^2 = 2 \norm{v}_{L^2(\eta\,\gamma_+)}^2 . 
  \]
  By using the relation between $\bj$ and $v$ from above, we can rewrite the divergence $\dgrad\cdot \bj$ in weak form for any $\psi\in C_\mathrm{c}^\infty(\Rd)$:
  \[
    \frac12\iint_G \dgrad \psi\,\eta \di\bj =  \iint_G \dgrad\psi \, v_+ \, \eta \di\gamma_1 = \iint_G \dgrad \psi \, v \, \eta \di\gamma_+.
  \]
  Now, the characterization \eqref{eq:charact:tangent} of $\bj \in T_\rho \mP_2(\Rd)$ is equivalent to
  \[
    \iint_G \abs{v}^2 \eta \di\gamma_+ \leq \iint_G \abs{v+w}^2 \eta \di\gamma_+ \quad\text{for all $w\in \cV^\mathrm{as}(G)$ so that $\iint_G \dgrad\psi\, w \, \eta \di\gamma_+ = 0\;\; \forall \psi\in C_\mathrm{c}^\infty(\Rd)$}.
  \]
  Hence $v^+$ belongs to the closure of $\set{\dgrad \varphi :  \varphi \in C_\mathrm{c}^\infty(\Rd) }$ in $L^2(\eta\,\gamma_+)$. From the antisymmetry of $v$ follows that $v^-$ belongs to the closure of $\set{\dgrad \varphi :  \varphi \in C_\mathrm{c}^\infty(\Rd) }$ in $L^2(\eta\,\gamma_-)$. Thus, the conclusion follows from the identity $\gamma_+ + \gamma_+^\top = \hat\gamma^v$ on $G$.
\end{proof}
\begin{rem}
Proposition \ref{prop:tangent-bundle} shows that for $\mu$ as in its statement, $\rho\in\mP_2(\Rd)$ and $\bj$ chosen from a dense subset of $T_{\rho}\mP_2(\Rd)$, there exists a measurable $\varphi:\Rd\to\R$ such that we have the identity
\[
\cA(\mu;\rho,\bj)=\cA(\mu;\rho,\dgrad\varphi\,\gamma_1)=\iint_G\abs[\big]{\bra[\big]{\dgrad\varphi}_+}^2\eta \di\gamma_1.
\]
\end{rem}
Finally, we provide an interesting property of absolutely continuous curves.
\begin{proposition}[Absolutely continuous curves stay supported on $\mu$] \label{prop:supp-in-mu}
    Let $\mu\in \cM^+(\Rd)$ satisfy Assumptions~\ref{it:as-conv} and~\ref{it:as-tight} and $\rho\in \AC([0,T],(\mP_2(\R^d),\cT_\mu))$ be such that $\supp \rho_0 \subseteq \supp \mu$. Then, for all $t\in[0,T]$, it holds $\supp\rho_t\subseteq \supp \mu$. 
\end{proposition}
\begin{proof}
   Since $(\rho_t)_{t\in[0,T]}$ is absolutely continuous, there exists by Proposition~\ref{prop:metric-velocity} a unique family $(\bj_t)_{t\in[0,T]}$ such that $(\rho,\bj) \in \CE_T$ and $\bj_t \in T_{\rho_t}\cP_2(\R^d)\subseteq \cM^{\mathrm{as}}_{\gamma_{1,t}}(G)$, where $\gamma_{1,t} = \rho_t \otimes \mu$, and $\abs{\rho_t'}^2= \cA(\mu;\rho_t,\bj_t)$ for a.e.~$t\in [0,T]$. 
   In particular, by Lemma~\ref{lem:action}, there exists a measurable family $(v_t)_{t\in[0,T]}\subset \cV^{\mathrm{as}}(G)$ such that 
   \[ 
     \di\bj_t(x,y) = v_t(x,y)_+\di\rho_t(x)\di\mu(y) - v_t(x,y)_-\di\mu(x)\di\rho_t(y).
   \]
   Without loss of generality, let $(\rho_t)_{t\in[0,T]}$ be the weakly continuous curve from Lemma~\ref{lem:nce:weak_cont} satisfying, for any test function $\varphi\in C_\mathrm{c}^\infty(\R^d)$ and $t\in[0,T]$,
   \begin{align*}
     \int_{\R^d} \varphi(x) \di\rho_t(x) &= \int_{\R^d} \varphi(x) \di\rho_0(x) + \frac12\int_0^t \iint_G \dgrad\varphi(x,y)\eta(x,y)\di\bj_s(x,y)\di{s} \\
     &= \int_{\R^d} \varphi(x) \di\rho_0(x) +  \int_0^t \iint_G \dgrad\varphi(x,y) v_s(x,y)_+ \eta(x,y)\di(\rho_s\otimes\mu)(x,y)\di{s} .
   \end{align*}
   Now, let $\varphi\in C_\mathrm{c}^\infty(\R^d)$ with $\varphi\geq 0$ and $\supp \varphi \subseteq \R^d \setminus \supp \mu$. Then, for all $t\in[0,T]$, it holds 
   \[
    \frac12\int_{\R^d} \varphi(x) \di\rho_t(x) = - \int_0^t \iint_G \varphi(x) v_s(x,y)_+ \eta(x,y)\di(\rho_s\otimes\mu)(x,y)\di{s} \leq 0 ,
   \]
   which implies that $\supp \rho_t \subseteq \supp \mu$, since $\rho_t \in \cP(\R^d)$ is in particular a non-negative measure for all $t\in[0,T]$ by Lemma~\ref{lem:nce:weak_cont}.
\end{proof}
%
%
\section{Nonlocal nonlocal-interaction equation} \label{sec:NLIE}
In this section we consider gradient flows in the spaces of probability measures $\mP_2(\Rd)$ endowed with the nonlocal transportation quasi-metric $\cT_\mu$, defined by \eqref{eq:nonloc-upwind-transp-cost}. From now until Section \ref{sec:stab-exist} (excluded) we fix $\mu\in \cM^+(\Rd)$ satisfying \ref{it:as-conv} and~\ref{it:as-tight}, unless otherwise specified. For this reason we shall use the simplifications $\cA(\rho,\bj)$ for $\cA(\mu;\rho,\bj)$ and $\cT$ for $\cT_\mu$.

In this section investigate the nonlocal nonlocal-interaction equation~\eqref{eq:nlnl-intro} as a gradient flow with respect to the metric $\cT$. We restate it in a one-line form and note that from now on we consider the external potential $P \equiv 0$. The extension to $P \not\equiv 0$ is straightforward; see Remark \ref{rem:ExtPot}.
\begin{equation}\label{eq:nlnl-interaction-eq}
     \partial_t\rho_t(x)+\int_\Rd \dgrad(K*\rho)(x,y)_- \eta(x,y) \rho_t(x) \di\mu(y) - \int_\Rd \dgrad(K*\rho)(x,y)_+ \eta(x,y) \di\rho_t(y)=0. \tag{NL$^2$IE}
 \end{equation}
In the classical setting of gradient flows in the spaces of probability measures endowed with the Wasserstein metric \cite{AGS,CDFLS11}, the nonlocal-interaction equation 
\begin{equation} \label{eq:NLIE}
\partial_t\rho_t+ \nabla \cdot ( \rho_t \nabla (K * \rho_t)) = 0
\end{equation}
is the gradient flow of the \emph{nonlocal-interaction energy} 
\begin{equation}\label{eq:inter-en-functional}
    \cE(\rho)= \frac{1}{2}\iint_{\Rd\times\Rd} K(x,y)\di\rho(x)\di\rho(y),
\end{equation}
We start by discussing  the geometry of \eqref{eq:nlnl-interaction-eq} and interpret it as the gradient flow of \eqref{eq:inter-en-functional} in the infinite-dimensional Finsler manifold 
of measures endowed with the Finsler metric associated to $\cT$. 
Following this, we develop a framework of gradient flows in the quasi-metric space $\cT$, which extends the setup of gradient flows in metric spaces \cite{AGS} to quasi-metric spaces. In particular, we build the existence theory for \eqref{eq:nlnl-interaction-eq} based on this approach.

Above, for simplicity, \eqref{eq:nlnl-interaction-eq} was written for $\rho\ll\mu$, where we recall that we used the notation~$\rho$ to denote both the measure and the density with respect to $\mu$. Our framework, however, also applies to the case when $\rho$ is not absolutely continuous with respect to $\mu$. The general weak form of \eqref{eq:nlnl-interaction-eq} is obtained in terms of the nonlocal continuity equation as introduced in Section \ref{subsec:nonloc-cont-eq}. Specifically, we have the following definition.
\begin{definition}\label{defn:nl2ie}
A curve $\rho \colon [0,T]\to \cP_2(\R^d)$ is called a \emph{weak solution} to~\eqref{eq:nlnl-interaction-eq} 
if, for the flux $\bj\colon [0,T]\to\cM(G)$ defined by
\begin{equation*}
  \di\bj_t(x,y)=\dgrad\frac{\delta \cE}{\delta\rho}(x,y)_- \di\rho_t(x)\di\mu(y)-\dgrad\frac{\delta \cE}{\delta\rho}(x,y)_+ \di\rho_t(y)\di\mu(x),
\end{equation*}
the pair $(\rho,\bj)$ is a weak solution to the continuity equation
\begin{equation*}
  \partial_t\rho_t+\dgrad\cdot\bj_t=0 \qquad \text{on}\ [0,T]\times\Rd,
\end{equation*}
according to Definition \ref{def:nce-flux-form}.
\end{definition}

Here we list the assumptions on the interaction kernel $K\colon\Rd\times\Rd\to\R$ we refer to throughout this section:
\begin{enumerate}[label=\textbf{(K\arabic*)}]
    \item\label{as:K:cont} $K\in C(\Rd\times\Rd)$;
    \item\label{as:K:sym} $K$ is symmetric, i.e., $K(x,y)=K(y,x)$ for all $(x,y)\in\Rd\times\Rd$;
    \item\label{as:K:LipQuad} $K$ is $L$-Lipschitz near the diagonal and at most quadratic far away, that is there exists some $L\in (0,\infty)$ such that, for all $(x,y),(x',y')\in\Rd\times\Rd$,
    \[
      |K(x,y)-K(x',y')|\le L\left(|(x,y)-(x',y')|\vee|(x,y)-(x',y')|^2\right).
    \]
\end{enumerate}

\begin{rem}\label{rem:ExtPot}
 Assumption~\ref{as:K:LipQuad} implies that, for some $C >0$ and all $x,y\in \Rd$,
 \begin{equation}\label{as:K:Quad}
   \abs{K(x,y)} \leq C \bra*{1+ \abs{x}^2 + \abs{y}^2};
 \end{equation}
 indeed, for fixed $(x',y')\in \Rd\times \Rd$, \ref{as:K:LipQuad} yields
 \[
   \abs{K(x,y)} - \abs{K(x',y')} \leq L \bra*{ 1 \vee 2\bra*{ |(x,y)|^2 + |(x',y')|^2}},
 \]
 and bounding the maximum ($\vee$) by the sum, we arrive at $\abs{K(x,y)} \leq L +2 L \bra*{|(x',y')|^2 + |(x,y)|^2} + \abs{K(x',y')}$, which gives~\eqref{as:K:Quad} with $C=2L\bigl(1+|(x',y')|^2\bigr) + \abs{K(x',y')}$. We notice, by the way, that the bound~\eqref{as:K:Quad} implies that $\cE\colon \mP_2(\Rd)\to \R$ is proper with domain equal to ~$\cP_2(\Rd)$.
 
As mentioned previously, the theory in this section can be easily extended to energies of the form~\eqref{e:energy:intro} including potential energies $\cE_P(\rho)=\int_\Rd P \dx{\rho}$ for some external potential $P\colon\R^d \to \R$ satisfying a local Lipschitz condition with at-most-quadratic growth at infinity; that is, similarly to~\ref{as:K:LipQuad}, there exists $L\in (0,\infty)$ so that for all $x,y\in \R^d$ we have
 \[
  \abs{P(x)-P(y)} \leq L \bra*{ |x-y|\vee |x-y|^2} . 
 \]
\end{rem}
We now show that, under the above assumptions on the interaction potential $K$, we have narrow continuity of the energy.
\begin{proposition}[Continuity of the energy]\label{prop:cont-energy}
  Let the interaction potential $K$ satisfy Assumptions~\ref{as:K:cont}--\ref{as:K:LipQuad}. Then, for any sequence $(\rho^n)_n \subset \mP_2(\Rd)$ such that $\rho^n \rightharpoonup \rho$ as $n\to\infty$ for some $\rho\in\mP_2(\Rd)$, we have
  \[
    \lim_{n\to\infty} \cE(\rho^n) = \cE(\rho).
  \]
\end{proposition}
\begin{proof}
  Let $(\rho^n)_n \subset \mP_2(\Rd)$ and $\rho\in\mP_2(\Rd)$ be such that $\rho^n \rightharpoonup \rho$ as $n\to\infty$. For all $R>0$, we write~$\overline B_R$ the closed ball of radius $R$ centered at the origin in $(\Rd)^2$ and $\varphi_R \colon (\Rd)^2 \to \R$ a continuous function such that $\varphi_R(z) = 1$ for all $z\in \overline B_R$, $\varphi_R(z) = 0$ for all $z\in (\Rd)^2 \setminus \overline B_{2R}$, and $\varphi_R(z) \leq 1$ for all $z\in (\Rd)^2$. For all $R>0$, we then set $K_R = \varphi_R K$ and
  \begin{equation*}
    \cE_R(\nu) = \frac{1}{2}\iint_{\Rd\times \R^d} K_R(x,y)\di\nu(y)\di\nu(x) \quad \mbox{for all $\nu\in\mP_2(\Rd)$}.
  \end{equation*}
  Since $(\rho^n)_n$ converges narrowly to $\rho$ as $n\to\infty$ and $K_R$ is bounded and continuous, we get
  \[
    \cE_R(\rho^n) \to \cE_R(\rho) \quad \mbox{as $n\to\infty$}.
  \]
  Furthermore, since $K_R \to K$ pointwise as $R\to\infty$, $|K_R| \leq |K|$ for all $R>0$, the domain of~$\cE$ is~$\mP_2(\Rd)$ and $\rho \in \mP_2(\Rd)$, we also have
  \[
    \cE_R(\rho) \to \cE(\rho) \quad \text{as $R\to\infty$}
  \]
  by the Lebesgue dominated convergence theorem. Similarly, we also have
  \[
    \cE_R(\rho^n) \to \cE(\rho^n) \quad \mbox{as $R\to\infty$ for all $n\in\N$}.
  \]
  By a diagonal argument, we deduce the result.
\end{proof}
\subsection{Identification of the gradient in Finsler geometry}\label{subsec:HeuristicFinsler}
Since the nonlocal upwind transportation cost $\mathcal{T}$ is only a quasi-metric, the underlying structure of $\mP_2(\Rd)$ does not have the formal Riemannian structure as it does in the classical gradient flow theory, but a Finslerian structure instead. This highlights the fact that at every point $\rho \in \mP_2(\Rd)$ the tangent space $T_\rho\cP_2(\Rd)$ is not a Euclidean space, but rather a manifold in its own right. 

In this section we provide calculations, in the spirit of Otto's calculus, that characterize the gradient descent in the infinite-dimensional Finsler manifold of probability measures endowed with the nonlocal transportation quasi-metric $\mathcal{T}$. To keep the following considerations simple, we assume that $\rho$ is a given probability measure which is absolutely continuous with respect to $\mu$.
In this way, we avoid the need to introduce yet another measure $\lambda\in \cM^+(G)$ with respect to which all of the occurring measures are absolutely continuous, similar to how we proceeded in Definition~\ref{def:action} for the action. This restriction is done solely to make the presentation clearer and highlight the geometric structure.
Hence any flux $\bj$ of interest is absolutely continuous with respect to $\mu \otimes \mu$ and we can think of $\bj$ via its density with respect to~$\mu \otimes \mu$, which we shall denote by $j$ (using a letter which is not bold).

At every tangent flux $\bj \in T_\rho\mP_2(\Rd)$ we define an inner product $g_{\rho,\bj}\colon T_\rho\mP_2(\Rd) \times T_\rho\mP_2(\Rd) \to \R$ by
\begin{equation}\label{eq:def:Finsler:g}
    g_{\rho,\bj}(\bj_1,\bj_2) = \frac{1}{2}\iint_G j_1(x,y)\,j_2(x,y)\, \eta(x,y) \left(\frac{\chi_{\{j>0\}}(x,y)}{\rho(x)} + \frac{\chi_{\{j<0\}}(x,y)}{\rho(y)} \right) \di \mu(x) \di \mu(y),
\end{equation}
where $\{j>0\}$ is an abbreviation for $\{(x,y) \in G \colon j(x,y)>0\}$ and similarly for $\{j<0\}$.
The ratios are well-defined since $\rho$ cannot be zero where $j$ is not zero. We note that this is the bilinear form that corresponds to the quadratic form defining the action (see Definition \ref{def:action} and Remark~\ref{rem:action:abs_cont}); namely,
\[    
    g_{\rho,\bj}(\bj,\bj)  = \cA(\mu; \rho, \bj).
\]
We refer the reader to Appendix \ref{app:minkowski} for a derivation of this inner product from a Minkowski norm on~$T_\rho\mP_2(\Rd)$ as it is required in Finsler geometry. We recall that from Proposition \ref{prop:tangent-bundle} a dense subset of tangent-fluxes $\bj$ are characterized by the existence of
a potential $\varphi \in C_\mathrm{c}^\infty(\Rd)$ such that, for $\mu \otimes\mu$-a.e.~$(x,y) \in G$,
\begin{equation} \label{eq:tan_flux_vec}
    j(x,y) = \dgrad \varphi(x,y) \left(\rho(x) \chi_{\{\dgrad \varphi>0\}}(x,y) + \rho(y) \chi_{\{\dgrad \varphi<0\}}(x,y) \right).
\end{equation}
In this Finsler setting, we now want to determine the direction of steepest descent from $\rho$, for the underlying energy defined in \eqref{eq:inter-en-functional}. The gradient vector of some energy $\cE\colon \cP(\R^d)\to \R$ at $\rho$, which we denote by $\grad \cE(\rho)$, is defined as the tangent vector which satisfies
\[ 
    \Diff_\rho \cE[\bj] = g_{\rho,\grad \cE(\rho)}\bigl(\grad \cE(\rho), \bj\bigr) \qquad \text{for all $\bj \in T_\rho\mP_2(\Rd)$},
\]
provided this vector exists and is unique. Here, we use the continuity equation Definition~\ref{def:nce-flux-form} to define variations via 
\[  
    \Diff_{\rho}\cE[\bj] = \left.\frac{\di}{\di t}\right|_{t=0} \cE(\tilde \rho_t), 
\]
where $\tilde \rho$ is any curve such that $\tilde \rho_0= \rho$ and $\left.\frac{\di}{\di t}\right|_{t=0}\tilde \rho_t = - \dgrad \cdot \bj$. From Definition \ref{def:nl_grad_div}, due to $\mu\otimes\mu$-absolute continuity of $\bj$ we have that 
\[ 
    -\dgrad \cdot \bj(x) = -\int \eta(x,y) j(x,y) \di{\mu}(y) \qquad \text{for $\mu$-a.e.\ $x \in \Rd$}.
\]
In the case, when $\cM$ is a finite-dimensional Finsler manifold, such gradient vector exists and is unique since the mapping $\ell\colon T_\rho \cM \to (T_{\rho}\cM)^*,\, \bj \mapsto g_{\rho,\bj}(\bj,\cdot)$, is a bijection; see \cite[Proposition 1.9]{Dahl}. For further details into Finsler geometry, we refer the reader to \cite{Shen,Chern}. In our case, we can at least claim that the functional $\ell_\rho\colon T_\rho\mP_2(\Rd) \to (T_\rho\mP_2(\Rd))^*$, given for $\bj\in T_\rho\mP_2(\Rd)$ by
\begin{equation}\label{eq:def:ell}
    \bj_2 \mapsto \ell_\rho(\bj)(\bj_2) = g_{\rho,\bj}(\bj,\bj_2) = \frac12\iint_G j_2(x,y)  \, \eta(x,y) \left(\frac{j(x,y)_+}{\rho(x)}-\frac{j(x,y)_-}{\rho(y)} \right) \di \mu(x) \di \mu(y) ,
\end{equation}
is injective $\eta\, \mu \otimes \mu$-a.e.; that is, the existence of a gradient implies its uniqueness ($\eta\, \mu \otimes \mu$-a.e.), in which case we have
\begin{equation*}
       \ell_\rho(\grad \cE(\rho)) = \Diff_\rho \cE. 
\end{equation*}
To see the injectivity of~\eqref{eq:def:ell}, we first note that $\ell_\rho$ is positively $1$-homogeneous by definition. Moreover, we have the following one-sided version of a Cauchy--Schwarz-type estimate
\begin{align}
  \ell_\rho(\bj)(\bj_2) &\leq \frac12\iint_G \!\left(\frac{ j_2(x,y)_+ j(x,y)_+}{\rho(x)} + \frac{ j_2(x,y)_- j(x,y)_-}{\rho(y)} \right) \eta(x,y) \di\mu(x) \di\mu(y) \nonumber\\
  &\leq \sqrt{\ell_\rho(\bj)(\bj) \, \ell_\rho(\bj_2)(\bj_2)}. \label{eq:Finsler:g:CSI}
\end{align}
Here, we also used that $\sqrt{ab}+\sqrt{cd}\leq \sqrt{(a+c)(b+d)}$ for all $a,b,c,d>0$.
Note that the above inequalities become strict if any of the integrands $j_2(x,y)_+ j(x,y)_-$ or $j_2(x,y)_- j(x,y)_+$ have a contribution. In particular, we could have $\ell_\rho(\bj)(\bj_2)=-\infty$ although the right-hand side is finite. Despite this, we still have equality in~\eqref{eq:Finsler:g:CSI} if and only if $\bj_2 = \beta \bj_1$ $\eta\, \mu \otimes \mu$-a.e.\ for some $\beta \geq 0$.  

To prove the injectivity of $\ell_\rho$, let us suppose that $\bj_1, \bj_2 \in T_\rho\mP_2(\Rd)$ are so that $\ell_\rho(\bj_1) = \ell_\rho(\bj_2)$. If $\bj_1 = 0$ or $\bj_2 = 0$ $\eta\, \mu \otimes \mu$-a.e., then $\ell_\rho(\bj_1) = \ell_\rho(\bj_2)$ implies that $\bj_1 = \bj_2 = 0$. If both $\bj_1$ and $\bj_2$ are nonzero, then by the above Cauchy--Schwarz inequality we get
\begin{equation*}
    0<g_{\rho,\bj_2}(\bj_2,\bj_2) = \ell_\rho(\bj_2)(\bj_2) = \ell_\rho(\bj_1)(\bj_2) = g_{\rho,\bj_1}(\bj_1,\bj_2) \leq \sqrt{ g_{\rho,\bj_1}(\bj_1,\bj_1)g_{\rho,\bj_2}(\bj_2,\bj_2) },
\end{equation*}
which, after dividing by $\sqrt{g_{\rho,\bj_2}(\bj_2,\bj_2)}$ yields $g_{\rho,\bj_2}(\bj_2,\bj_2) \leq g_{\rho,\bj_1}(\bj_1,\bj_1)$. Similarly, one gets $g_{\rho,\bj_1}(\bj_1,\bj_1) \leq g_{\rho,\bj_2}(\bj_2,\bj_2)$, from which we get
\begin{equation}
   g_{\rho,\bj_1}(\bj_1,\bj_1) = g_{\rho,\bj_2}(\bj_2,\bj_2).
\end{equation}
Hence
\begin{equation*}
    g_{\rho,\bj_1}(\bj_1,\bj_2) = \ell_\rho(\bj_1)(\bj_2) = \ell_\rho(\bj_2)(\bj_2) = g_{\rho,\bj_2}(\bj_2,\bj_2) = \sqrt{g_{\rho,\bj_1}(\bj_1,\bj_1)g_{\rho,\bj_2}(\bj_2,\bj_2)},
\end{equation*}
which is the equality case in the Cauchy--Schwarz inequality. Therefore, there exists $\beta\geq0$ such that $\bj_2 = \beta \bj_1$. By positive $1$-homogeneity of $\ell_\rho$ we get $\ell_\rho(\bj_2) = \ell_\rho(\beta \bj_1) = \beta \ell_\rho(\bj_1) = \beta \ell_\rho(\bj_2)$, so that $\beta = 1$, since $\ell_\rho(\bj_2)(\bj_2) \neq 0$. This ends the proof of the claim of injectivity of $\ell_\rho$.
\medskip

The direction of the steepest descent on Finsler manifolds is in general not $-\grad \cE(\rho)$, but is defined to be the tangent flux, which we denote by $\grad^- \cE(\rho)$, such that
\begin{equation} \label{eq:diff_gradd}
    -\Diff_\rho \cE[\bj] = g_{\rho,\grad^- \cE(\rho)\,}(\grad^- \cE(\rho), \bj) \qquad \text{for all $\bj \in T_\rho\mP_2(\Rd)$}.
\end{equation} 
In other words, we define $\grad^- \cE(\rho)$ as the tangent vector (provided it exists) such that
\begin{equation}\label{eq:ident:grad-}
    \ell_\rho(\grad^- \cE(\rho)) = -\Diff_\rho \cE.
\end{equation}
Here we clearly see that in general $\grad^- \cE(\rho) \neq -\grad \cE(\rho)$ since $\ell_\rho$ is not negatively $1$-homogeneous. We can justify that $\grad^- \cE(\rho)$ indeed corresponds to the direction of steepest descent at $\rho$ via the following criterion, which is analogous to the Riemann case. We first note that if $\Diff_\rho \cE = 0$ then 
$\grad^- \cE(\rho)=0$. If $\Diff_\rho \cE \neq 0$ we note that minimizers $\bj^*$ of
\begin{equation*}
   \bj \mapsto   \Diff_\rho \cE[\bj], \qquad \text{with the constraint that $g_{\rho,\bj}(\bj,\bj) = 1$},
\end{equation*}
are of the form $\bj^* = \beta \grad^- \cE(\rho)$ for some $\beta >0$. Indeed, using the fact that $\left.\frac{\di}{\di s}\right|_{s=0}g_{\rho,\bj + s\bj_1}(\bj +s\bj_1,\bj + s\bj_1) = 2g_{\rho,\bj}(\bj,\bj_1)$ for all $\bj,\bj_1\in T_\rho\mP_2(\Rd)$ (as shown in~\eqref{e:Minkowski:1st} of Appendix \ref{app:minkowski}) and using the Lagrange multiplier $\beta$ and the functional
\begin{equation*}
    H(\beta,\bj) := \Diff_\rho \cE[\bj] + \tfrac{\beta}{2} (g_{\rho,\bj}(\bj,\bj) -1), \qquad \bj \in T_\rho\mP_2(\Rd),\quad \beta \in \R,
\end{equation*}
yields, for a constrained minimizer $\bj^*$, the condition
\begin{equation}\label{eq:min-H}
   \Diff_\rho \cE = - \beta^* g_{\rho,\bj^*}(\bj^*,\cdot) = - \beta^* \ell_\rho(\bj^*). 
\end{equation}
By the definition of $\bj^*$ we have  $0> \Diff_\rho \cE[\bj^*]  = - \beta^* g_{\rho,\bj^*}(\bj^*,\bj^*)$, which implies  that $\beta^*>0$.
By injectivity and positive $1$-homogeneity of $\ell_\rho$, we get 
\begin{equation*}
    \bj^* = \ell_\rho^{-1}\left(-\frac{1}{\beta^*}\Diff_\rho \cE\right) = \frac{1}{\beta^*} \ell_\rho^{-1}(-\Diff_\rho \cE) = \frac{1}{\beta^*} \grad^- \cE(\rho).
\end{equation*}
The gradient flows with respect to $\cE$ in the Finsler space $(\mP_2(\Rd),\mathcal{T})$ can thus be written
\begin{equation}\label{eq:grad-descent-Finsler}
    \partial_t \rho_t = \dgrad \cdot \grad^- \cE(\rho).
\end{equation}
These considerations stay valid for general energy functionals $\cE\colon \mP_2(\Rd)\to \R$.

Let us compute the gradient flux for the specific case of the interaction energy \eqref{eq:inter-en-functional}. A direct computation using the symmetry of $K$ and Definition \ref{def:nl_grad_div} gives, for all $\bj \in T_\rho\mP_2(\Rd)$,
\begin{align*}
    &-\Diff_\rho \cE[\bj]
    = \frac12\iint_G \bigl(-\dgrad(K*\rho)\bigr)(x,y) \, \eta(x,y) \, j(x,y) \di\mu(x)\di\mu(y) \\
    &\phantom{=}= \frac12\iint_G j(x,y) \, \eta(x,y) \left( \frac{\rho(x) \bigl(-\dgrad(K*\rho)\bigr)_+(x,y)}{\rho(x)} -  \frac{\rho(y) \bigl(-\dgrad(K*\rho)\bigr)_-(x,y)}{\rho(y)} \right)  \di\mu(x)\di\mu(y)  \\
    &\phantom{=}=\frac12\iint_G j(x,y) \, \eta(x,y)\\
    &\phantom{==\frac12\iint_G}\bigl(-\dgrad (K*\rho)(x,y)\bigr) \left( \frac{\rho(x) \chi_{\{-\dgrad K*\rho>0)\}} (x,y)}{\rho(x)} + \frac{\rho(y) \chi_{\{-\dgrad K*\rho<0\}}(x,y)}{\rho(y)} \right) \di\mu(x)\di\mu(y)  \\
    &\phantom{=}= \ell_{\rho}\bigl(\grad^- \cE(\rho)\bigr)(\bj) ,
\end{align*}
where by comparison with~\eqref{eq:def:ell}, we observe that $\grad^- \cE(\rho)$ is given for $\mu \otimes\mu$-a.e.\ $(x,y) \in G$ by
\begin{equation}\label{eq:ident:grad-E}
    \grad^- \cE(\rho)(x,y) = -\dgrad(K*\rho)(x,y) \left(\rho(x)\chi_{\{-\dgrad K*\rho>0\}}(x,y) + \rho(y)\chi_{\{-\dgrad K*\rho<0\}}(x,y) \right) . 
\end{equation}
This shows by~\eqref{eq:ident:grad-} the existence and by our previous argument also uniqueness of $\grad^- \cE(\rho)$. It is easily observed that it has exactly the form \eqref{eq:tan_flux_vec} with the corresponding potential given by $\varphi = -K*\rho$.

We conclude this section by mentioning that the Finsler gradient flow structure of differential equations has been discovered and investigated in other systems; see \cite{Agueh, OhtaSturm09,OhtaSturm12}.
\subsection{Variational characterization for the nonlocal nonlocal-interaction equation} \label{sec:DeGiorgi}
Section \ref{subsec:HeuristicFinsler} shows that the nonlocal nonlocal-interaction equation \eqref{eq:nlnl-interaction-eq} can in fact be written as the gradient descent of the energy $\cE$ according to the Finsler gradient operator; see \eqref{eq:grad-descent-Finsler} and \eqref{eq:ident:grad-E}. This is why we refer to weak solutions of \eqref{eq:nlnl-interaction-eq} as gradient flows.

In this section we consider $(\mP_2(\Rd),\cT)$ as a quasi-metric space rather than a Finsler manifold, which allows us to prove rigorous statements more easily.
In particular, we show that the weak solutions of \eqref{eq:nlnl-interaction-eq} are 
curves of maximal slope for the energy \eqref{eq:inter-en-functional} in the 
quasi-metric space $(\mP_2(\Rd),\cT)$ and vice versa. We then establish the existence and stability of gradient flows using the variational framework of curves of maximal slope. 
To develop the variational formulation, we adapt the approach of \cite{AGS} to curves of maximal slope in metric spaces to the quasi-metric space $(\mP_2(\Rd),\cT)$.
This requires introducing a \textit{one-sided} version of the usual concepts from \cite{AGS} to cope with the asymmetry of the quasi-metric~$\cT$.
\begin{definition}[One-sided strong upper gradient]\label{def:strong_upper_gradient}
A function $h\colon\mP_2(\Rd)\to[0,\infty]$ is a \emph{one-sided strong upper gradient} for $\cE$ if for every $\rho\in \AC([0,T];(\mP_2(\Rd),\cT))$ the function $h\circ\rho$ is Borel and
\begin{equation}\label{eq:def-strong-upper-gr}
\cE(\rho_t)-\cE(\rho_s) \ge - \int_s^th(\rho_\tau)|\rho_\tau'|\di \tau \quad \mbox{\; for all $0\le s\le t\le T$},
\end{equation}
where $|\rho'|$ is the metric derivative of $\rho$ as defined in~\eqref{e:def:metric-derivative}. 
\end{definition} The above one-sided definition is sufficient to characterize the
curves of maximal slope:
\begin{definition}[Curve of maximal slope]\label{defn:cms}
    A curve $\rho \in \AC([0,T];\mP_2(\Rd))$ is a \emph{curve of maximal slope} for $\cE$ with respect to its one-sided strong upper gradient $h$ if and only if $t\mapsto \cE(\rho_t)$ is non-increasing and
\begin{equation}\label{eq:inequality_maximal_slope}
	\cE(\rho_t)-\cE(\rho_s)+\frac{1}{2}\int_s^t \Bigl( h(\rho_\tau)^2+|\rho_\tau'|^2 \Bigr) \di \tau \le 0 \quad \mbox{for all\; $0\le s\le t\le T$}.
\end{equation}
\end{definition}
\begin{rem}
Note that by using Young's inequality in \eqref{eq:def-strong-upper-gr}, we get
\[
\cE(\rho_t)-\cE(\rho_s)+\frac{1}{2}\int_s^t \Bigl( h(\rho_\tau)^2+|\rho_\tau'|^2 \Bigr) \di \tau \ge0 \quad \mbox{\; for all $0\le s\le t\le T$}.
\]
Hence, if the curve $(\rho_t)_{t\in[0.T]}$ is a curve of maximal slope for $\cE$ with respect to its strong upper gradient $h$, we actually have an equality in \eqref{eq:inequality_maximal_slope}.
\end{rem}
Therefore, in order to give a variational characterization of \eqref{eq:nlnl-interaction-eq} we need to detect the right one-sided strong upper gradient. As showed in \cite{ErbarBoltz}, the variation of the energy along the solution to the equation provides the suitable candidate. In the following we clarify this point as well as the strategy.

We recall that Proposition~\ref{prop:metric-velocity} ensures that for any $\rho \in \AC\bigl([0,T]; (\cP_2(\Rd),\cT)\bigr)$ there exists a unique flux $(\bj_t)_{t\in[0,T]}$ in $T_\rho\mP_2(\Rd)$ such that $\int_0^T\cA(\rho_t,\bj_t)\,\di{t}<\infty$, $(\rho,\bj)\in \CE_T$ and $|\rho_t'|^2=\cA(\rho,\bj_t)$ for a.e.\ $t\in[0,T]$. Moreover, according to Lemma~\ref{lem:action} there exists an antisymmetric measurable vector field $w\colon [0,T]\times G \to \R$ such that
\begin{equation}\label{eq:rel:j:w}
  \di\bj_t(x,y) = w_t(x,y)_+ \di\gamma_{1,t}(x,y) - w_t(x,y)_- \di\gamma_{2,t}(x,y).
\end{equation}
It will be convenient to work directly with this vector field $(w_t)_{t\in[0,T]}$: from now on we write $(\rho,w)\in \CE_T$ for $(\rho,\bj)\in \CE_T$ as well as $\widehat\cA(\rho_t,w_t)$ for $\cA(\rho_t,\bj_t)$ according to~\eqref{eq:action:vas}. With this convention, we can define a Finsler-type product on velocities in analogy to ~\eqref{eq:def:Finsler:g} as
\begin{equation}\label{eq:Finsler:g:velocities}
    \widehat g_{\rho,w}(u,v) = \frac12\iint_G u(x,y)\,v(x,y)\, \eta(x,y) \bra[\big]{\chi_{\{w>0\}}(x,y)\di\gamma_1(x,y)  + \chi_{\{w<0\}}(x,y) \di\gamma_2(x,y)}. 
\end{equation}
Note that, under the absolute-continuity assumptions of Section~\ref{subsec:HeuristicFinsler}, by comparing with~\eqref{eq:def:Finsler:g} we have that $\widehat g_{\rho,w}(u,v)= g_{\rho,\bj}(\bj_1,\bj_2)$, where $\bj_1,\bj_2$ are obtained from $u,v$ by~\eqref{eq:rel:j:w}, respectively. Moreover, taking \eqref{eq:def:ell} into account, we also define
\begin{equation}\label{eq:Finsler:ell:velocities}
  \widehat\ell_{\rho}(w)(v) = \widehat g_{\rho,w}(w,v) .
\end{equation}
Arguing as in~\eqref{eq:Finsler:g:CSI}, we arrive at the following one-sided Cauuchy--Schwarz inequality.
\begin{lemma}[One-sided Cauchy--Schwarz inequality]\label{lem:g:CSI}
 For all $v,w \in T_\rho \cP_2(\Rd)$ it holds
 \begin{equation}\label{eq:g:one-side:CauchySchwarz}
  \widehat g_{\rho,w}(w,v) \leq \sqrt{ \widehat g_{\rho,v}(v,v) \, \widehat g_{\rho,w}(w,w)},
 \end{equation}
 with equality if and only if, for some $\lambda>0$, $v(x,y)_+= \lambda w(x,y)_+$ for $\eta \, \rho\otimes\mu$-a.e.\ $(x,y)\in G$ (and thus, by antisymmetry, also $v(x,y)_-= \lambda w(x,y)_-$ for $\eta \, \mu\otimes\rho$-a.e.\ $(x,y)\in G$). 
\end{lemma}
\begin{proof}
Using $v=v_+-v_-$ and the usual Cauchy--Schwarz inequality in $L^2(\eta\,\rho\otimes\mu)$, we get
  \begin{align*}
    \widehat g_{\rho,w}(w,v) &= \frac12\iint_G v(x,y) \eta(x,y) \bigl( w(x,y)_+ \di\rho(x) \di\mu(y) - w(x,y)_- \di\mu(x) \di\rho(y)\bigr) \\
    &\leq \frac12\iint_G v(x,y)_+ w(x,y)_+ \eta(x,y) \di\rho(x) \di\mu(y) \\
    &\qquad + \frac12\iint_G v(x,y)_-w(x,y)_-\eta(x,y) \di\mu(x) \di\rho(y)\\
    &\leq \sqrt{ \widehat g_{\rho,v}(v,v) \, \widehat g_{\rho,w}(w,w)}.
  \end{align*}
From the usual Cauchy--Schwarz inequality we have equalities above if and only if there exists $\lambda > 0$ such that $v(x,y)_+=\lambda w(x,y)_+$ for $\eta\rho\otimes\mu$-a.e.\ $(x,y) \in G$ and $v(x,y)_-=\lambda w(x,y)_-$ for $\eta\mu\otimes\rho$-a.e.\ $(x,y)\in G$, since all the contributions are positive.
\end{proof}
Now note that, from the weak formulation of the nonlocal continuity equation~\eqref{eq:CE:weak:t01}, we have for any $\varphi \in C_\mathrm{c}^\infty(\Rd)$ and any $0\leq s < t \leq T$ the following chain rule:
\begin{align}
    \int_\Rd &\varphi(x) \di\rho_t(x)-\int_\Rd \varphi(x) \di\rho_s(x) =  \frac12\int_s^t\iint_G\dgrad\varphi(x,y)\,\eta(x,y) \di\bj_\tau(x,y)\di\tau \notag \\
    &= \frac12\int_s^t\iint_G\dgrad\varphi(x,y)\,\eta(x,y) \left( w_\tau(x,y)_+ \di\gamma_{1,\tau}(x,y) - w_\tau(x,y)_- \di\gamma_{2,\tau}(x,y) \right)\di \tau \notag \\
    &= \frac12\int_s^t \iint_G\dgrad\varphi(x,y)w_\tau(x,y)\,\eta(x,y) \left(\chi_{\{w>0\}}\di\gamma_{1,\tau}(x,y) + \chi_{\{w<0\}}\di\gamma_{2,\tau}(x,y)\right) \di \tau \notag \\
    &= \int_s^t \widehat g_{\rho_\tau,w_\tau}(w_\tau,\dgrad \varphi)\di \tau = \int_s^t \widehat\ell_{\rho}(w_\tau)(\dgrad\varphi) \di \tau. \label{eq:CE:Finsler}
\end{align}
Moreover, we still have the identification of the product $\widehat g$ with the action in the form of Lemma~\ref{lem:action}:
\begin{align}
    \widehat g_{\rho_t,w_t}(w_t,w_t) &= \frac12\iint_G w_t(x,y)^2\eta(x,y) \left( \chi_{\{w>0\}}(x,y)\di\gamma_{1,t}(x,y)  + \chi_{\{w<0\}}(x,y) \di\gamma_{2,t}(x,y )\right)  \notag \\
    &= \frac12\iint_G w_t(x,y)_+^2 \eta(x,y) \di\gamma_{1,t}(x,y)  +  \frac12\iint_G w_t(x,y)_-^2 \eta(x,y) \di\gamma_{2,t}(x,y)  \notag \\
    &= \frac12\iint_G \left( w_t(x,y)_+^2 + w_t(y,x)_-^2 \right)\eta(x,y) \di\gamma_{1,t}(x,y) =\hat\cA(\rho_t,w_t),  \label{eq:metric:action}
\end{align}
which shows that the action is the norm with respect to the Finsler structure.

A crucial step toward the variational characterization of \eqref{eq:nlnl-interaction-eq} mentioned above is to obtain the chain rule~\eqref{eq:CE:Finsler} for the energy functional~\eqref{eq:inter-en-functional}, which is done in Proposition~\ref{prop:chain-rule} below by a suitable regularization. As a consequence, by using the one-sided Cauchy--Schwarz inequality from Lemma~\ref{lem:g:CSI}, we obtain in Corollary \ref{cor:strongupper} that the square root $\sqrt{\cD}$ of the local slope, defined below in \eqref{eq:def:dissipation}, is a one-sided strong upper gradient for $\cE$ with respect to the quasi-metric~$\cT$ in the sense of Definition~\ref{def:strong_upper_gradient}, where $|\rho_t'|^2=\hat\cA(\rho_t,w_t)=\widehat g_{\rho_t,w_t}(w_t,w_t)$ for a.e.\ $t\in[0,T]$ due to Proposition \ref{prop:metric-velocity} and \eqref{eq:metric:action}. This allows us to define the De Giorgi functional, which provides the characterization of weak solutions as curves of maximal slope.
\begin{definition}[Local slope and De Giorgi functional]\label{defn:ls-Giorgi}
For any $\rho\in\mP_2(\Rd)$, let the \emph{local slope} at $\rho$ be given by
\begin{equation}\label{eq:def:dissipation}
    \cD(\rho) :=  \widehat g_{\rho,-\dgrad \frac{\delta \cE}{\delta \rho}}\bra*{-\dgrad \frac{\delta \cE}{\delta \rho},-\dgrad \frac{\delta \cE}{\delta \rho}}.
\end{equation}
For any $\rho \in \AC([0,T];(\mP_2(\Rd),\cT))$, the \emph{De Giorgi functional} at $\rho$ is defined as
\begin{equation}\label{eq:def:deGiorgi}
	\cG_T(\rho):=\cE(\rho_T)-\cE(\rho_0)+\frac{1}{2}\int_0^T\bra[\big]{\cD(\rho_\tau) + |\rho_\tau'|^2}\di \tau.
\end{equation}
When the dependence on the base measure $\mu$ needs to be explicit, the local slope and the De Giorgi functional are denoted by $\cD(\mu;\rho)$ and $\cG_T(\mu;\rho)$, respectively.
\end{definition}
If the potential $K$ satisfies Assumptions \ref{as:K:cont}--\ref{as:K:LipQuad}, we note that whenever $\rho$ is a weak solution to \eqref{eq:nlnl-interaction-eq} and $\rho\in \AC([0,T];\mP_2(\Rd))$ the quantity $\cG_T(\rho)$ is finite; indeed, the domain of the energy is all of $\mP_2(\Rd)$ and Proposition \ref{prop:metric-velocity} yields that both the local slope (since it is equal to the action of $(\rho,\bj)$, where $\bj$ is given in Definition \ref{defn:nl2ie}) and metric derivative are finite.

We are ready to state our main theorem.
\begin{theorem}\label{thm:CurveMaxSlope}
    Suppose that $\mu$ satisfies Assumptions \ref{it:as-conv} and \ref{it:as-tight} and $K$ satisfies Assumptions~\ref{as:K:cont}--\ref{as:K:LipQuad}. A curve $(\rho_t)_{t\in[0,T]} \subset \mP_2(\Rd)$ is a weak solution to \eqref{eq:nlnl-interaction-eq} according to Definition \ref{defn:nl2ie} if and only if $\rho$ belongs to $\AC([0,T];(\mP_2(\Rd),\cT))$ and is a curve of maximal slope for $\cE$ with respect to $\sqrt{\cD}$ in the sense of Definition \ref{defn:cms}, that is, satisfies
    \begin{equation}\label{eq:DeGiorgi_eq}
        \cG_T(\rho) = 0,
    \end{equation}
    where $\cG_T$ is the De Giorgi functional as given in Definition \ref{defn:ls-Giorgi}.
\end{theorem}
Note that in the above theorem, the implicit assumption that $\sqrt{\cD}$ is a one-sided strong upper gradient for $\cE$ is made; this is in fact true thanks to Corollary \ref{cor:strongupper} below. In light of this we can represent the result via the following diagram:
\[
\rho \text{ is a weak solution of } \eqref{eq:nlnl-interaction-eq} \!\!\iff\!\! \rho \text{ is a curve of maximal slope for } \cE  \text{ w.r.t. }\! \sqrt{\cD} \! \iff \! \cG_T(\rho)\! =\! 0.
\]
\subsection{The chain rule and proof of Theorem~\ref{thm:CurveMaxSlope}}
Firstly, we focus on the chain-rule property, which is the main technical step for proving Theorem~\ref{thm:CurveMaxSlope}.
\begin{proposition}\label{prop:chain-rule}
Let $K$ satisfy Assumptions~\ref{as:K:cont}--\ref{as:K:LipQuad}. For all $\rho\in \AC\bigl([0,T];(\mP_2(\Rd),\mt)\bigr)$ and $0\le s\le t\le T$ we have the chain-rule identity
  \begin{equation}\label{eq:chain_rule_int:vectorfield}
    \cE(\rho_t) - \cE(\rho_s) = \int_s^t \widehat g_{\rho_\tau,w_\tau}\bra*{w_\tau , \dgrad\frac{\delta \cE}{\delta\rho}(\rho_\tau)} \di\tau,
  \end{equation}
where $(w_t)_{t\in[0,T]}$ is the antisymmetric vector field associated by \eqref{e:bj:v} to $(\rho,\bj)\in \CE_T$.
\end{proposition}
\begin{proof}
Since the curve $\rho\in \AC\bigl([0,T];(\mP_2(\Rd),\mt)\bigr)$, according to Proposition \ref{prop:metric-velocity} there exists a unique family $(\bj_t)_{t\in[0,T]}$ belonging to $T_{\rho}\mP_{2}(\Rd)$ for a.e.\ $t\in[0,T]$ such that:
\begin{enumerate}
    \item[(i)] $(\rho,\bj)\in \CE_T$;
    \item[(ii)] $\int_0^T\sqrt{\cA(\rho_t,\bj_t)}\di t<\infty$;
    \item[(iii)] $|\rho_t'|^2=\cA(\rho_t,\bj_t)$ for a.e.\ $t\in[0,T]$;
    \item[(iv)] $\di\bj_t(x,y) = w_t(x,y)_+ \di\gamma_{1,t}(x,y) - w_t(x,y)_- \di\gamma_{2,t}(x,y)$.
\end{enumerate}
Then the identity~\eqref{eq:chain_rule_int:vectorfield} is equivalent to proving
\begin{equation}\label{eq:chain_rule_int}
    \cE(\rho_t)-\cE(\rho_s)  = \frac12\int_s^t\iint_G\dgrad\frac{\delta \cE}{\delta\rho}(\rho_\tau)(x,y)\, \eta(x,y) \di\bj_\tau(x,y)\di \tau  .
\end{equation}
We proceed by applying two regularization procedures. First, for all $(x,y)\in \Rd\times\Rd$ we define $K^\varepsilon(x,y)=K*m_\varepsilon(x,y)=\iint_{\Rd\times\Rd} K(z,z')m_{\varepsilon}(x-z,y-z')\di z\di z'$, where $m_\varepsilon(z)=\frac{1}{\varepsilon^{2d}}m(\frac{z}{\varepsilon})$ for all $z\in\R^{2d}$ and $\varepsilon>0$, where $m$ is a standard mollifier on $\R^{2d}$.  We also introduce a smooth cut-off function $\varphi_R$ on $\Rdd$ such that $\varphi(z)=1$ on $B_R$, $\varphi(z)=0$ on $\Rdd\setminus B_{2R}$ and $|\nabla\varphi_R|\le\frac{2}{R}$, where $B_R$ is the  ball of radius $R$ in $\Rdd$ centered at the origin.   
We set $K_R^\varepsilon:=\varphi_R K^\varepsilon$ and note that it is a~$C_\mathrm{c}^\infty(\Rdd)$ function. We now introduce the approximate energies, indexed by $\eps$ and $R$,
\[
\cE_R^\varepsilon(\nu)=\frac{1}{2}\int_\Rd\int_{\Rd} K_R^\varepsilon(x,y)\di\nu(y)\di\nu(x) \quad \mbox{for all $\nu\in\mP_2(\Rd)$}.
\]
Let us extend $\rho$ and $\bj$ to $[-T,2 T]$ periodically in time, meaning that $\rho_{-s}=\rho_{T-s}$ and $\rho_{T+s}=\rho_{s}$ for all $s\in (0,T]$ and likewise for $\bj$. We regularize  $\rho$ and $\bj$ in time by using a standard mollifier $n$ on $\R$ supported on $[-1,1]$, by setting $n_\sigma(t)=\frac{1}{\sigma}n(\frac{t}{\sigma})$ and 
\begin{align*}
    &\rho_t^\sigma(A)=n_\sigma*\rho_t(A)=\int_{-\sigma}^\sigma n_\sigma(t-s)\rho_s(A)\di s, \qquad \forall A\subseteq\Rd,\\
    &\bj_t^\sigma(U)=n_\sigma*\bj_t(A)=\int_{-\sigma}^\sigma n_\sigma(t-s)\bj_s(U)\di s, \qquad \forall U\subset G,
\end{align*}
for any $\sigma\in(0,T)$; whence $\rho_t^\sigma\in\mP_2(\Rd)$. Let us now show that the integral of the action is uniformly bounded with respect to $\sigma$. Let $|\lambda| \in \cM^+(G)$ be such that $\gamma_{1,t},\gamma_{2,t},|\bj_t| \ll |\lambda|$ for all $t\in[0,T]$. Then by using the joint convexity of the function $\alpha$ from~\eqref{eq:def:alpha}, Jensen's inequality and Fubini's Theorem, we get
\begin{align*}
    \int_0^T\cA(\rho_t^\sigma,\bj_t^\sigma) \di t&= 
    \frac12\int_0^T \iint_G  \alpha\bra*{ \int_{-\sigma}^{\sigma} \pderiv{\bj_{t-s}}{|\lambda|} n_\sigma(s)\di s, \int_{-\sigma}^{\sigma} \pderiv{\gamma_{1,t-s}}{|\lambda|} n_\sigma(s)\di s  } \eta \di|\lambda| \di t \\
    &\quad +\frac12\int_0^T \iint_G \alpha\bra*{ - \int_{-\sigma}^{\sigma} \pderiv{\bj_{t-s}}{|\lambda|} n_\sigma(s)\di s, \int_{-\sigma}^{\sigma} \pderiv{\gamma_{2,t-s}}{|\lambda|} n_\sigma(s)\di s  }\eta \di|\lambda| \di t
    \\
    &\leq \frac12\int_0^T \iint_G \int_{-\sigma}^\sigma \bra*{\alpha\bra*{ \pderiv{\bj_{t-s}}{|\lambda|}, \pderiv{\gamma_{1,t-s}}{|\lambda|} }
    + \alpha\bra*{ - \pderiv{\bj_{t-s}}{|\lambda|}, \pderiv{\gamma_{2,t-s}}{|\lambda|} }} n_\sigma(s) \di s \, \eta \di|\lambda| \di t  \\
    &= \int_{-\sigma}^{+\sigma} \int_0^T \cA(\rho_{t-s},\bj_{t-s})  \di t\, n_\sigma(s)\di s\\
    &\leq \int_{-T}^{2T} \cA(\rho_{t},\bj_{t})  \di t = 3 \int_{0}^{T} \cA(\rho_{t},\bj_{t})  \di t<\infty.
\end{align*}
It is easy to check that $(\rho^\sigma,\bj^\sigma)$ is still a solution to the nonlocal continuity equation on $[0,T]$. By arguing as in the proof of Proposition \ref{prop:compactness-sol-ce}, we get that along subsequences it holds  $\rho_t^\sigma\rightharpoonup\tilde{\rho}_t$ as $\sigma \to 0$ for all $t\in[0,T]$ for some curve $(\tilde \rho_t)_{t\in[0,T]}$ in $\mP_2(\Rd)$, and $\bj^\sigma\rightharpoonup\hat{\bj}$ in $\cM_{\mathrm{loc}}(G \times[0,T])$.
with $\di\hat\bj := \di\tilde \bj_t\di t$, for some curve $(\tilde \bj_t)_{t\in[0,T]}$ in $\cM(G)$.
Note that $n_\sigma\rightharpoonup\delta_0$ as $\sigma\to0$, and, as a consequence, $\rho_t^\sigma\rightharpoonup\rho_t$ for all $t\in[0,T]$ in the view of Proposition \ref{prop:comparison_with_W_1}. Thus, we actually have $\tilde{\rho}=\rho$ and $\tilde{\bj}=\bj$ by uniqueness of the limit and the flux, as highlighted above.
Using the regularity for~$\eps>0$ and $\sigma>0$, we get
\begin{equation}
    \pderiv{}{t} \cE_R^\varepsilon(\rho_t^\sigma)=\int_\Rd (K_R^\varepsilon*\rho_t^\sigma)(x)\partial_t\rho_t^\sigma(x)\di\mu(x)=\frac12\iint_G\dgrad(K_R^\varepsilon*\rho_t^\sigma)(x,y)\,\eta(x,y)\di\bj_t^\sigma(x,y).
\end{equation}
For the sake of completeness, we note that the second equality follows from the definition of $\CE_T$ by using again a cut-off argument on the function $K_R^\varepsilon*\rho_t^\sigma$. We omit this step as it is a standard procedure. By integrating in time between $s$ and $t$, with $s\le t$, it follows
\begin{align}\label{eq:chain_rule_int_approx}
    \cE_R^\varepsilon(\rho_t^\sigma)-\cE_R^\varepsilon(\rho_s^\sigma)&=\frac12\int_s^t\iint_G\dgrad(K_R^\varepsilon*\rho_\tau^\sigma)(x,y)\, \eta(x,y) \di\bj_\tau^\sigma(x,y)\di \tau\\
    &=\frac12\int_{s}^{t}\iint_G \int_\Rd \bra*{ K_R^\varepsilon(y,z)-K_R^\varepsilon(x,z)} \di\rho_\tau^\sigma(z) \eta(x,y) \di\bj_\tau^\sigma(x,y)\di \tau. \nonumber 
\end{align}
In order to obtain \eqref{eq:chain_rule_int} we need to let $\varepsilon$ and $\sigma$ go to $0$ and $R$ go to $\infty$ in \eqref{eq:chain_rule_int_approx}. The left-hand side is easy to handle since $\rho_t^\sigma\rightharpoonup\rho_t$ as $\sigma\to0$ for any $t\in[0,T]$, and $K_R^\varepsilon\to K_R$ uniformly on compact sets as $\varepsilon\to0$. Finally, by letting $R$ go to $\infty$ we have convergence to $\cE(\rho_t)$. 

In order to pass to the limit in the right-hand side of \eqref{eq:chain_rule_int_approx}, we use a truncation argument similar to that in the proof of Proposition~\ref{prop:compactness-sol-ce}. Let $\delta>0$ and let us set $N_\delta = \overline{B}_{\delta^{-1}} \times \overline{B}_{\delta^{-1}}$, where $B_{\delta^{-1}}= \set*{x \in \R^d: |x|< \delta^{-1}}$, and $G_\delta=\set[\big]{(x,y)\in G:\delta\le|x-y|}$. We can consider a family $(\varphi_\delta)_{\delta>0} \subset C_\mathrm{c}^\infty(\Rd \times G;[0,1])$ of truncation functions such that, for all $\delta>0$,
\[
  \set{\varphi_\delta = 1} \supseteq \overline{B}_{\delta^{-1}} \times G_\delta\cap N_\delta.
\]
Now, we add and subtract $\varphi_\delta$ in the integral on the RHS of \eqref{eq:chain_rule_int_approx} and we argue as follows. Since $\rho_t^\sigma\otimes\bj_t^\sigma\rightharpoonup\rho_t\otimes\bj_t$ for any $t\in[0,T]$ as $\sigma\to0$, and $K^\varepsilon_R\to K_R$ uniformly on compact sets as $\varepsilon\to0$, we can pass to the limit in $\sigma$ and $\varepsilon$, for any $R$ and $\delta>0$:
\begin{align}\label{eq:first-limit}
  & \frac12\int_{s}^{t} \iint_G \int_\Rd \varphi_\delta(z,x,y) \bra*{ K_R^\varepsilon(y,z)-K_R^\varepsilon(x,z)} \di\rho_\tau^\sigma(z) \eta(x,y)\di\bj_\tau^\sigma(x,y)\di\tau \notag\\
  &\to  \frac12 \int_{s}^{t} \iint_G \int_\Rd \varphi_\delta(z,x,y) \bra*{ K_R(y,z)-K_R(x,z)} \di\rho_\tau(z) \eta(x,y) \di\bj_\tau(x,y)\di\tau. 
\end{align}
By using $\varphi_\delta\le1$, Assumption \ref{as:K:LipQuad}, Lemma \ref{lem:A:TV:bound} with $\Phi(x,y)=|x-y|\vee |x-y|^2$ and \ref{it:as-conv}, we can bound the modulus of ~\eqref{eq:first-limit} for any $\tau\in[s,t]$ by
\begin{align}\label{eq:chain_rule:action_bound}
  &\frac{1}{2}\iint_G \int_\Rd  \frac{\abs{K_R(y,z)-K_R(x,z)}}{|x-y|\vee |x-y|^2}\di\rho_t(z) \bra*{|x-y|\vee |x-y|^2} \eta(x,y) \di|\bj_t|(x,y)
  \leq  L \sqrt{2C_\eta \, \cA(\rho_t,\bj_t)}.
\end{align}
Hence the integral is uniformly bounded in $\delta$ and $R$, and by the Lebesgue dominated convergence theorem we can pass to the limit in \eqref{eq:first-limit} in $\delta$ and $R$, obtaining
\[
\frac12 \int_{s}^{t} \iint_G \int_\Rd \bra*{ K(y,z)-K(x,z)} \di\rho_\tau(z) \eta(x,y) \di\bj_\tau(x,y)\di\tau.
\]
Now, it remains to control the integral involving the term $1-\varphi_\delta(z,x,y)$ in the integrand. Let us note that, for all $\delta>0$,
\[
  \bra*{\Rd \times G} \setminus \set{\varphi_\delta =1} \subseteq \bra[\big]{\overline{B}_{\delta^{-1}}^\mathrm{c} \times G } \cup \bra[\big]{ \Rd \times ( G\setminus (G_\delta \cap N_\delta))} =:  M_\delta.
\]
Using Assumption \ref{as:K:LipQuad} and splitting each contribution, we obtain
\begin{align*}
  &\abs*{\iint_G \int_\Rd \bra*{1-\varphi_\delta(z,x,y)} \bra*{ K_R^\varepsilon(y,z)-K_R^\varepsilon(x,z)} \di\rho_t^\sigma(z) \eta(x,y) \di\bj_t^\sigma(x,y)} \\
  &\leq L \iiint_{M_\delta}\bra*{|x-y|\vee |x-y|^2} \eta(x,y)\di\bj_t^\sigma(x,y)\di\rho^\sigma_t(z) \\
  &\leq L \iiint_{\overline{B}_{\delta^{-1}}^\mathrm{c} \times G}\bra*{|x-y|\vee |x-y|^2} \eta(x,y)\di\bj_t^\sigma(x,y)\di\rho^\sigma_t(z)\\
  &\ \qquad + 2L \int_{\Rd}\di\rho^\sigma_t(z)\iint_{G_\delta^\mathrm{c}}\bra*{|x-y|\vee |x-y|^2}w_t(x,y)_+ \eta(x,y)\di\rho_t^\sigma(x)\di\mu(y)\\
  &\ \qquad + 2L \int_{\Rd}\di\rho^\sigma_t(z)\iint_{N_\delta^\mathrm{c}}\bra*{|x-y|\vee |x-y|^2}w_t(x,y)_+ \eta(x,y)\di\rho_t^\sigma(x)\di\mu(y).
\end{align*}
Using Lemma \ref{lem:A:TV:bound} with $\Phi(x,y)=|x-y|\vee |x-y|^2$, \ref{it:as-conv} and the Cauchy--Schwarz inequality with respect to $\eta\, \rho_t^\sigma\otimes\mu$, the right-hand side in the inequality above can be further bounded by
\begin{align*}
   &4L \sqrt{ C_\eta \cA(\rho_t^\sigma,\bj_t^\sigma)} \  \rho_t^\sigma\bra*{\overline{B}_{\delta^{-1}}^\mathrm{c}}\\
  & + 2L \sqrt{\cA(\rho_t^\sigma,\bj_t^\sigma)} 
  \bra*{ \bra*{ \iint_{G_\delta^\mathrm{c}} \abs*{x-y}^2 \eta(x,y) \di\rho_t^\sigma(x) \di\mu(y)}^{\frac{1}{2}} +  \sqrt{C_\eta \rho_t^{\sigma}\bra*{\overline{B}_{\delta^{-1}}^\mathrm{c}}}}.  
\end{align*}
Thanks to the uniform second moment bound of $\rho_t^\sigma$ from Lemma~\ref{lem:CE:tightness} and Assumption~\ref{it:as-tight}, the above terms converge to zero as $\delta\to 0$, which concludes the proof.
\end{proof}
That $\sqrt{\cD}$ is a one-sided strong upper gradient for $\cE$ is an easy consequence of the previous result:
\begin{corollary}\label{cor:strongupper}
    For any curve $\rho\in \AC([0,T];(\mP_2(\Rd),\mt))$ it holds
    \begin{equation}\label{eq:strongupper}
        \cE(\rho_t)-\cE(\rho_s) \geq - \int_s^t\sqrt{\cD(\rho_\tau)}\,|\rho_\tau'| \di \tau \quad \mbox{for all\;\; $0\le s\le t\le T$},
    \end{equation}
i.e., $\sqrt{\cD}$ is a one-sided strong upper gradient for $\cE$ in the sense of Definition~\ref{def:strong_upper_gradient}.
\end{corollary}
\begin{proof}
Without loss of generality we assume $\int_s^t\sqrt{\cD(\rho_\tau)}|\rho'|(\tau)\di \tau<\infty$, as otherwise the inequality~\eqref{eq:strongupper} is trivially satisfied. We obtain the result as consequence of Proposition \ref{prop:chain-rule} by applying the one-sided Cauchy--Schwarz inequality (Lemma \ref{lem:g:CSI}) to \eqref{eq:chain_rule_int:vectorfield} as follows: for any $0\le s\le t\le T$,
\begin{align*}
    \cE(\rho_t) - \cE(\rho_s) &= \int_s^t \widehat g_{\rho_\tau,w_\tau}\bra*{w_\tau, \dgrad\frac{\delta \cE(\rho_\tau)}{\delta\rho}} \di\tau = - \int_s^t \widehat g_{\rho_\tau,w_\tau}\bra*{w_\tau,-\dgrad \frac{\delta \cE(\rho_\tau)}{\delta \rho}} \dx{\tau}\\
    &\ge-\int_s^t\sqrt{\widehat g_{\rho_\tau,-\dgrad \frac{\delta \cE(\rho_\tau)}{\delta\rho}}\left(-\dgrad \frac{\delta \cE(\rho_\tau)}{\delta\rho} ,-\dgrad \frac{\delta \cE(\rho_\tau)}{\delta\rho}\right)}\sqrt{\widehat g_{\rho_t,w_\tau}(w_\tau,w_\tau)}\,\di\tau\\
    &=\int_s^t\sqrt{\cD(\rho_\tau)}\,\sqrt{\hat\cA(\rho_\tau,w_t)} \di \tau \\
    &=\int_s^t\sqrt{\cD(\rho_\tau)}\,|\rho'|(\tau) \di \tau.
\end{align*}
Note that the last two equalities are provided by identity \eqref{eq:metric:action} and Proposition \ref{prop:metric-velocity}.
\end{proof}
At this point, we have collected all auxiliary results to deduce Theorem~\ref{thm:CurveMaxSlope}.
\begin{proof}[Proof of Theorem \ref{thm:CurveMaxSlope}]
Let us start by assuming that $\rho$ is a weak solution to \eqref{eq:nlnl-interaction-eq}. In view of Definition \ref{defn:nl2ie}, a weak solution is obtained from the weak formulation of the nonlocal continuity equation~\eqref{eq:nce-weak} if we set
\[
\di\bj_t(x,y)=\dgrad\frac{\delta \cE}{\delta\rho}(x,y)_- \di\rho_t(x)\di\mu(y)-\dgrad\frac{\delta \cE}{\delta\rho}(x,y)_+ \di\rho_t(y)\di\mu(x).
\]
Then, by writing $v_t^\cE(x,y)=-\dgrad\frac{\delta \cE}{\delta\rho}(x,y)$, it is easy to check
\[
  \cA(\rho_t,\bj_t)=\widehat\cA(\rho_t,v_t^\cE)=\cD(\rho_t) < \infty,
\]
where the finiteness follows from Assumptions~\ref{as:K:LipQuad} and~\ref{it:as-conv}, as shown by the computation
\begin{align*}
    \cD(\rho_t) &= \iint_G |(\dgrad K*\rho_t(x,y))_-|^2\eta(x,y)\di\rho_t(x)\di\mu(y)\\
    &\leq \iint_G (\dgrad K*\rho_t(x,y))^2\eta(x,y)\di\rho_t(x)\di\mu(y)\\
    &= \iint_G \left( \int_\Rd (K(x,z)-K(y,z))\di{\rho_t}(z) \right)^2 \eta(x,y)\di\rho_t(x)\di\mu(y)\\
    &\leq \iint_G \int_\Rd (K(x,z)-K(y,z))^2\di{\rho_t}(z) \eta(x,y)\di\rho_t(x)\di\mu(y)\\
    &\leq L^2 \int_\Rd \iint_G \left(|x-y|^2\vee|x-y|^4\right) \eta(x,y)\di\mu(y) \di\rho_t(x)\di{\rho_t}(z)\\
    &\leq L^2 C_\eta \int_\Rd \int_\Rd \di\rho_t(x)\di{\rho_t}(z) = L^2 C_\eta .
\end{align*}
Thanks to Proposition \ref{prop:metric-velocity}, this also proves that $\rho \in \AC([0,T];(\mP_2(\Rd),\cT))$ and $|\rho_t'|^2\le\cD(\rho_t)$ for a.e.\ $t\in[0,T]$.
In view of Proposition \ref{prop:chain-rule}, we thus obtain
\begin{align*}
\cE(\rho_t)-\cE(\rho_s)&= \int_s^t \widehat g_{\rho_\tau,v_\tau^\cE}\bra*{v_\tau^\cE,\dgrad\frac{\delta \cE}{\delta\rho}(\rho_\tau)} \di\tau 
= - \int_s^t \widehat g_{\rho_\tau,v_\tau^\cE}\bra*{v_\tau^\cE,-\dgrad\frac{\delta \cE}{\delta\rho}(\rho_\tau)} \di\tau \\
&=-\int_s^t\iint_G\left|\dgrad\frac{\delta \cE}{\delta\rho}(x,y)_-\right|^2\eta(x,y)\di\rho_\tau(x)\di\mu(y)\di\tau\\
&=-\int_s^t \cD(\rho_\tau)\di\tau\le-\int_s^t\sqrt{\cD(\rho_\tau)}|\rho'_\tau|\di\tau.
\end{align*}
This implies that:
\begin{itemize}
    \item[(i)] the map $t\mapsto \cE(\rho_t)$ is non-increasing;
    \item[(ii)] $\cE(\rho_t)-\cE(\rho_s)+\frac{1}{2}\int_s^t \cD(\rho_\tau)+|\rho_\tau'|^2\di \tau = 0$, by Corollary \ref{cor:strongupper}.
\end{itemize}
Whence the first part of the theorem follows for $s=0$ and $t=T$ since $\cG_T(\rho)=0$.

Consider now $\rho\in\AC([0,T];(\mP_2(\Rd),\cT))$ satisfying the equality \eqref{eq:DeGiorgi_eq}. Let us verify that it is a weak solution of~\eqref{eq:nlnl-interaction-eq}. By Proposition \ref{prop:metric-velocity} there exists a unique family $(\bj_t)_{t\in[0,T]}$ in $T_{\rho_t}\mP_2(\Rd)$ 
such that $(\rho,\bj)\in \CE_T$, $\int_0^T\sqrt{\cA(\rho_t,\bj_t)}\di t<\infty$ and $|\rho_t'|^2=\cA(\rho_t,\bj_t)$ for a.e.\ $t\in[0,T]$. Moreover, by Lemma~\ref{lem:action} we find an antisymmetric measurable vector field $w\colon [0,T]\times G \to \R$ such that
\begin{equation*}
  \di\bj_t(x,y) = w_t(x,y)_+ \di\gamma_{1,t}(x,y) - w_t(x,y)_- \di\gamma_{2,t}(x,y).
\end{equation*}
Thanks to Proposition \ref{prop:chain-rule}, by applying the one-sided Cauchy--Schwarz, using the identification~\eqref{eq:metric:action}, the definition of the local slope~\eqref{eq:def:dissipation} and Young inequality, we get
\begin{align*}
    \cE(\rho_T)-\cE(\rho_0)&=\int_0^T \widehat g_{\rho_\tau,w_\tau}\bra*{w_\tau , \dgrad\frac{\delta \cE}{\delta\rho}(\rho_\tau) } \di\tau = - \int_0^T \widehat g_{\rho_\tau,w_\tau}\bra*{w_\tau , -\dgrad\frac{\delta \cE}{\delta\rho}(\rho_\tau) } \di\tau \\ 
    &\ge-\int_0^T\sqrt{\cD(\rho_\tau)}\sqrt{\cA(\rho_\tau,\bj_\tau)}\di \tau =-\int_0^T\sqrt{\cD(\rho_\tau)}|\rho_\tau'|\di \tau\\&
    \ge-\frac{1}{2}\int_0^T \cD(\rho_\tau)\di \tau - \frac{1}{2} \int_0^T |\rho_\tau'|^2\di \tau .
\end{align*}
Thanks to the equality \eqref{eq:DeGiorgi_eq}, we actually have that the above inequalities are equalities, which holds if and only if $w_t(x,y)=-\dgrad\frac{\delta \cE}{\delta\rho}(x,y)$ for a.e.\ $t\in[0,T]$ and $\gamma_{1,t}$-a.e.\ $(x,y)\in G$. Hence $(\rho,\bj)\in\CE_T$ with $w=-\dgrad\frac{\delta \cE}{\delta\rho}$, that is, $\rho$ is a weak solution to~\eqref{eq:nlnl-interaction-eq}.
\end{proof}
\subsection{Stability and existence of weak solutions} \label{sec:stab-exist}
Theorem~\ref{thm:CurveMaxSlope} provides a characterization of (weak) solutions to~\eqref{eq:nlnl-interaction-eq} as minimizers of $\cG_T$ attaining the value $0$. The direct method of calculus of variations gives existence of minimizers of $\cG_T$. However, it is not clear a priori whether they attain the value $0$ and are thus actually weak solutions to~\eqref{eq:nlnl-interaction-eq}. Hence we prove compactness and stability of gradient flows (see Theorem \ref{thm:graph:stability}) and approximate the desired problem by discrete problems for which the existence of solutions is easy to show; see the proof of Theorem \ref{thm:existence-weak}. We start by proving that the local slope $\cD$ is narrowly lower semicontinuous jointly in its arguments, $\mu$ and $\rho$; see Lemma \ref{lem:lsc-dissipation}. We then establish the compactness coming from a uniform control of the De Giorgi functional $\cG_T$, as well as its joint narrow lower semicontinuity (see Lemma \ref{lem:lsc-Giorgi}), which we prove using compactness in $\CE_T$ and the joint narrow lower semicontinuity of the action (see Proposition \ref{prop:compactness-sol-ce}) and of the local slope. (See also \cite[Theorem 2]{Serfaty} for an analogous strategy.)

In Theorem \ref{thm:graph:stability} we prove one of our main results, namely that the functional $\cG_T$ is stable under variations in base measures, defining the vertices of the graph, and absolutely continuous curves. A particular consequence of this theorem is that weak solutions to~\eqref{eq:nlnl-interaction-eq} with respect to graphs defined by random samples of a measure $\mu$ converge to weak solutions to~\eqref{eq:nlnl-interaction-eq} with respect to $\mu$; see Remark~\ref{rem:DC-data}. 

The existence of weak solutions of~\eqref{eq:nlnl-interaction-eq} (and thus gradient flows) with respect to $\cE$ proved in Theorem \ref{thm:existence-weak} shows that, indeed, the De Giorgi functional~\eqref{eq:def:deGiorgi} corresponding to an interaction potential $K$ satisfying \ref{as:K:cont}--\ref{as:K:LipQuad} admits a minimizer when $\mu(\Rd)$ is finite.
\begin{lemma}\label{lem:lsc-dissipation}
    Let $(\mu^n)_n\subset\cM^+(\Rd)$ and suppose that $(\mu^n)_n$ narrowly converges to $\mu$. Assume that the base measures $(\mu^n)_n$ and $\mu$ are such that~\ref{it:as-conv} and~\ref{it:as-tight} hold uniformly in $n$, and let $K$ satisfy Assumptions~\ref{as:K:cont}--\ref{as:K:LipQuad}. Let moreover $(\rho^n)_n$ be a sequence such that $\rho^n \in \mP_{2}(\Rd)$ for all $n\in\mathbb{N}$ and $\rho^n\rightharpoonup \rho$ as $n\to\infty$ for some $\rho\in \mP_2(\Rd)$.
    Then
    \begin{equation*}
        \liminf_{n\to\infty} \cD(\mu^n;\rho^n) \geq \cD(\mu;\rho) .
    \end{equation*}
\end{lemma}
\begin{proof}
For every $n\in\mathbb{N}$ we set $u^n = \dgrad K*\rho^n$. Furthermore, we write $u= \dgrad K*\rho$ and define $g\colon \R \to \R$ by $g(x) = (x_+)^2$ for all $x \in \R$. Then note that $g$ is convex and continuous, and
    \begin{equation}
        \cD(\mu^n;\rho^n) = \iint_G g(u^n(x,y)) \eta(x,y) \di\rho^n(x) \di \mu^n(y),
    \end{equation}
    and, similarly, 
    \begin{equation}
        \cD(\mu;\rho) = \iint_G g(u(x,y)) \eta(x,y) \di\rho(x)\di \mu(y).
    \end{equation}
    We want to use \cite[Theorem 5.4.4 (ii)]{AGS} to prove the desired $\liminf$ inequality. Observe that $u^n \in L^2(\eta\,\gamma_1^n)$ and $u \in L^2(\eta\,\gamma_1)$; indeed, \ref{as:K:LipQuad} and \ref{it:as-conv} give
    \begin{align*}
        \iint_G u^n(x,y)^2 \eta(x,y)\di\gamma_1^n(x,y) &= \iint_G (K*\rho^n(y) - K*\rho^n(x))^2 \eta(x,y)\di\gamma_1^n(x,y)\le L^2 C_\eta,
    \end{align*}
    and, similarly, for $u$. Let now $\varphi \in C_\mathrm{c}^\infty(G)$. We have
    \begin{align*}
        \MoveEqLeft{\iint_G u^n(x,y)\varphi(x,y)\eta(x,y) \di \gamma_1^n(x,y)}\\
        &= \iint_G \left( \int_\Rd K(y,z)\di\rho^n(z) - \int_\Rd K(x,z)\di\rho^n(z)\right) \varphi(x,y)\eta(x,y) \di \gamma_1^n(x,y)\\
        &= \iint_G \int_\Rd ( K(y,z) - K(x,z) )\varphi(x,y)\eta(x,y) \di (\rho^n\otimes\gamma_1^n)(z,x,y)\\
        &=\iint_{\supp\varphi}\int_{\Rd\cap B_R}( K(y,z) - K(x,z) )\varphi(x,y)\eta(x,y) \di (\rho^n\otimes\gamma_1^n)(z,x,y)\\
        &\quad + \iint_{\supp\varphi}\int_{\Rd\setminus B_R}( K(y,z) - K(x,z) )\varphi(x,y)\eta(x,y) \di (\rho^n\otimes\gamma_1^n)(z,x,y).
    \end{align*}
The last integral is actually vanishing as $R\to\infty$ since \ref{as:K:LipQuad}, \ref{it:as-conv} and Prokhorov's Theorem give
\begin{align*}
    \MoveEqLeft{\left|\iint_{\supp\varphi}\int_{\Rd\setminus B_R}(K(y,z) - K(x,z) )\varphi(x,y)\eta(x,y) \di (\rho^n\otimes\gamma_1^n)(z,x,y)\right|}\\
    &\le\frac{L\|\varphi\|_{\infty}\rho^n(\Rd\setminus B_R)}{\inf_{\supp\varphi}(|x-y|\vee|x-y|^2)}\iint_{\supp\varphi}(|x-y|^2\vee|x-y|^4)\eta(x,y)\,\di\mu^n(y)\,\di\rho^n(x)\\
    &\le\frac{LC_\eta\|\varphi\|_{\infty}\rho^n(\Rd\setminus B_R)}{\inf_{\supp\varphi}(|x-y|\vee|x-y|^2)}\underset{R\to\infty}{\longrightarrow}0.
\end{align*}
The function $(z,x,y) \mapsto (K(y,z) - K(x,z))\varphi(x,y)\eta(x,y)$ is continuous and bounded on $(\Rd\cap B_R)\times G$ thanks to Assumption~\eqref{it:as:pos-sym-lsc}. In addition, we note that $(\rho^n\otimes\gamma_1^n)_n$ narrowly converges to $\rho\otimes\gamma_1$ in $\cP(\Rd)\times\cM^+(G)$. Therefore, we obtain for any $R>0$ the convergence
\begin{align*}
 \lim_{n\to \infty}\iint_{\supp\varphi}\int_{\Rd\cap B_R} &( K(y,z) - K(x,z) )\varphi(x,y)\eta(x,y) \di (\rho^n\otimes\gamma_1^n)(z,x,y)\\
 &=\iint_{\supp\varphi}\int_{\Rd\cap B_R}( K(y,z) - K(x,z) )\varphi(x,y)\eta(x,y) \di (\rho\otimes\gamma_1)(z,x,y) .
\end{align*}
By sending $R\to\infty$, we obtain
    \begin{align*}
        \MoveEqLeft{\lim_{n\to\infty} \iint_G u^n(x,y)\varphi(x,y) \eta(x,y)\di \gamma_1^n(x,y)}\\
        &= \iint_G \int_\Rd ( K(y,z) - K(x,z) )\varphi(x,y)\eta(x,y) \di (\rho\otimes\gamma_1)(z,x,y)\\
        &= \iint_G u(x,y)\varphi(x,y) \eta(x,y)\di \gamma_1(x,y).
    \end{align*}
    Thus, $u^n$ converges weakly to $u$ as $n\to\infty$ in the sense of \cite[Definition 5.4.3]{AGS}. By \cite[Theorem 5.4.4 (ii)]{AGS} we therefore conclude that
    \begin{align*}
         \liminf_{n\to\infty} \cD(\mu^n;\rho^n)&= \liminf_{n\to\infty} \iint_G g(u^n(x,y)) \eta(x,y) \di\rho^n(x) \di \mu^n(y)\\
         &\geq \iint_G g(u(x,y)) \eta(x,y) \di\rho(x)\di \mu(y) = \cD(\mu;\rho) ,
    \end{align*}
    which is the desired result.
\end{proof}
Let us also prove the compactness and narrow lower semicontinuity of the De Giorgi functional.
\begin{lemma}[Compactness and lower semicontinuity of the De Giorgi functional]\label{lem:lsc-Giorgi}
    Let $(\mu^n)_n\subset\cM^+(\R^d)$ and suppose that $(\mu^n)_n$ narrowly converges to $\mu$. Assume that the base measures $\mu^n$ and $\mu$ satisfy~\ref{it:as-conv} and~\ref{it:as-tight} uniformly in $n$, and let $K$ satisfy~\ref{as:K:cont}--\ref{as:K:LipQuad}. Let moreover $(\rho^n)_n$ be a sequence so that $\rho^n \in \AC([0,T];(\mP_2(\Rd),\cT_{\mu^n}))$ for all $n\in\N$ with $\sup_{n\in\N} M_2(\rho_0^n) < \infty$ and $\sup_{n\in \N} \cG_T(\mu^n;\rho^n)<\infty$. Then, up to a subsequence, $\rho^n_t \rightharpoonup \rho_t$ as $n\to\infty$ for all $t\in[0,T]$ for some $\rho \in \AC([0,T];(\mP_2(\Rd),\cT_\mu))$ and
     \begin{equation*}
        \liminf_{n\to\infty} \cG_T(\mu^n;\rho^n) \geq \cG_T(\mu;\rho).
    \end{equation*}
\end{lemma}
\begin{proof}
    For any $n\in\N$, recall the definition
    \[
        \cG_T(\mu^n;\rho^n) = \cE(\rho^n_T) - \cE(\rho^n_0) + \frac12 \int_0^T \cD(\mu^n; \rho^n_t) \di t + \frac12 \int_0^T |(\rho_t^n)'|_{\mathcal{T}_{\mu^n}}^2 \di t,
    \]
    where we are careful to take the metric derivative of $\rho^n$ with respect to $\mathcal{T}_{\mu^n}$ (as given in Definition~\ref{defn:metric}). Since the domain of the energy $\cE$ is all of $\mP_2(\Rd)$ and the local slope $\cD$ is non-negative, the bound $\sup_{n\in \N} \cG_T(\mu^n;\rho^n)<\infty$ ensures that
    \[
        \sup_{n\in\N} \int_0^T |(\rho^n_t)'|_{\cT_{\mu^n}}^2 \di t < \infty.
    \]
    For all $n\in\N$, since $\rho^n \in \AC([0,T];(\mP_2(\Rd),\cT_{\mu^n}))$, Proposition \ref{prop:metric-velocity} yields the existence of a flux $\bj^n$ such that $(\rho^n,\bj^n)\in \CE_T$ and $|(\rho^n_t)'|^2 = \cA(\mu^n;\rho^n_t,\bj^n_t)$ for almost all $t\in[0,T]$. We then get
    \[
        \sup_{n\in\N} \int_0^T \cA(\mu^n;\rho^n_t,\bj^n_t) \di t = \sup_{n\in\N} \int_0^T |(\rho^n_t)'|_{\cT_{\mu^n}}^2 \di t < \infty.
    \]
    By Proposition \ref{prop:compactness-sol-ce}, there now exists $(\rho,\bj) \in \CE_T$ such that, up to subsequences, $\rho_t^n \rightharpoonup \rho_t$ for all $t\in[0,T]$ and $\bj^n \rightharpoonup \bj$ as $n\to\infty$, and
    \begin{equation*}
         \infty > \liminf_{n\to\infty} \int_0^T \cA(\mu^n;\rho^n_t, \bj^n_t) \di t \geq \int_0^T \cA(\mu;\rho_t,\bj_t) \di t.
    \end{equation*}
    By Proposition \ref{prop:metric-velocity}, we therefore have $\rho \in \AC([0,T];(\mP_2(\Rd),\cT_\mu))$ and $|(\rho_t)'|_{\cT_\mu}^2 \leq \cA(\mu;\rho_t,\bj_t)$ for almost all $t\in[0,T]$, which finally gives
    \begin{equation}\label{eq:metric-der-lsc}
        \liminf_{n\to\infty} \int_0^T |(\rho^n_t)'|^2_{\cT_{\mu^n}} \di t \geq \int_0^T |\rho_t'|_{\cT_\mu} \di t. \qedhere
    \end{equation}
    By the narrow continuity of the energy proved in Proposition \ref{prop:cont-energy}, we get
    \begin{equation}\label{eq:energy-cont}
        \lim_{n\to\infty} \cE(\rho^n_T) = \cE(\rho_T) \quad \text{and} \quad \lim_{n\to\infty} \cE(\rho_0^n) = \cE(\rho_0).
    \end{equation}
    Furthermore, by Fatou's lemma and the narrow lower semicontinuity of the local slope shown in Lemma \ref{lem:lsc-dissipation}, we have
    \begin{equation}\label{eq:dissipation-lsc}
        \liminf_{n\to\infty} \int_0^T \cD(\mu^n;\rho^n_t) \di t \geq \int_0^T \cD(\mu;\rho_t) \di t.
    \end{equation}  
    Gathering \eqref{eq:metric-der-lsc}, \eqref{eq:energy-cont} and \eqref{eq:dissipation-lsc}, we finally obtain
    \[
        \liminf_{n\to\infty} \cG_T(\mu^n;\rho^n) \geq \cE(\rho_T) - \cE(\rho_0) + \frac12 \int_0^T \cD(\mu;\rho_t) \di t + \frac12 \int_0^T |\rho_t'|_{\cT_\mu}^2 \di t = \cG_T(\mu;\rho),
    \]
    which ends the proof.
\end{proof}
We now get our stability result.
\begin{theorem}[Stability of gradient flows]\label{thm:graph:stability}
     Let $(\mu^n)_n\subset\cM^+(\R^d)$ and suppose that $(\mu^n)_n$ narrowly converges to $\mu$. Assume that the base measures $\mu^n$ and $\mu$ satisfy~\ref{it:as-conv} and~\ref{it:as-tight} uniformly in $n$, and let the interaction potential $K$ satisfy~\ref{as:K:cont}--\ref{as:K:LipQuad}. Suppose that $\rho^n$ is a gradient flow of $\cE$ with respect to $\mu^n$ for all $n\in\N$, that is,
  \begin{equation*}
      \cG_T(\mu^n;\rho^n) = 0 \quad \mbox{for all $n\in\N$},
  \end{equation*}
  such that $(\rho_0^n)_n$ satisfies $\sup_{n\in \N} M_2(\rho_0^n)< \infty$ and $\rho_t^n \rightharpoonup \rho_t$ as $n\to\infty$ for all $t\in[0,T]$ for some curve $(\rho_t)_{t\in[0,T]} \subset \mP_2(\Rd)$. Then, $\rho \in \AC([0,T];(\mP_2(\Rd),\cT_\mu))$ and $\rho$ is a gradient flow of $\cE$ with respect to $\mu$, that is,
    \begin{equation*}
      \cG_T(\mu;\rho) = 0.
  \end{equation*}
\end{theorem}
\begin{proof}
    By Lemma \ref{lem:lsc-Giorgi} we directly obtain that $\rho \in \AC([0,T];(\mP_2(\Rd),\cT_\mu))$ and, up to a subsequence,
    \[
        0 = \liminf_{n\to\infty} \cG_T(\mu^n;\rho^n) \geq \cG(\mu;\rho).
    \]
    Finally, since $\cG_T(\mu;\rho) \geq 0$ by Young's inequality and Corollary \ref{cor:strongupper}, we obtain $\cG_T(\mu;\rho) = 0$.
\end{proof}
Note that, via Theorem \ref{thm:CurveMaxSlope}, the above theorem also shows stability of weak solutions to~\eqref{eq:nlnl-interaction-eq}. Typically, in Theorem \ref{thm:graph:stability}, $(\mu^n)_n$ is a sequence of atomic measures used to approximate, or sample, the support of $\mu$. Indeed, we  now use this approach to show the existence of weak solutions to the nonlocal nonlocal-interaction equation.
\begin{theorem}[Existence of weak solutions]\label{thm:existence-weak}
    Let $K$ be an interaction potential satisfying Assumptions \ref{as:K:cont}--\ref{as:K:LipQuad}. Suppose that $\mu \in \cM^+(\R^d)$ is finite, i.e., $\mu(\Rd)<\infty$, and satisfies~\ref{it:as-tight}. Assume furthermore that for some $C_\eta' > 0$ it holds that
    \begin{equation}\label{eq:stronger-eta}
       \sup_{(x,y) \in G \cap \supp \mu \otimes \mu} \bra*{ |x-y|^2 \vee |x-y|^4 } \eta(x,y) \leq C_\eta'.
    \end{equation}
Consider $\rho_0 \in \mP_2(\Rd)$ which is $\mu$-absolutely continuous. Then there exists a weakly continuous curve $\rho\colon [0,T] \to \cP(\Rd)$ such that $\supp\rho_t\subseteq \supp\mu$ for all $t\in[0,T]$, which is a weak solution of~\eqref{eq:nlnl-interaction-eq} and satisfies the initial condition $\rho(0)=\rho_0$.
\end{theorem}
\begin{proof}
Let $(\mu^n)_n \subset \cM^+(\R^d)$ be a sequence of atomic measures such that $(\mu^n)_n$ converges narrowly to $\mu$. Moreover, assume that $\mu^n$ has finitely many atoms and $\mu^n(\Rd) \leq \mu(\Rd)$ and $\supp\mu^n \subseteq \supp\mu$ for all $n\in\N$.
Let $\hat \mu^n$ be the normalization of $\mu^n$ which has the same total mass as $\mu$, that is,
\[ 
  \hat \mu^n = \frac{\mu(\R^d)}{\mu^n(\R^d)} \, \mu^n , 
\]
and let $\pi^n$ be optimal transportation plan between $\mu$ and $\hat \mu^n$ for the quadratic cost. Let $\rho_0^n$ be the second marginal of $\tilde \rho_0 \pi^n$, where $\tilde \rho_0$ is the density of the measure $\rho_0$ with respect to $\mu$; namely, let $\rho_0^n(A) = \int_{\R^d \times A} \tilde \rho_0(x) \di\pi^n(x,y)$ for any Borel set $A\subset \Rd$. 
Note that $\rho_0^n(\R^d) = \rho_0(\R^d)$ and $\rho_0^n \ll \mu^n$ for all $n\in\N$, and that, since $\tilde \rho_0 \pi^n$ is a transport plan between $\rho_0$ and $\rho_0^n$, $\rho_0^n \rightharpoonup \rho_0$ as $n\to\infty$.   

Thanks to Assumption~\eqref{eq:stronger-eta}, it holds, for all $n\in\N$,
\begin{equation}\label{it:as-conv:mu}
  \esssupmu_{x\in \R^d} \int \bra{ |x-y|^2 \vee |x-y|^4 } \eta(x,y) \di \mu^n(y) \leq \mu^n(\Rd) C_\eta' \leq \mu(\Rd) C_\eta' .
\end{equation}
Since by construction $\rho_0^n \ll \mu^n$, we have $\supp \rho_0^n \subseteq \supp \mu^n \subseteq \supp \mu$. This nested support property is, thanks to Proposition~\ref{prop:supp-in-mu}, preserved in time, so that $\supp\rho_t^n \subseteq \supp\mu$ for all $t\in[0,T]$ and $n\in\N$. For this reason, \eqref{it:as-conv:mu} can be used, under the stated support restriction on $\rho_0$, instead of Assumption~\ref{it:as-conv} uniformly in $n$ when calling Lemma \ref{lem:lsc-Giorgi} and Theorem \ref{thm:graph:stability} later in this proof. Since $\mu^n$ consists of finitely many atoms and $\mu$ satisfies~\ref{it:as-tight}, the family $(\mu_n)_n$ satisfies~\ref{it:as-tight} uniformly in $n$.
    
By Remark \ref{rem:existence-strong-discrete}, we know that the ODE system~\eqref{eq:intro:CE}--\eqref{eq:intro:vE} admits a unique solution for all $n\in\N$. It can be easily checked that this solution, which we denote by $\rho^n$, is a weak solution to \eqref{eq:nlnl-interaction-eq} with respect to $\mu^n$ starting from $\rho_0^n$, according to Definition \ref{defn:nl2ie}. 
By Theorem \ref{thm:CurveMaxSlope}, we then get that~$\rho^n$ is a gradient flow of $\cE$ with respect to $\mu$ starting from $\rho_0^n$ for all $n\in\N$.
    
Combining the compactness part of Lemma \ref{lem:lsc-Giorgi} and the stability from Theorem \ref{thm:graph:stability}, we get that, up to a subsequence, $\rho_t^n \rightharpoonup \rho_t$ as $n\to\infty$ for all $t\in[0,T]$, where $\rho \in \AC([0,T];(\mP_2(\Rd),\cT_\mu))$ is a gradient flow of $\cE$ with respect to $\mu$ starting from $\rho_0$. Theorem \ref{thm:CurveMaxSlope} finally shows that $\rho$ is a weak solution to \eqref{eq:nlnl-interaction-eq} with respect to $\mu$ starting from $\rho_0$.
\end{proof}
\begin{rem}
 Assumption~\eqref{eq:stronger-eta} is only needed to arrive at an atomic approximation sequence~$(\mu^n)_n$ of $\mu$ such that Assumptions~\ref{it:as-conv} and~\ref{it:as-tight} hold uniformly in $n$. On a case-by-case basis, one could drop~\eqref{eq:stronger-eta} and try to construct the sequence $(\mu^n)_n$ explicitly in such a way as to satisfy both assumptions uniformly in $n$. 
\end{rem}
\begin{rem} \label{rem:DC-data}
We conclude the section by remarking on the relevance of the Theorem \ref{thm:graph:stability} to the setting of machine learning. Namely, there $\mu$ is the measure modeling the true data distribution, which can be assumed to be compact. Let $(x_i)_i$ be a sequence of i.i.d.\ samples of $\mu$ and let $\mu^n  = \frac{1}{n} \sum_{i=1}^n \delta_{x_i}$ be the empirical measure of the first $n$ sample points. Assume $(\rho^n)_n$ is a narrowly converging sequence of probability measures such that $\supp \rho^n \subseteq \{x_1, \dots, x_n\}$ for all $n\in\N$, and denote by $\rho$ its limit. Assume that $\eta$ is an edge weight kernel such that $\mu$ and $\eta$ satisfy~\ref{it:as-tight} and~ \eqref{eq:stronger-eta}. Let $K$ be an interaction kernel satisfying~\ref{as:K:sym} and \ref{as:K:LipQuad}. Finally, let $(\tilde \rho^n)_n$ be the sequence of solutions of \eqref{eq:nlnl-interaction-eq} in the sense of Definition \ref{defn:nl2ie} such that $\tilde \rho^n_0 = \rho^n$ for all $n\in\N$. Then, by Lemma \ref{lem:lsc-Giorgi}, the sequence~$(\tilde \rho_t^n)_n$ narrowly converges along a subsequence for all $t\in[0,T]$, and furthermore, by Theorem~\ref{thm:existence-weak}, any curve $(\tilde \rho_t)_{t\in[0,T]}$ of subsequential limits yields a solution $\tilde \rho$ of \eqref{eq:nlnl-interaction-eq} with initial condition $\rho$. 
\end{rem}
\subsection{Discussion of the local limit}\label{subsec:local_limit}
Here we discuss at a formal level the connection between the nonlocal nonlocal-interaction equation and its limit as the graph structure localizes. 
    We first present a very formal justification as to why we expect the solutions of 
    \eqref{eq:nlnl-interaction-eq} to converge to the solutions of a nonlocal-interaction equation as the localizing parameter $\varepsilon \to 0^+$, i.e., as the edge-weight function $\eta = \eta_\varepsilon$ localizes.
  We conclude this section with an example that cautions that the formal argument cannot be justified in full generality. Proving the convergence of \eqref{eq:nlnl-interaction-eq} in the limit $\eps \to 0^+$, under appropriate conditions, remains an intriguing open problem.
    
  Take $\mu=\operatorname{Leb}(\Rd)$ and choose $\eta_\varepsilon$ given by~\eqref{eq:def:eta:local}.
    Consider a smooth interaction potential $K\colon \Rd \times \Rd \to \R$ and a compactly supported initial condition $\rho_0$ which has a continuous density with respect to $\mu$. Let $\rho^\eps$ be the solution of \eqref{eq:nlnl-interaction-eq} starting from $\rho_0$ for the edge weight function $\eta_\eps$. Assume that $\rho^\eps_t$ is absolutely continuous with respect to $\mu$ for all $t$. In the following we drop the $t$-dependence of $\rho^\eps$ for brevity. From \eqref{eq:nlnl-interaction-eq}, by adding and subtracting $\rho^\eps(x) \int_\Rd (\dgrad K*\rho^\eps(x,y))_{+} \eta_\eps(x,y) \di y$, it follows that
    \begin{equation*}
        \partial_t \rho^\eps(x) = - \rho^\eps(x) \int_{\Rd} \dgrad K*\rho^\eps(x,y) \eta_\varepsilon(x,y) \di y - \int_{\Rd} \dgrad \rho^\eps(x,y) (\dgrad K*\rho^\eps(x,y))_+ \eta_\varepsilon(x,y) \di y.
    \end{equation*}
    Then, for almost all $x \in \Rd$ we have 
    \begin{align*} 
        \int_{\Rd} \dgrad K*\rho^\eps(x,y) \eta_\varepsilon(x,y) \di y &= \frac{2(2+d)}{\varepsilon^2} \int_\Rd (K*\rho^\eps(y) - K*\rho^\eps(x)) \frac{\chi_{B_\varepsilon(x)}(y)}{|B_\varepsilon|} \di y\\
        & = \frac{2(2+d)}{\varepsilon^2} \left( \frac{1}{|B_\varepsilon|} \int_{B_\varepsilon(x)} K*\rho^\eps(y) \di y - K*\rho^\eps(x) \right).
    \end{align*}
A standard calculation, using a second-order Taylor expansion, shows that the right-hand side approximates $\Delta K*\rho^\eps(x)$ when $\eps$ is small, provided that derivatives of $\rho^\eps$ remain uniformly bounded.
 
Similarly, by Taylor expanding $\dgrad \rho^\varepsilon$ and $\dgrad K \ast \rho^\varepsilon$ to first order and changing variable over the unit sphere while carefully tracking the positive part, one gets
\begin{equation}
  \int_\Rd \dgrad \rho^\eps(x,y) (\dgrad K*\rho^\eps(x,y))_+ \eta_\varepsilon(x,y) \di y \approx \nabla \rho^\eps(x) \cdot \nabla K*\rho^\eps(x) \quad \text{for small }  \eps.  \label{e:expanse:nabla}
\end{equation}
Combining the expressions above yields
\begin{equation*}
  \partial_t \rho^\eps(x) \approx -\rho^\eps(x) \Delta K*\rho^\eps(x) - \nabla \rho^\eps(x) \cdot \nabla K*\rho^\eps(x) = -\nabla \cdot (\rho^\eps \nabla K*\rho^\eps)(x).
\end{equation*}
This suggests that if $\rho^\eps$ converge as $\eps \to 0^+$, then the limiting $\rho$ is a solution of the standard nonlocal interaction equation \eqref{eq:NLIE}. A possible way to attack the local limit within the variational framework is via a stability statement similar to that of Theorem~\ref{thm:graph:stability}, but now with respect to the family $(\eta_\eps)_{\eps>0}$ in the limit $\eps\to 0^+$.
The following remark indicates that this will require further regularity assumptions on the interaction kernel $K$.
\begin{rem}
We present an example that indicates that, in certain situations, solutions of \eqref{eq:nlnl-interaction-eq} cannot be expected to converge to solutions of \eqref{eq:NLIE} as the interaction kernel $\eta_\eps$ becomes more concentrated. Namely, consider $d=1$, $\Omega = (-2,2)$ and $\mu=\operatorname{Leb}(\Omega)$. Let $K(x,y) = 1-e^{-|x-y|}$ for all $x,y\in\Omega$ and $\eta$ be a smooth, even function, positive on $(-0.2,0.2)$ and zero otherwise. Consider $\rho_0 = \frac12 (\delta_{-1} + \delta_1)$. It is straightforward to verify that $\rho_t = \rho_0$ for all $t\in[0,T]$ yields a weak solution of 
\eqref{eq:nlnl-interaction-eq} for all $\eps>0$. In particular, note that the corresponding velocity field satisfies $v(-1,y) = -(K*\rho_0(y) - K*\rho_0(-1)) \leq 0$ for all $y \in (-1.2,-0.8)$, and thus the flux from $x=-1$ remains zero, and analogously from $x=1$. Therefore, one cannot expect the weak solutions for the interaction potential $K$ to converge to weak solutions of \eqref{eq:NLIE} as $\eps \to 0^+$. 
We believe that, for these particular kernel $K$ and edge weights $\eta$, the problem persists for strong solutions for initial data close to $\rho_0$, only that explicit solutions are not available.
\end{rem}


\begin{appendix}
\section{Minkowski norm of the underlying Finsler structure}\label{app:minkowski}
In this appendix we show that, given $\rho \in \mP_2(\Rd)$ and $\bj \in T_\rho\mP_2(\Rd)$, the inner product $g_{\rho,\bj}$ from Section \ref{subsec:HeuristicFinsler} derives from a so-called Minkowski norm, as it should be in the theory of Finsler geometry; see \cite{Chern,Shen,Dahl}. 

Let us fix $\rho \in \mP_2(\Rd)$ which is absolutely continuous with respect to $\mu$, in accordance with Section~\ref{subsec:HeuristicFinsler}. For $\bj\in T_\rho\mP_2(\Rd)$, we denote $j$ its density with respect to $\mu \otimes \mu$. We show that the function $F_\rho\colon T_\rho\mP_2(\Rd) \to \R$ given by
\begin{equation*}
    F_\rho(\bj) = \sqrt{\frac12\iint_G j(x,y)^2 \eta(x,y) \left(\frac{\chi_{\{j>0\}}(x,y)}{\rho(x)} + \frac{\chi_{\{j<0\}}(x,y)}{\rho(y)}\right) \di\mu(x)\di\mu(y)} \quad \text{for all $\bj\in T_\rho\mP_2(\Rd)$},
\end{equation*}
is a Minkowski norm, that is, it is smooth away from $0$, positively $1$-homogeneous and, whenever $\bj$ in nonzero $\eta\, \mu\otimes\mu$-a.e., its second variation is a symmetric positive definite bilinear form. In fact, we now prove that, for all $\bj,\bj_1,\bj_2 \in T_\rho\mP_2(\Rd)$ such that $\bj$ is nonzero $\eta\,\mu\otimes\mu$-a.e.,
\begin{equation}\label{eq:minkowski-ip}
    \left. \frac12 \frac{\partial^2}{\partial t\partial s} \right|_{s=t=0} F_\rho^2(\bj + s\bj_1 + t\bj_2) = g_{\rho,\bj}(\bj_1,\bj_2).
\end{equation}
Indeed, let $\bj \in T_\rho\mP_2(\Rd)$, $s\in \R$ and $\bj_1 \in T_\rho\mP_2(\Rd)$ such that $\bj+s\bj_1 \in T_\rho\mP_2(\Rd)$. Then,
\begin{align*}
    &F_\rho^2(\bj+s\bj_1) - F_\rho^2(\bj)\\
    &= \frac12\iint_G (j(x,y)+sj_1(x,y))^2 \eta(x,y) \left(\frac{\chi_{\{j+sj_1>0\}}(x,y)}{\rho(x)} + \frac{\chi_{\{j+sj_1<0\}}(x,y)}{\rho(y)}\right) \di\mu(x)\di\mu(y)\\
    &\phantom{=}- \frac12\iint_G j(x,y)^2 \eta(x,y) \left(\frac{\chi_{\{j>0\}}(x,y)}{\rho(x)} + \frac{\chi_{\{j<0\}}(x,y)}{\rho(y)}\right) \di\mu(x)\di\mu(y)\\
    &= \frac12\iint_G j(x,y)^2 \eta(x,y) \left( \frac{\chi_{\{j+sj_1>0\}}(x,y)}{\rho(x)}-\frac{\chi_{\{j>0\}}(x,y)}{\rho(x)}\right) \di\mu(x)\di\mu(y)\\
    &\phantom{=}+ \frac12\iint_G j(x,y)^2\eta(x,y) \left(\frac{\chi_{\{j+sj_1<0\}}(x,y)}{\rho(y)} - \frac{\chi_{\{j<0\}}(x,y)}{\rho(y)}\right) \di\mu(x)\di\mu(y)\\
    &\phantom{=}+  \frac{s^2}{2}\iint_G j_1(x,y)^2 \eta(x,y) \left(\frac{\chi_{\{j+sj_1>0\}}(x,y)}{\rho(x)} + \frac{\chi_{\{j+sj_1<0\}}(x,y)}{\rho(y)}\right) \di\mu(x)\di\mu(y)\\
    &\phantom{=}+ s\iint_G j(x,y)j_1(x,y) \eta(x,y) \left(\frac{\chi_{\{j+sj_1>0\}}(x,y)}{\rho(x)} + \frac{\chi_{\{j+sj_1<0\}}(x,y)}{\rho(y)}\right) \di\mu(x)\di\mu(y)\\
    &=: I_1 + I_2 + I_3 + I_4.
\end{align*}
Note that
\begin{align*}
    I_1 &= \frac12\iint_{\{0<j \leq -sj_1\}} j(x,y)^2 \eta(x,y) \left( \frac{\chi_{\{j+sj_1>0\}}(x,y)}{\rho(x)} - \frac{\chi_{\{j>0\}}(x,y)}{\rho(x)} \right)\di\mu(x)\di\mu(y)\\
    &\phantom{=}+ \frac12\iint_{\{-sj_1<j \leq 0\}} j(x,y)^2 \eta(x,y) \left( \frac{\chi_{\{j+sj_1>0\}}(x,y)}{\rho(x)}-\frac{\chi_{\{j>0\}}(x,y)}{\rho(x)}\right) \di\mu(x)\di \mu(y)\\
    &= -\frac12\iint_{\{0<j \leq -sj_1\}} \frac{j(x,y)^2 \eta(x,y)}{\rho(x)} \di\mu(x)\di\mu(y) + \frac12\iint_{\{-sj_1<j < 0\}} \frac{j(x,y)^2 \eta(x,y)}{\rho(x)} \di\mu(x)\di\mu(y).
\end{align*}
Therefore,
\begin{equation*}
    |I_1| \leq s^2 \iint_G \frac{j_1(x,y)^2 \eta(x,y)}{\rho(x)} \di\mu(x)\di\mu(y).
\end{equation*}
Similarly, one gets
\begin{equation*}
    |I_2| \leq s^2 \iint_G \frac{j_1(x,y)^2 \eta(x,y)}{\rho(y)} \di\mu(x)\di\mu(y).
\end{equation*}
We also have 
\begin{equation*}
    |I_3| \leq \frac{s^2}{2}\iint_G \frac{j_1(x,y) \eta(x,y)}{\rho(x)} \di\mu(x)\di\mu(y) + \frac{s^2}{2}\iint_G \frac{j_1(x,y) \eta(x,y)}{\rho(y)} \di\mu(x)\di\mu(y)
\end{equation*}
Using Lebesgue's dominated convergence theorem, one gets moreover that
\begin{equation*}
    \lim_{s\to0}\frac{I_4}{s} = \iint_G j_1(x,y)j(x,y) \eta(x,y) \left(\frac{\chi_{\{j>0\}}(x,y)}{\rho(x)} + \frac{\chi_{\{j<0\}}(x,y)}{\rho(y)}\right) \di\mu(x)\di\mu(y).
\end{equation*}
We thus overall get
\begin{align*}
    \lim_{s\to0} \frac{F_\rho^2(\bj+s\bj_1) - F_\rho^2(\bj)}{s} &= \iint_G j_1(x,y)j(x,y) \eta(x,y) \left(\frac{\chi_{\{j>0\}}(x,y)}{\rho(x)} + \frac{\chi_{\{j<0\}}(x,y)}{\rho(y)}\right) \di\mu(x)\di\mu(y),
\end{align*}
which shows that
\begin{align}
    \left.\frac12 \frac{\di}{\di s}\right|_{s=0} F_\rho^2(\bj+s\bj_1) &= \frac12\iint_G j_1(x,y)j(x,y) \eta(x,y) \left(\frac{\chi_{\{j>0\}}(x,y)}{\rho(x)} + \frac{\chi_{\{j<0\}}(x,y)}{\rho(y)} \right) \di\mu(x)\di\mu(y) \nonumber\\
    &= g_{\rho,\bj}(\bj_1,\bj). \label{e:Minkowski:1st}
\end{align}
Note that this equality was used in Section \ref{subsec:HeuristicFinsler} to determine that $\grad^- \cE(\rho)$ is indeed the direction of steepest descent from $\rho$, i.e., to get \eqref{eq:min-H}. Computing now a further derivative in direction $\bj_2 \in T_\rho\mP_2(\Rd)$ and using similar boundedness arguments, we get
\begin{align*}
    &\left.\frac12 \frac{\partial^2}{\partial t\partial s}\right|_{s=t=0} F_\rho^2(\bj + s\bj_1 + t\bj_2)\\
    &\phantom{=}= \frac12\iint_G j_1(x,y)j_2(x,y) \eta(x,y) \left(\frac{\chi_{\{j>0\}}(x,y)}{\rho(x)} + \frac{\chi_{\{j<0\}}(x,y)}{\rho(y)}\right) \di\mu(x)\di\mu(y)\\
    &\phantom{==} + \frac12\iint_{\{j = 0\}} j_1(x,y)j_2(x,y) \eta(x,y) \left(\frac{\chi_{\{j_2>0\}}(x,y)}{\rho(x)} + \frac{\chi_{\{j_2<0\}}(x,y)}{\rho(y)}\right) \di\mu(x)\di\mu(y)\\
    &\phantom{=}= g_{\rho,\bj}(\bj_1,\bj_2) + \frac12\iint_{\{j = 0\}} j_1(x,y)j_2(x,y) \eta(x,y) \left(\frac{\chi_{\{j_2>0\}}(x,y)}{\rho(x)} + \frac{\chi_{\{j_2<0\}}(x,y)}{\rho(y)}\right) \di\mu(x)\di\mu(y).
\end{align*}
Note that the presence of the integral over the set $\{j=0\}$ comes from the fact that $j$ is not a multiplicative function of the integrand anymore (as it was the case for the first derivative), so that the set of points where $j = 0$ has to be considered. Assuming then that $\eta \, \mu \otimes \mu$-a.e.\ we have $\bj \neq 0$ we obtain that this integral over $\{j=0\}$ is equal to zero, which yields the claim.   
\end{appendix}
\subsection*{Acknowledgements}
The authors are grateful to Jos\'e Antonio Carrillo for several enlightening  discussions, to Mark Peletier for insightful remarks on an earlier version of the paper, and to Triphon Georgiou for valuable information.
The authors would like to thank the American Institute of Mathematics (AIM)
 for its support through the workshop \emph{Nonlocal differential equations in collective behavior}.
AE and AS gratefully acknowledge the support of the Hausdorff Research Institute for Mathematics (Bonn), through the Junior Trimester Program on \emph{Kinetic Theory}. AE acknowledges support by the EU-funded Erasmus Mundus programme ``MathMods - Mathematical
models in engineering: theory, methods, and applications'' at the University of L'Aquila.
AS is supported by the Deutsche Forschungsgemeinschaft (DFG, German Research Foundation) under Germany's Excellence Strategy EXC 2047
-- 390685813, the \emph{Hausdorff Center for Mathematics}, as well as the
Collaborative Research Center 1060 -- 211504053, \emph{The Mathematics
of Emergent Effects} at the Universität Bonn. DS is grateful to NSF for support via grants DMS 1516677 and DMS 1814991 and via Ki-Net (SF Research Network Grant RNMS 1107444). 
DS and FSP are grateful to the Center for Nonlinear Analysis of CMU for its support. 

\bibliography{references}
\bibliographystyle{abbrv}
\end{document}